%% file: openclosedBmodel.tex
\documentclass{amsart}
\usepackage[letterpaper, margin=1.4in]{geometry}
\usepackage{amssymb, latexsym, MnSymbol}
\usepackage{mathtools}
\usepackage{color, xcolor}
\usepackage{amscd, graphicx, psfrag, tikz}
\usepackage[mathscr]{euscript}
\usepackage[all]{xy}

\usepackage{enumitem}
\usepackage{stmaryrd}
\usepackage{url}
\usepackage{hyperref}
\hypersetup{
    colorlinks=true,
    linkcolor=black,
    citecolor=black,
    urlcolor=black
}
\usepackage{verbatim}

\usepackage{arydshln}

\newcommand{\Si}{ {\Sigma} }

\newcommand{\bC}{ {\mathbb{C}} }

\newcommand{\bF}{\mathbb{F}}

\newcommand{\bK}{\mathbb{K}}
\newcommand{\bL}{\mathbb{L}}

\newcommand{\bP}{\mathbb{P}}
\newcommand{\bQ}{\mathbb{Q}}
\newcommand{\bR}{\mathbb{R}}
\newcommand{\bS}{\mathbb{S}}
\newcommand{\bT}{\mathbb{T}}

\newcommand{\bZ}{\mathbb{Z}}

\newcommand{\cB}{\mathcal{B}}
\newcommand{\cC}{\mathcal{C}}
\newcommand{\cD}{\mathcal{D}}
\newcommand{\cE}{\mathcal{E}}
\newcommand{\cF}{\mathcal{F}}

\newcommand{\cH}{\mathcal{H}}
\newcommand{\cI}{\mathcal{I}}
\newcommand{\cL}{\mathcal{L}}
\newcommand{\cM}{\mathcal{M}}
\newcommand{\cO}{\mathcal{O}}
\newcommand{\cP}{\mathcal{P}}
\newcommand{\cQ}{\mathcal{Q}}
\newcommand{\cR}{\mathcal{R}}
\newcommand{\cS}{\mathcal{S}}

\newcommand{\cV}{\mathcal{V}}
\newcommand{\cW}{\mathcal{W}}
\newcommand{\cX}{\mathcal{X}}
\newcommand{\cY}{\mathcal{Y}}
\newcommand{\cZ}{\mathcal{Z}}

\newcommand{\age}{\mathrm{age}}

\newcommand{\CR}{ {\mathrm{CR}} }

\newcommand{\Hom}{\mathrm{Hom}}

\newcommand{\Res}{\mathrm{Res}}
\newcommand{\Spec}{\mathrm{Spec}}

\newcommand{\eff}{ {\mathrm{eff}} }

\newcommand{\pt}{\mathrm{pt}}

\newcommand{\Nef}{{\mathrm{Nef}}}
\newcommand{\NE}{{\mathrm{NE}}}

\newcommand{\rank}{\mathrm{rank}}

\newcommand{\Tot}{\mathrm{Tot}}
\newcommand{\Vol}{\mathrm{Vol}}

\renewcommand{\Box}{\mathrm{Box}}

\newcommand{\HV}{\mathrm{HV}}

\newcommand{\one}{\mathbf{1}}

\newcommand{\btau}{\boldsymbol{\tau}}

\newcommand{\bSi}{\mathbf{\Si}}
\newcommand{\bXi}{\mathbf{\Xi}}

\newcommand{\bw}{\mathbf{w}}

\newcommand{\fa}{\mathfrak{a}}

\newcommand{\fl}{\mathfrak{l}}
\newcommand{\fm}{\mathfrak{m}}

\newcommand{\fo}{\mathfrak{o}}
\newcommand{\fp}{\mathfrak{p}}
\newcommand{\fr}{\mathfrak{r}}
\newcommand{\fs}{\mathfrak{s}}

\newcommand{\fu}{\mathfrak{u}}

\newcommand{\fb}{\mathfrak{b}}

\newcommand{\su}{\mathsf{u}}

\newcommand{\sX}{\mathsf{X}} 

\newcommand{\tDelta}{\widetilde{\Delta}}

\newcommand{\tSi}{\widetilde{\Sigma}}
\newcommand{\trho}{\widetilde{\rho}}

\newcommand{\tbL}{\widetilde{\bL}}

\newcommand{\tD}{\widetilde{D}}

\newcommand{\tH}{\widetilde{H}}

\newcommand{\tM}{\widetilde{M}}
\newcommand{\tN}{\widetilde{N}}

\newcommand{\tT}{\widetilde{T}}
\newcommand{\tU}{\widetilde{U}}

\newcommand{\tX}{\widetilde{X}}

\newcommand{\tb}{\widetilde{b}}

\renewcommand{\th}{\widetilde{h}}
\newcommand{\tk}{\widetilde{k}}
\newcommand{\tl}{\widetilde{l}}
\newcommand{\tm}{\widetilde{m}}

\newcommand{\tq}{\widetilde{q}}
\newcommand{\tu}{\widetilde{u}}
\newcommand{\tv}{\widetilde{v}}

\newcommand{\tGamma}{\widetilde{\Gamma}}

\newcommand{\tgamma}{\widetilde{\gamma}}
\newcommand{\tsi}{\widetilde{\sigma}}
\newcommand{\tbeta}{\widetilde{\beta}}
\newcommand{\tpsi}{\widetilde{\psi}}

\newcommand{\ttau}{\widetilde{\tau}}
\newcommand{\talpha}{\widetilde{\alpha}}
\newcommand{\tbtau}{\widetilde{\btau}}
\newcommand{\tkappa}{\widetilde{\kappa}}
\newcommand{\tOmega}{\widetilde{\Omega}}
\newcommand{\tomega}{\widetilde{\omega}}

\newcommand{\tpi}{\widetilde{\pi}}

\newcommand{\tbw}{\widetilde{\bw}}

\newcommand{\tNef}{\widetilde{\Nef}}
\newcommand{\tNE}{\widetilde{\NE}}

\newcommand{\tcD}{\widetilde{\cD}}
\newcommand{\tcH}{\widetilde{\cH}}
\newcommand{\tcP}{\widetilde{\cP}}
\newcommand{\tcX}{\widetilde{\cX}}
\newcommand{\tbSi}{\widetilde{\bSi}}

\newcommand{\vzero}{\vec{0}}

\newcommand{\inner}[1]{\langle  #1 \rangle}
\newcommand{\floor}[1]{\lfloor  #1 \rfloor}
\newcommand{\ceil}[1]{\lceil  #1 \rceil}

\newtheorem{dummy}{dummy}[section]
\newtheorem{lemma}[dummy]{Lemma}
\newtheorem{theorem}[dummy]{Theorem}

\newtheorem{corollary}[dummy]{Corollary}
\newtheorem{proposition}[dummy]{Proposition}
\newtheorem{remark}[dummy]{Remark}
\newtheorem{definition}[dummy]{Definition}
\newtheorem{example}[dummy]{Example}
\newtheorem{notation}[dummy]{Notation}

\newtheorem{assumption}[dummy]{Assumption}

\begin{document}
\title{Hodge-theoretic Open/Closed Correspondence}

\author{Song Yu}
\address{Yau Mathematical Sciences Center, Tsinghua University, Beijing, 100084, China}
\email{song-yu@tsinghua.edu.cn}


\begin{abstract}
We continue the B-model development of the open/closed correspondence proposed by Mayr and Lerche-Mayr, complementing the A-model study in the preceding joint works with Liu and providing a Hodge-theoretic perspective. Given a corresponding pair of open geometry on a toric Calabi-Yau 3-orbifold $\mathcal{X}$ relative to a framed Aganagic-Vafa brane $\mathcal{L}$ and closed geometry on a toric Calabi-Yau 4-orbifold $\widetilde{\mathcal{X}}$, we consider the Hori-Vafa mirrors $\mathcal{X}^\vee$ and $\widetilde{\mathcal{X}}^\vee$, where the mirror of $\mathcal{L}$ can be given by a family of hypersurfaces $\mathcal{Y} \subset \mathcal{X}^\vee$. We show that the Picard-Fuchs system associated to $\widetilde{\mathcal{X}}$ extends that associated to $\mathcal{X}$ and characterize the full solution space in terms of the open string data. Furthermore, we construct a correspondence between integral 4-cycles in $\widetilde{\mathcal{X}}^\vee$ and relative 3-cycles in $(\mathcal{X}^\vee, \mathcal{Y})$ under which the periods of the former match the relative periods of the latter. On the dual side, we identify the variations of mixed Hodge structures on the middle-dimensional cohomology of $\widetilde{\mathcal{X}}^\vee$ with that on the middle-dimensional relative cohomology of $(\mathcal{X}^\vee, \mathcal{Y})$ up to a Tate twist. 
\end{abstract}

\maketitle

\setcounter{tocdepth}{1}
\tableofcontents

\input{intro}

\input{geometry}

\input{hyper}

\input{picardfuchs}

\input{cycles}

\input{hodge}


\input{ref}

\end{document}

%% file: intro.tex

\section{Introduction}\label{sect:Intro}

\subsection{Open/closed correspondence and mirror symmetry}\label{sect:Overview}
This work is a sequel to the joint works \cite{LY21,LY22} with Liu where we aim to build a mathematical theory for the \emph{open/closed correspondence} proposed by Mayr \cite{Mayr01} and Lerche-Mayr \cite{LM01}. The correspondence is a class of dualities between open strings on a Calabi-Yau 3-fold background and closed strings on a dual Calabi-Yau 4-fold. When the open geometry is a \emph{toric} Calabi-Yau 3-fold $\cX$ with boundary condition specified by a Lagrangian $\cL$ known as an \emph{Aganagic-Vafa A-brane} \cite{AV00,AKV02}, the dual closed geometry is a toric Calabi-Yau 4-fold $\tcX$.

The correspondence provides a multifaceted relationship between the two geometries that spans across both the A- and B-models of mirror symmetry, in genus zero and in the classical sense of e.g. \cite{CdGP91,Givental98,LLY97,CK99}. The present non-compact setup also fits into considerations of \emph{local} mirror symmetry \cite{HV00,HIV00,CKYZ99,Hosono06}. On the A-side, the correspondence is interpreted via genus-zero Gromov-Witten theory and equates the open (specifically disk) invariants of $(\cX, \cL)$ and the closed invariants of $\tcX$. This is the subject of \cite{LY21,LY22}, generalizing the case studied by Bousseau-Brini-van Garrel \cite{BBvG20,BBvG20b} where the geometries arise from \emph{Looijenga pairs}. The correspondence passes through the relative Gromov-Witten invariants of a certain relative geometry that bridges the open and closed sides, which provides examples of the \emph{log-local principle} of van Garrel-Graber-Ruddat \cite{vGGR19} in the non-compact setting. The Gromov-Witten correspondence leads to structural results on the open invariants \cite{YZ23} as well as a correspondence of open/closed BPS invariants which are shown to have integrality properties \cite{BBvG20,Yu24}.

On the B-side, in the toric Calabi-Yau setting, consider the Hori-Vafa \cite{HV00} mirror families
$$
  (\cX^\vee_q, \Omega_q), \qquad (\tcX^\vee_{\tq}, \tOmega_{\tq}).
$$
In the closed sector, periods of the holomorphic volume forms $\Omega_q, \tOmega_{\tq}$ along middle-dimensional cycles recover the genus-zero closed Gromov-Witten theory of $\cX, \tcX$ respectively under mirror symmetry. From a Hodge-theoretic point of view, the periods are solutions to a \emph{Picard-Fuchs} system of differential equations which govern the variations of the volume form in cohomology and are mirror to the quantum differential equations. Solutions to the Picard-Fuchs system can be obtained explicitly as components of a hypergeometric function, called the \emph{$I$-function}, which is identified under mirror symmetry with the \emph{$J$-function} coming from genus-zero closed Gromov-Witten theory.

In the open sector, taking the Lagrangian $\cL \subset \cX$ into consideration, Aganagic-Vafa \cite{AV00} and Aganagic-Klemm-Vafa \cite{AKV02} proposed a version of open mirror symmetry which recovers the disk invariants of $(\cX, \cL)$ from \emph{relative} periods of $\Omega_q$ on $\cX^\vee_q$. Specifically, one integrates over relative 3-cycles with boundary in the so-called \emph{B-branes} which are certain 2-cycles in $\cX^\vee_q$ mirror to the A-brane $\cL$. As studied in \cite{GZ02,FL13,FLT12}, such relative periods produce a hypergeometric function, which we call the \emph{B-model disk function}, that is identified under open mirror symmetry with the generating function of disk invariants of $(\cX, \cL)$, which we call the \emph{A-model disk function}. As suggested in \cite{LMW02a,LMW02b}, the boundary condition may be changed from the B-branes to a family of divisors $\cY_{q,x}$ in $\cX^\vee_q$. Moreover, the relative periods can be characterized via the variation of $\Omega_q$ in the \emph{relative cohomology}. For the quintic 3-fold and other compact Calabi-Yau hypersurfaces in toric varieties, relative periods and the variation formalism have been studied in \cite{AHMM09,JS09,MW09,LLY12}; see also \cite{Walcher07,Walcher08,PSW08}.

A key observation by Lerche and Mayr \cite{Mayr01,LM01} for the B-model open/closed correspondence is that the relative periods on the mirror side of the open geometry $(\cX, \cL)$ are solutions to an \emph{extension} of the Picard-Fuchs system of $\cX$, and the extended system \emph{coincides} with the usual Picard-Fuchs system of the closed geometry $\tcX$. The additional solutions to the extended Picard-Fuchs system encode the open mirror map and the B-model disk function. This has been a central guideline in the study of open mirror symmetry and Gromov-Witten theory of $(\cX, \cL)$ \cite{FL13,FLT12,KZ15,GRZ22}. The extended system should also govern the variation of $\Omega_q$ in the relative cohomology of $(\cX^\vee_q, \cY_{q,x})$ considered as a mixed Hodge structure, which should thus be identified with the variation of $\tOmega_{\tq}$ in the usual cohomology of $\tcX^\vee_{\tq}$ \cite{LMW02a,LMW02b}. Dually, there should be an explicit correspondence between relative 3-cycles on $\cX^\vee_q$ and absolute 4-cycles on $\tcX^\vee_{\tq}$ under which the periods are identified. 

In this work, we fulfill the above proposals of the B-model correspondence for a general toric Calabi-Yau 3-orbifold $\cX$, complementing the A-model development in \cite{LY21,LY22} and emphasizing a Hodge-theoretic perspective. This expands and contextualizes a first step made in \cite{LY22} where we showed that the B-model disk function of $(\cX, \cL)$ can be recovered from the $I$-function of $\tcX$. We now give a more detailed account of our main results.

\subsection{Extended Picard-Fuchs equations and solutions from open strings}
Consider the open geometry on a semi-projective toric Calabi-Yau 3-orbifold $\cX$, presented as a GIT quotient stack
$$
  \cX = \left[\bC^R \sslash (\bC^*)^{R-3} \right],
$$
and an Aganagic-Vafa outer brane $\cL \subset \cX$ equipped with a parameter $f \in \bQ$ called the framing. The corresponding closed geometry is a semi-projective toric Calabi-Yau 4-orbifold, presented as a GIT quotient stack
$$
  \tcX = \left[\bC^{R+2} \sslash (\bC^*)^{R-2} \right].
$$
The geometric relation between $(\cX, \cL, f)$ and $\tcX$ is manifested in that the linear action of $(\bC^*)^{R-2}$ on $\bC^{R+2}$ is determined by the linear action of $(\bC^*)^{R-3}$ on $\bC^R$ and the input of $(\cL, f)$. In the simple case where $\cX$ is smooth and $f \in \bZ$, under suitable bases, the action of $(\bC^*)^{R-2}$ is described by an $(R-2)$-by-$(R+2)$ \emph{charge matrix} of form
$$
  \begin{bmatrix}
    & \begin{array}{:cccc:}
      \hdashline
      \star & \phantom{f} & \phantom{-f-1} & \phantom{0 - \cdots -} \star\\
      \vdots & \phantom{f} & \phantom{-f-1} & \phantom{0 - \cdots -}  \vdots\\
      \star & \phantom{f} & \phantom{-f-1} & \phantom{0 - \cdots -} \star\\
      \hdashline
    \end{array} 
    & \begin{array}{cc}
      0 & \phantom{1}0\\
      \vdots & \phantom{1}\vdots\\
      0 & \phantom{1}0
    \end{array}\\
    & \begin{array}{cccccc}
      1 & f & -f-1 & 0 \phantom{\displaystyle \frac{1}{1}} \cdots \phantom{\displaystyle \frac{1}{1}} 0 
    \end{array} & 
    \begin{array}{cc}
      1 & -1
    \end{array}
  \end{bmatrix}
$$
where the first $(R-3)$-by-$R$ block is the charge matrix of the action of $(\bC^*)^{R-3}$ and the indexing of the first three columns is determined by the position of $\cL$ in $\cX$.

\begin{example}\rm{
For $\cX = \bC^3$ and $f=1$, we have $\tcX = \Tot(\cO_{\bP^2}(-2) \oplus \cO_{\bP^2}(-1))$ (cf. \cite[Section 2.6.1]{LY22}). The charge matrix above consists only of the additional charge vector
$$
  \begin{bmatrix}
    1 & 1 & -2 & 1 & -1
  \end{bmatrix}.
$$
We will use this example as the running example in this paper; see Examples \ref{ex:C3}, \ref{ex:C3PicardFuchs}, \ref{ex:C3Cycles}, \ref{ex:C3MHS}. The main example studied in \cite{Mayr01,LM01} is $\cX = K_{\bP^2}$ (cf. \cite[Section 2.6.2]{LY22}). For $f \in \bZ$, the charge matrix above is
$$
  \begin{bmatrix}
    -3 & 1 & 1 & 1 & 0 & 0\\
    1 & f & -f-1 & 0 & 1 & -1
  \end{bmatrix}.
$$
}
\end{example}

Let
$$
  \cP, \qquad \tcP
$$
denote the \emph{Picard-Fuchs system} of differential equations of $\cX$, $\tcX$ respectively. They are defined by the toric data and can be written down explicitly using the charge matrices above (see Section \ref{sect:PicardFuchsDef}). The relation between the charge matrices implies that $\tcP$ is an extension of $\cP$ in the sense that solutions to $\cP$ are also solutions to $\tcP$; this is verified in detail in Proposition \ref{prop:SolnSubspace}.

Solutions to the Picard-Fuchs systems can be given by the components of the \emph{$I$-functions}
$$
  I_{\cX}(q, z) = 1 + z^{-1}\btau_2(q) + o(z^{-1}), \qquad I_{\tcX}(\tq, z) = 1 + z^{-1}\tbtau_2(\tq) + o(z^{-1})
$$
which are hypergeometric functions valued in Chen-Ruan orbifold cohomology \cite{Givental98,Iritani09} (see Section \ref{sect:IFunction} for the precise definition). They depend on the complex moduli parameters
$$
  q = (q_1, \dots, q_{R-3}), \qquad \tq = (\tq_1, \dots, \tq_{R-2})
$$
and an additional variable $z$. Components in the $z^{-1}$-coefficient $\btau_2 = \btau_2(q)$, $\tbtau_2 = \tbtau_2(\tq)$ give the \emph{closed mirror maps} under which the $I$-functions are identified with the $J$-functions from genus-zero Gromov-Witten theory \cite{Givental98,CCK15,CCIT15}.  

Now consider the framed Aganagic-Vafa outer brane $(\cL, f)$ in the open geometry. We have $\pi_1(\cL) \cong \bZ \times \mu_{\fm}$ where the latter component is the cyclic group of some order $\fm \in \bZ_{>0}$. We write $f \in \bQ$ as $f = \frac{\fb}{\fa}$ for coprime integers $\fa \in \bZ_{>0}$ and $\fb \in \bZ$. Under open mirror symmetry \cite{AV00,AKV02,GZ02,FL13,FLT12}, the \emph{B-model disk function}
$$
  W^{\cX, (\cL, f)}_{\tk}(q, x)
$$
of $(\cX, \cL, f)$ is a hypergeometric function that recovers the generating function of disk Gromov-Witten invariants (see Section \ref{sect:DiskFunction} for the precise definition). It depends on the closed moduli parameters $q$ and an additional open moduli parameter $x$. It is also indexed by an element $\tk \in \mu_{\fa\fm}$ which classifies the winding and monodromy of the disk invariants. Open mirror symmetry passes through the closed mirror map $\btau_2(q)$ above and an additional \emph{open mirror map}
$$
  \log \sX = \log \sX(q, x).
$$

Lerche and Mayr \cite{Mayr01,LM01} observed that the extended Picard-Fuchs system $\tcP$ of the closed geometry $\tcX$ governs the data from the open geometry $(\cX, \cL, f)$, which is by now a folklore result. Specifically, the open mirror map is a solution to $\tcP$ and that the B-model disk function satisfies the recursion relations provided by $\tcP$. Building on these observations and the first steps made in \cite{LY22}, we provide a description of the full solution space to $\tcP$. The description involves the disk functions of not only $(\cL, f)$ but also nearby framed Aganagic-Vafa outer branes $(\cL^1, f_1), \dots, (\cL^S, f_S)$ from which the same closed geometry $\tcX$ arises (cf. Remark \ref{rem:MultiPhases}). From the perspective of phase shifts or wall-crossings (cf. \cite[Section 1.6.3]{LY21}), the closed phase $\tcX$ corresponds to all these different open phases.


\begin{proposition}[See Propositions \ref{prop:SolnSubspace}, \ref{prop:OpenMirrorMapSoln}, \ref{prop:DiskFnSoln}]\label{prop:IntroExtendedPF}
The solution space of $\tcP$ is spanned by that of $\cP$ and the following additional solutions (under $\tq_a = q_a$, $a = 1, \dots, R-3$, and $\tq_{R-2} = x$):
\begin{itemize}
  \item the open mirror map $\log \sX(q,x)$, appearing in the $z^{-1}$-coefficient of $I_{\tcX}$ and identified with the additional closed mirror map of $\tcX$;

  \item components in the $z^{-2}$-coefficient of $I_{\tcX}$ of which the $\tq_{R-2}$-dependence of the power series part is given by the B-model disk functions of $(\cL^1, f_1), \dots, (\cL^S, f_S)$.
\end{itemize}
\end{proposition}

The identification of the open mirror map $\log \sX$ and the additional closed mirror map of $\tcX$ is given in \cite[Proposition 6.14]{LY22}. The study of the second type of additional solutions is based on the open/closed correspondence of hypergeometric functions established in \cite[Theorem 1.4]{LY22}. In more detail, the B-model disk function of $(\cX, \cL, f)$ can be recovered from the (equivariant) $I$-function of $\tcX$ as
$$
    W^{\cX, (\cL, f)}_{\tk}(q, x) = [z^{-2}] \left(I_{\tcX}^{\tT'}(\tq, z), \tgamma_{\tk} \right)_{\tcX}^{\tT'} \bigg|_{T_f}
$$
for any $\tk \in \mu_{\fa\fm}$. Here, $\tT'$ denotes the Calabi-Yau 3-torus of $\tcX$ and $\big|_{T_f}$ denotes the weight restriction to a 1-subtorus $T_f$ determined by $f$. Moreover, $\left(-, \tgamma_{\tk} \right)_{\tcX}^{\tT'}$ is the $\tT'$-equivariant Poincar\'e pairing on cohomology with a specific class $\tgamma_{\tk}$ and $[z^{-2}]$ stands for taking the $z^{-2}$-coefficient. This result is applied to the $S$ different Aganagic-Vafa branes.


We note that an additional solution of the second type and in the untwisted sector in general involves disk functions of \emph{two} neighboring branes. Moreover, even when $\cX$ is smooth and $f$ is an integer, in general $\tcX$ is an orbifold and the corresponding framings $f_s$ of the relevant branes above are rational numbers, and it is thus necessary to treat the generality of orbifolds with rational framings, as is done in \cite{LY22}.

\subsection{Hori-Vafa mirrors, integral cycles, and periods}
To contextualize the open/closed correspondence of Picard-Fuchs systems and hypergeometric functions, we consider the B-model theory on the Hori-Vafa mirrors of $\cX$, $\tcX$, which are the families of hypersurfaces
\begin{align*}
  & \cX^\vee_q := \{(u,v, X, Y) \in \bC^2 \times (\bC^*)^2: uv = H(X, Y, q) \},\\
  & \tcX^\vee_{\tq} := \{(u,v, X, Y, Z) \in \bC^2 \times (\bC^*)^3: uv = \tH(X, Y, Z, \tq) \}
\end{align*}
parameterized by $q$, $\tq$ respectively. Here, $H(X,Y,q)$ (resp. $\tH(X, Y, Z, \tq)$) is a Laurent polynomial in $X, Y$ (resp. $X,Y,Z$) whose Newton polytope is the 2- (resp. 3-) dimensional polytope $\Delta$ (resp. $\tDelta$) that is the convex hull of the generators of 1-cones of the fan of $\cX$ (resp. $\tcX$). It holds that the powers of $Z$ in $\tH(X, Y, Z, \tq)$ are non-negative, and under $\tq_a = q_a$ for $a = 1, \dots, R-3$ we have
$$
  \tH(X, Y, Z, \tq) \big|_{Z=0} = H(X, Y, q).
$$
Thus the projection $Z: \tcX^\vee_{\tq} \to \bC^*$ defines a family of deformations of $\cX^\vee_q$ which can be viewed as the central fiber over $Z=0$.

It is well-known that period integrals of the holomorphic volume form $\Omega_q$ (resp. $\tOmega_{\tq}$) over integral 3-cycles on the Hori-Vafa mirror $\cX^\vee_q$ (resp. 4-cycles on $\tcX^\vee_{\tq}$) generate the solution space of the Picard-Fuchs system $\cP$ (resp. $\tcP$) \cite{Hosono06,KM10,CLT13}. As we view $\tcP$ as the extended system of $\cP$ by the open string data, we show that the additional solutions to $\tcP$ (described in Proposition \ref{prop:IntroExtendedPF}) can be recovered from periods over \emph{relative} 3-cycles on $\cX^\vee_q$. To be more specific, we consider relative cycles with boundary in the family of divisors 
$$
  \cY_{q,x} := \cX^\vee_q \cap \{X^{\fa}Y^{\fb} = -x^{\fa}\} = \{(u,v, Y) \in \bC^2 \times \bC^*: uv = H_0(Y, q, x) := H(X, Y, q) \big|_{X^{\fa}Y^{\fb} = -x^{\fa}} \}
$$
in $\cX^\vee_q$ parameterized by the open moduli parameter $x$. On its own, $\cY_{q,x}$ can be viewed as the Hori-Vafa mirror of a semi-projective toric Calabi-Yau surface (2-orbifold) $\cX_0$ that connects the 3- and 4-dimensional geometries. The Newton polytope of the restricted Laurent polynomial $H_0(Y, q, x)$ is a closed interval $\Delta_0$ that is the convex hull of the generators of 1-cones of the fan of $\cX_0$.

We establish the following correspondence of integral cycles and periods for the families $(\cX^\vee_q, \cY_{q,x})$ and $\tcX^\vee_{\tq}$ over a suitable domain $(q,x) = \tq \in \tU_\epsilon$ where the Laurent polynomials involved the definitions above satisfy certain regularity properties (see Section \ref{sect:Laurent}).

\begin{theorem}[See Theorem \ref{thm:CycleCorr}]\label{thm:IntroCycleCorr}
For $(q,x) = \tq \in \tU_\epsilon$, there is an injective map
$$
  \iota: H_3(\cX^\vee_q, \cY_{q,x}; \bZ) \to H_4(\tcX^\vee_{\tq}; \bZ)
$$
such that for any $\Gamma \in H_3(\cX^\vee_q, \cY_{q,x}; \bZ)$, we have
$$
  \int_{\Gamma}  \Omega_q = \frac{1}{2\pi\sqrt{-1}} \int_{\iota(\Gamma)}  \tOmega_{\tq}.
$$
Moreover, $\iota$ is an isomorphism over $\bQ$.
\end{theorem}

The construction of the map $\iota$ in Theorem \ref{thm:IntroCycleCorr} is based on the structure of $\tcX^\vee_{\tq}$ as a family of deformations of $\cX^\vee_q$ over $Z \in \bC^*$, with $\cX^\vee_q$ the central fiber over $Z=0$. On a high level, given a relative 3-cycle $\Gamma \in H_3(\cX^\vee_q, \cY_{q,x}; \bZ)$ included in the central fiber, we deform it over a small circle $\{|Z| = \epsilon'\}$ and make necessary adjustments to obtain the desired absolute 4-cycle $\iota(\Gamma)$ without boundary. The construction also makes use of the structures of the Hori-Vafa mirrors $\cX^\vee_q$, $\tcX^\vee_{\tq}$, $\cY_{q,x}$ as conic fibrations over the algebraic tori $(\bC^*)^2$, $(\bC^*)^3$, $\bC^*$ respectively whose discriminant loci are the affine hypersurfaces
$$
  C_q := \{H(X, Y, q) = 0\}, \qquad S_{\tq} := \{\tH(X, Y, Z, \tq) = 0\}, \qquad P_{q,x} := \{H_0(Y, q, x) = 0 \}
$$
defined by the Laurent polynomials. Periods of the holomorphic volume forms on the Hori-Vafa mirrors are known to correspond to periods of the standard holomorphic volume forms of the algebraic tori over relative cycles with boundary in the hypersurfaces \cite{DK11,CLT13}.

\begin{remark}\rm{
We conjecture that the map $\iota$ in Theorem \ref{thm:IntroCycleCorr} is surjective and is thus an isomorphism over $\bZ$.
}\end{remark}


As noted before, the original proposal of \cite{Mayr01,LM01} considered relative cycles on $\cX^\vee_q$ with boundary on certain 2-cycles parameterized by $x$, known as Aganagic-Vafa B-branes and considered to be mirror to the A-brane $\cL$ \cite{AV00,AKV02}. It was suggested by \cite{LMW02a,LMW02b} that the boundary condition can be changed from the B-branes to the divisors. We provide a detailed account of the relationship between the two types of boundary conditions in Section \ref{sect:AVBBranes}.

\subsection{Variations of mixed Hodge structures}
On the dual side of Theorem \ref{thm:IntroCycleCorr} on integral cycles, we study the holomorphic volume forms on the Hori-Vafa mirrors in the context of (variations of) mixed Hodge structures ((V)MHS) on cohomology. In particular, following the ideas in \cite{LMW02a,LMW02b,LLY12}, the relative periods of $\Omega_q$ are governed by its variations in the relative cohomology $H^3(\cX^\vee_q, \cY_{q,x}; \bC)$. We establish the following open/closed correspondence of VMHS.

\begin{theorem}[See Theorem \ref{thm:MHSCorr}]\label{thm:IntroMHSCorr}
Over $(q,x) = \tq \in \tU_\epsilon$, there is an isomorphism of VMHS
$$
  H^3(\cX^\vee_q, \cY_{q,x}; \bC) \cong H^4(\tcX^\vee_{\tq}; \bC) \otimes T(1)
$$
that identifies $[\Omega_q]$ with $[\tOmega_{\tq}]$.
\end{theorem}

Here $\otimes T(1)$ denotes the Tate twist by 1. Similar to Theorem \ref{thm:IntroCycleCorr}, we use the conic fibration structure of the Hori-Vafa mirrors over the algebraic tori. The VMHS on the cohomology of the affine hypersurfaces in tori are explicitly computed by Batyrev \cite{Batyrev93} via the combinatorial data of the Laurent polynomials and their Newton polytopes, in a way similar to the construction of \emph{Jacobian rings}. Based on this result, Stienstra \cite{Stienstra98} and Konishi-Minabe \cite{KM10} obtained a similar description for the VMHS on the cohomology of the tori relative to the hypersurfaces. As for the VMHS on the cohomology of Hori-Vafa mirrors, Konishi-Minabe \cite{KM10} gave an analogous combinatorial description for 2-dimensional \emph{reflexive} polytopes, which has been generalized in upcoming joint work \cite{AY25} with Aleshkin for general polytopes of any dimensions. The relation between the two sides is seen through the isomorphism of VMHS
$$
  H^2((\bC^*)^2, C_q; \bC) \cong H^3(\cX^\vee_q; \bC) \otimes T(1)
$$
and similarly for the 4- and 2-dimensional geometries.

We construct the isomorphism in Theorem \ref{thm:IntroMHSCorr} explicitly using the combinatorial presentations. A key technical step is the combinatorial construction of the following extensions of VMHS
$$
  \xymatrix{
        0 \ar[r] & H^1(\bC^*, P_{q,x}; \bC) \ar[r] & H^3((\bC^*)^3, S_{\tq}; \bC) \otimes T(1) \ar[r] & H^2((\bC^*)^2, C_q; \bC) \ar[r] & 0,
    }
$$
$$    
  \xymatrix{
        0 \ar[r] & H^2(\cY_{q,x}; \bC) \ar[r] & H^4(\tcX^\vee_{\tq}; \bC) \otimes T(1) \ar[r] & H^3(\cX^\vee_q; \bC) \ar[r] & 0.
    }
$$
See Corollary \ref{cor:MHSExtension}.

\subsection{Related and future works}
We conclude the introduction by a brief discussion of related and future works.

\subsubsection{The case of the quintic 3-fold}
The open/closed correspondence for the quintic 3-fold relative to the real quintic has been studied by Aleshkin-Liu \cite{AL23}, following \cite{Walcher07,Walcher08,PSW08}. In this case, they showed that the extended Picard-Fuchs system satisfied by the B-model disk function is not the Picard-Fuchs system of any Calabi-Yau variety, but rather originates from a particular \emph{gauged linear sigma model (GLSM)} that serves as the corresponding closed geometry. The B-model open/closed correspondence in this case is also shown to be compatible with the Landau-Ginzburg/Calabi-Yau correspondence \cite{CR10}.

\subsubsection{Hodge-theoretic open mirror symmetry} Hodge-theoretic mirror symmetry for toric stacks \cite{KKP08, Iritani09,CCIT20} identifies the A-model quantum $D$-module with the B-model Gelfand-Kapranov-Zelevinsky-type $D$-module that packages the Picard-Fuchs equations and the VMHS. The Hodge-theoretic aspects of the B-model open/closed correspondence studied in this paper may be developed for the A-model as well. As a starting point, in joint work \cite{YZ23} with Zong, we used the Witten-Dijkgraaf-Verlinde-Verlinde (WDVV) equations of $\tcX$ to deduce a system of open WDVV equations for $(\cX, \cL)$, which may be used to construct a $D$-module for the open A-model. Combining the ingredients above would lead to a Hodge-theoretic version of the open mirror symmetry for $(\cX, \cL)$.

\subsection{Organization of the paper}
In Section \ref{sect:Geometry}, we review the constructions of the 3-dimensional open geometry $(\cX, \cL)$ and the 4-dimensional closed geometry $\tcX$ and introduce the 2-dimensional geometry $\cX_0$ that relates the two sides. The subsequent sections concern the following levels of the B-model open/closed correspondence respectively:
\begin{itemize}
  \item Section \ref{sect:Hyper}: hypergeometric functions (as obtained in \cite{LY22});
  
  \item Section \ref{sect:PicardFuchs}: extended Picard-Fuchs system and additional solutions (Proposition \ref{prop:IntroExtendedPF});
  
  \item Section \ref{sect:Cycles}: integral cycles and periods (Theorem \ref{thm:IntroCycleCorr});
  
  \item Section \ref{sect:MHS}: VMHS (Theorem \ref{thm:IntroMHSCorr}).
\end{itemize}


\subsection{Acknowledgments}
I would like to thank Chiu-Chu Melissa Liu for the encouragement of this work and for many stimulating discussions and ideas. I would like to thank Konstantin Aleshkin for valuable discussions on VMHS. A preliminary version of the results in this work appeared in the author's doctoral dissertation \cite{Yu23} and I would like to thank Mohammed Abouzaid, Mina Aganagic, and Andrei Okounkov for serving on the defense committee and offering illuminating feedback.

%% file: geometry.tex

\section{Toric Calabi-Yau geometries}\label{sect:Geometry}
In this section, we recall the 3-dimensional open geometry and the 4-dimensional closed geometry on the two sides of the open/closed correspondence, following \cite{LY22}. We further introduce a 2-dimensional geometry that bridges the two sides. We work over $\bC$.

\subsection{Notations on toric Calabi-Yau orbifolds}\label{sect:Notations}
Let $\cZ$ be an $r$-dimensional \emph{toric orbifold}, or smooth toric Deligne-Mumford stack \cite{BCS05,FMN10} with trivial generic stabilizer. The geometry of $\cZ$ is combinatorially specified by a \emph{stacky fan} \cite{BCS05} $\bXi' = (\bZ^r, \Xi, \alpha')$, where $\Xi$ is a simplicial fan in $\bR^r = \bZ^r \otimes \bR$ and $\alpha':\bZ^{R'} \to \bZ^r$ is the homomorphism determined by the primitive generators $b_1, \dots, b_{R'} \in \bZ^r$ of the 1-dimensional cones (or rays) in $\Xi$. The coarse moduli space of $\cZ$ is the simplicial toric variety $Z$ determined by $\Xi$. The Deligne-Mumford torus acting on $\cZ$ (and $Z$) is given by $\bT := \bZ^r \otimes \bC^* = (\bC^*)^r$. In this paper, we will work with the following two conditions:
\begin{itemize}
    \item \emph{Calabi-Yau} condition - the canonical bundle $K_Z$ of $Z$ is trivial. In terms of the fan, this means that there exists $u \in \Hom(\bZ^r, \bZ)$ such that $\inner{u,b_i} =1$ for $i = 1, \dots, R'$, where $\inner{-,-}$ is the natural pairing. 
    
    \item \emph{Semi-projectivity} - $Z$  has at least one $\bT$-fixed point and is projective over its affinization $\Spec(H^0(Z, \cO_Z))$. In terms of the fan, this means that the support of $\Xi$ is $r$-dimensional and convex. 
\end{itemize}

The above two conditions imply that the support of the fan $\Xi$ is the cone over a convex lattice polytope $\Delta$ in the hyperplane of $\bR^r$ defined by $\inner{u,-} = 1$, and $\Xi$ is specified by a regular triangulation of $\Delta$. In this case, if we write
$$
    \Delta \cap \bZ^r = \{b_1, \dots, b_{R'}, \dots, b_R\}
$$
and let $\alpha:\bZ^{R} \to \bZ^r$ be the surjective homomorphism determined by $b_1, \dots, b_R$, then $\bXi = (\bZ^r, \Xi, \alpha)$ is an \emph{extended stacky fan} \cite{Jiang08} of $\cZ$. In this paper, we will always work with this extended stacky fan where all lattice points in $\Delta$ are selected.

The element $u \in \Hom(\bZ^r, \bZ)$ given by the Calabi-Yau condition above defines a character $\su: \bT \to \bC^*$ whose kernel $\bT' \cong (\bC^*)^{r-1}$ is referred to as the \emph{Calabi-Yau subtorus} of $\bT$.

\subsubsection{Torus orbits and stabilizers}
For $d = 0, \dots, r$, let $\Xi(d)$ denote the set of $d$-dimensional cones in $\Xi$. For each $\sigma \in \Xi(d)$, we set
$$
    I_\sigma' := \{ i \in \{1, \dots, R'\} : b_i \in \sigma \}, \qquad I_\sigma := \{1, \dots, R\} \setminus I_\sigma'.
$$
Let $\cV(\sigma) \subseteq \cZ$ denote the codimension-$d$ $\bT$-invariant closed substack of $\cZ$ corresponding to $\sigma$ and
$$
    \iota_\sigma: \cV(\sigma) \to \cZ
$$
denote the inclusion\footnote{The letter ``$\iota$'' will be used to denote various types of inclusion maps in this paper.}. The action of $\bT'$ on $\cZ$ (and $Z$) has the same codimension-$d$ orbits as that of $\bT$ for $d = 2, \dots, r$.

Let $G_\sigma$ denote the group of generic stabilizers of $\cV(\sigma)$ which is a finite abelian group. There is an identification
$$
    G_\sigma \cong \left. \left( \bZ^r \cap \sum_{i \in I_\sigma'} \bR b_i \right) \middle/ \sum_{i \in I_\sigma'} \bZ b_i \right.
$$
under which a set of representatives for $G_\sigma$ is given by
$$
    \Box(\sigma) := \left\{ v \in \bZ^r: v = \sum_{i \in I_\sigma'} c_ib_i \text{ for some } 0 \le c_i < 1, c_i \in \bQ \right\}.
$$
For $v = \sum_{i \in I_\sigma'} c_i(v)b_i \in \Box(\sigma)$, we set
$$
    \age(v) := \sum_{i \in I_\sigma'} c_i(v).
$$

\subsubsection{Chen-Ruan orbifold cohomology and pairings}
The inertia stack of the toric orbifold $\cZ$ is
$$
    \cI \cZ = \bigsqcup_{j \in \Box(\cZ)} \cZ_j
$$
where the index set of the components is
$$
    \Box(\cZ) := \bigcup_{\text{cone $\sigma$ in $\Xi$}} \Box(\sigma) = \bigcup_{\sigma \in \Xi(r)} \Box(\sigma).
$$
For each $j \in \Box(\cZ)$, if $\sigma$ is the minimal cone such that $j \in \Box(\sigma)$, then $\cZ_j = \cV(\sigma)$. In particular, the untwisted sector is $\cZ_{\vzero} = \cZ$.

Let $\bF = \bQ$ or $\bC$. The \emph{Chen-Ruan orbifold cohomology} \cite{CR04} of $\cZ$ with $\bF$-coefficients is
$$
    H_{\CR}^*(\cZ; \bF) := \bigoplus_{j \in \Box(\cZ)} H^*(\cZ_j; \bF)[2\age(j)]
$$
where $[2\age(j)]$ denotes a degree shift by $2\age(j)$. Let $\one_j$ denote the unit of $H^*(\cZ_j; \bF)$ which is an element of $H_{\CR}^{2\age(j)}(\cZ; \bF)$. The Chen-Ruan orbifold cohomology \emph{with compact support} is
$$
    H_{\CR, c}^*(\cZ; \bF) := \bigoplus_{j \in \Box(\cZ)} H^*_c(\cZ_j; \bF)[2\age(j)].
$$
There is a perfect pairing
\begin{equation}\label{eqn:CpctPairing}
    \left(-, - \right)_{\cZ}: H^*_{\CR}(\cZ; \bF) \times H^{2r-*}_{\CR,c}(\cZ; \bF) \to \bF.
\end{equation}

In addition, for a subtorus $\mathbb{S} \subseteq \bT$, the \emph{$\mathbb{S}$-equivariant} Chen-Ruan cohomology group of $\cZ$ is
$$
    H_{\CR, \mathbb{S}}^*(\cZ; \bF) := \bigoplus_{j \in \Box(\cZ)} H^*_{\mathbb{S}}(\cZ_j; \bF)[2\age(j)],
$$
which is a module over $H^*_{\mathbb{S}}(\pt;\bF)$. Specializing to the Calabi-Yau subtorus $\mathbb{S} = \bT'$, the $\bT'$-equivariant Poincar\'e pairing
\begin{equation}\label{eqn:EquivPairing}
    \left(-, - \right)_{\cZ}^{\bT'}: H^*_{\CR, \bT'}(\cZ; \bF) \times H^{*}_{\CR, \bT'}(\cZ; \bF) \to \cQ_{\bT', \bF}
\end{equation}
is perfect, where $\cQ_{\bT', \bF}$ denotes the field of fractions of $H^*_{\bT'}(\pt; \bF)$.


\subsubsection{Fixed points, invariant lines, flags}
A maximal cone $\sigma \in \Xi(r)$ corresponds to a $\bT$-fixed point $\fp_{\sigma} := \cV(\sigma)$ in $\cZ$, and a codimension-1 cone $\tau \in \Xi(r-1)$ corresponds to a $\bT$-invariant line $\fl_\tau := \cV(\tau)$ whose coarse moduli space is either $\bC^1$ or $\bP^1$. Let $\Xi(r-1)_c$ denote the subset of cones in $\Xi(r-1)$ whose corresponding lines are compact.

A \emph{flag} in the fan $\Xi$ is a pair $(\tau, \sigma) \in \Xi(r-1) \times \Xi(r)$ such that $\tau$ is a facet of $\sigma$. This corresponds to the inclusion of the $\bT$-fixed point $\fp_{\sigma}$ into the $\bT$-invariant line $\fl_{\tau}$. Let $F(\Xi)$ denote the set of flags in $\Xi$. We set
$$
    \fr(\tau, \sigma):= \frac{|G_\sigma|}{|G_\tau|}.
$$
There is an exact sequence
$$
    \xymatrix{
        1 \ar[r] & G_\tau \ar[r] & G_\sigma \ar[r] & \mu_{\fr(\tau, \sigma)} \ar[r] & 1
    }
$$
where for $a \in \bZ_{>0}$, $\mu_a$ denotes the cyclic group of $a$-th roots of unity.

\subsection{3-Dimensional open geometry}\label{sect:OpenGeometry}
Let $\cX$ be a toric Calabi-Yau 3-orbifold with semi-projective coarse moduli space $X$. Let $\bSi = (N, \Sigma, \alpha)$ be the extended stacky fan of $\cX$ defined as in Section \ref{sect:Notations}, where $N \cong \bZ^3$. Let $T := N \otimes \bC^* \cong (\bC^*)^3$ be the Deligne-Mumford torus of $\cX$. Let $M := \Hom(N, \bZ) = \Hom(T, \bC^*)$ and $u_3 \in M$ be given by the Calabi-Yau condition, with corresponding character $\su_3$. Let $T':= \ker(\su_3) \cong (\bC^*)^2$ denote the Calabi-Yau subtorus of $T$.

Let $\Delta$ be the 2-dimensional convex lattice polytope associated to $\cX$, and write
$$
    \Delta \cap N = \{b_1, \dots, b_{R'}, \dots, b_R\}
$$
among which $b_1, \dots, b_{R'}$ are the primitive generators of the rays of $\Sigma$:
$$
    \Sigma(1) = \{\rho_1, \dots, \rho_{R'}\}, \qquad \rho_i := \bR_{\ge 0}b_i.
$$
We have $\inner{u_3, b_i} = 1$ for all $i = 1, \dots, R$. This list of vectors characterizes the surjective homomorphism $\alpha: \bZ^R \to N$. Let $\bL := \ker(\alpha) \cong \bZ^{R-3}$, which fits into the exact sequence
\begin{equation}\label{eqn:XSES}
    \xymatrix{
        0 \ar[r] & \bL \ar[r]^\psi & \bZ^R \ar[r]^\alpha & N \ar[r] & 0.
    }
\end{equation}
The inclusion map $\psi$ above gives a linear action of the connected algebraic torus $\bL \otimes \bC^* \cong (\bC^*)^{R-3}$ on $\bC^R$ under which $\cX$ arises as a GIT quotient. The map $\psi$ and the induced action can be described in coordinates in terms of a collection of \emph{charge vectors}. Specifically, let $\{e_1, \dots, e_R\}$ be a basis of $\bZ^R$, so that $\alpha(e_i) = b_i$, and let $\{\epsilon_1, \dots, \epsilon_{R-3}\}$ be a basis of $\bL$. For $a = 1, \dots, R-3$ we have the charge vector
$$
    l^{(a)} = (l^{(a)}_1, \dots, l^{(a)}_R) := \psi(\epsilon_a) \in \bZ^R.
$$

In the 3-dimensional case, the stabilizer group $G_\tau$ for any 2-cone $\tau \in \Sigma(2)$ is cyclic. For any flag $(\tau, \sigma) \in F(\Sigma)$ we denote 
$$
    \fm(\tau, \sigma):= |G_\tau|.
$$

The boundary condition for the open geometry is a Lagrangian suborbifold $\cL \subset \cX$ of \emph{Aganagic-Vafa type}; see e.g. \cite{AV00,FLT12,LY22}. It is invariant under the action of the maximal compact subgroup $T'_{\bR} \cong U(1)^2$ of the Calabi-Yau torus $T'$. Moreover, it intersects a unique $T$-invariant line $\fl_{\tau_0}$ in $\cX$, for some $\tau_0 \in \Sigma(2)$. We have
$$
    \pi_1(\cL) = H_1(\cL; \bZ) \cong \bZ \times G_{\tau_0}.
$$
Let $L \subset X$ denote the coarse moduli space of $\cL$.

As in \cite{LY22}, we assume that $\cL$ is an \emph{outer brane}, i.e. $\tau_0 \not \in \Sigma(2)_c$. Then $\tau_0$ is contained in a unique 3-cone $\sigma_0 \in \Sigma(3)$. Without loss of generality, we take the indexing
$$
    I'_{\tau_0} = \{2, 3\}, \qquad I'_{\sigma_0} = \{1, 2, 3\}
$$
such that $b_1, b_2, b_3$ appear in $\Delta$ in counterclockwise order. Let $\tau_2, \tau_3 \in \Sigma(2)$ be the other two facets of $\sigma_0$ with
\begin{equation}\label{eqn:OtherFacets}
    I'_{\tau_2} = \{1, 3\}, \qquad I'_{\tau_3} = \{1, 2\}.
\end{equation}
Let
$$
    \fr := \fr(\tau_0, \sigma_0), \qquad \fm := \fm(\tau_0, \sigma_0) = |G_{\tau_0}|.
$$
The flag $(\tau_0, \sigma_0)$ determines a preferred basis $\{v_1, v_2, v_3\}$ of $N$ under which we have the following coordinates:
$$
    b_1 = (\fr, -\fs, 1), \qquad b_2 = (0, \fm, 1), \qquad b_3 = (0, 0, 1)
$$
where $\fs \in \{0, \dots, \fr-1\}$. We use this basis to introduce coordinates to the other vectors as
$$
    b_i = (m_i, n_i, 1), \qquad i = 1, \dots, R.
$$
Moreover, let $\{u_1, u_2, u_3\}$ be the dual basis of $M$ (where $u_3$ comes from the Calabi-Yau condition as before), and $\{\su_1, \su_2, \su_3\}$ denote the corresponding characters of $T$. We have
$$
    H^*_T(\pt; \bZ) \cong \bZ[\su_1, \su_2, \su_3], \qquad H^*_{T'}(\pt; \bZ) \cong \bZ[\su_1, \su_2].
$$

We introduce an additional rational parameter $f \in \bQ$ called the \emph{framing} of the Aganagic-Vafa brane $\cL$. We write $f$ as the fraction
$$
    f = \frac{\fb}{\fa}
$$
where $\fa \in \bZ_{>0}, \fb \in \bZ$ are coprime. This determines a 1-dimensional subtorus $T_f := \ker(\fa\su_1 - \fb\su_2) \subset T'$. We have
$$
    H^*_{T_f}(\pt; \bZ) \cong \bZ[\su]
$$
where $\su$ is a generator of $H^2_{T_f}(\pt; \bZ) \cong \bZ$. The natural restriction $H^*_{T'}(\pt; \bZ) \to H^*_{T_f}(\pt; \bZ)$ is given by
$$
    \su_1 \mapsto \fa\su, \qquad \su_2 \mapsto \fb\su.
$$



\subsection{4-Dimensional closed geometry}\label{sect:ClosedGeometry}
The closed geometry corresponding to the open geometry $(\cX, \cL, f)$ is a toric Calabi-Yau 4-orbifold $\cX$ with semi-projective coarse moduli space $\tX$. The extended stacky fan is $\tbSi = (\tSi, \tN, \talpha)$ where $\tN := N \oplus \bZ v_4$ is a 4-dimensional lattice and the surjective homomorphism
$$
    \talpha: \bZ^{R+2} = \bigoplus_{i = 1}^{R+2} \bZ e_i \to \tN = \bZ v_1 \oplus \bZ v_2 \oplus \bZ v_3 \oplus \bZ v_4, \qquad e_i \mapsto \tb_i := \talpha(e_i)
$$
is defined by the coordinates
$$
    \tb_i = (m_i, n_i, 1, 0), \qquad i = 1, \dots, R,
$$
$$
    \tb_{R+1} = (-\fa, -\fb, 1, 1), \quad \tb_{R+2} = (0, 0, 1, 1).
$$
The map $\talpha$ fits into the exact sequence in the second row of the following commutative diagram, in which the first row is \eqref{eqn:XSES}:
\begin{equation}\label{eqn:tXSES}
    \xymatrix{
        0 \ar[r] & \bL \ar[r]^\psi \ar[d] & \bZ^R \ar[r]^\alpha \ar[d] & N \ar[r] \ar[d] & 0\\
        0 \ar[r] & \tbL \ar[r]^{\tpsi} & \bZ^{R+2} \ar[r]^{\talpha} & \tN \ar[r] & 0.
    }
\end{equation}
Here $\tbL := \ker(\talpha) \cong \bZ^{R-2}$ contains $\bL$ as a sublattice. Similar to the case of $\cX$, the inclusion map $\tpsi$ gives a linear action of $\tbL \otimes \bC^* \cong (\bC^*)^{R-2}$ on $\bC^{R+2}$ under which $\tcX$ arises as a GIT quotient. Extending the basis $\{\epsilon_1, \dots, \epsilon_{R-3}\}$ of $\bL$ to a basis of $\tbL$ by adding an additional vector $\epsilon_{R-2} \in \tbL$, we can describe the action by charge vectors
$$
    \tl^{(a)} := \tpsi(\epsilon_a) \in \bZ^{R+2}, \qquad a = 1, \dots, R-2
$$
where for $a = 1, \dots, R-3$ we have
$$
    \tl^{(a)} = (l_1^{(a)}, \dots, l_{R}^{(a)}, 0, 0).
$$

Let $\tT := \tN \otimes \bC^* \cong (\bC^*)^4$ be the Deligne-Mumford torus of $\tcX$ and $\tM := \Hom(\tN, \bZ) = \Hom(\tT, \bC^*)$. If $\{u_1, u_2, u_3, u_4\}$ is the basis of $\tM$ dual to $\{v_1, v_2, v_3, v_4\}$, then $u_3$ is the element specifying the Calabi-Yau condition: $\inner{u_3, \tb_i} = 1$ for all $i = 1, \dots, R+2$. Let $\{\su_1, \su_2, \su_3, \su_4\}$ denote the corresponding characters of $\tT$. Let $\tT' := \ker(\su_3) \cong (\bC^*)^3$ denote the Calabi-Yau subtorus of $\tT$. We have
$$
    H^*_{\tT}(\pt; \bZ) \cong \bZ[\su_1, \su_2, \su_3, \su_4], \qquad H^*_{\tT'}(\pt; \bZ) \cong \bZ[\su_1, \su_2, \su_4].
$$


Now we describe the fan $\tSi$. Let $\tDelta$ be the convex hull of $\Delta \cup \{\tb_{R+1}, \tb_{R+2}\}$ which is the 3-dimensional convex lattice polytope associated to $\tcX$. Note that $\tDelta$ contains $\Delta$ as the facet defined by $u_4 = 0$, and
$$
    \tDelta \cap \tN = \{\tb_1, \dots, \tb_{R+2}\}.
$$
Take the unique triangulation of $\tDelta$ that extends the triangulation of the convex hull of $\Delta \cup \{\tb_{R+2}\}$ given by the cone over the triangulation of $\Delta$ at the vertex $\tb_{R+2}$. The fan $\tSi$ is specified by the cone over this triangulation at the origin in $\tN \otimes \bR$. Note that it contains $\Sigma$ as a subfan, which corresponds to the inclusion
$$
    \iota: \cX \to \cX \times \bC \to \tcX.
$$
In terms of cones, $\tSi(d) \supset \Sigma(d)$ for $d = 0, 1, 2, 3$. In particular, the set of rays is given by
$$
    \tSi(1) = \{\trho_{1}, \dots, \trho_{R'}, \trho_{R+1}, \trho_{R+2}\}, \qquad \trho_i:= \bR_{\ge 0} \tb_i.
$$
In addition, for $d = 0,1,2,3$, there is an injective map
\begin{equation}\label{eqn:ConeMapIota}
    \iota: \Sigma(d) \to \tSi(d+1)
\end{equation}
defined by
$$
    I_{\iota(\sigma)}' = I_{\sigma}' \sqcup \{R+2\}.
$$
In particular, it identifies the set $\Sigma(3)$ of maximal cones in $\Sigma$ as a subset of the set $\tSi(4)$ of maximal cones in $\tSi$, and thus also the corresponding torus fixed points. For any cone $\sigma$ in $\Sigma$, we have
$$
    G_{\iota(\sigma)} = G_{\sigma}, \qquad \Box(\iota(\sigma)) = \Box(\sigma).
$$

\subsection{Additional 4-cones}\label{sect:ExtraCones}
We focus on the additional maximal cones in $\tSi$. For any $\tsi \in \tSi(4) \setminus \iota(\Sigma(3))$, there exists a 2-cone $\delta_0(\tsi) \in \Sigma(2) \setminus \Sigma(2)_c$ such that
$$
    I'_{\tsi} = I'_{\delta_0(\tsi)} \sqcup \{R+1, R+2\}.
$$
Note that $\iota(\delta_0(\tsi)) \in \tSi(3)_c$. We denote
$$
    I'_{\delta_0(\tsi)} = \{i_2(\tsi), i_3(\tsi)\}
$$
such that $i_2(\tsi), i_3(\tsi) \in \{1, \dots, R'\}$ appears on the boundary of $\Delta$ in counterclockwise order. The facets of $\tsi$ are $\iota(\delta_0(\tsi))$, $\delta_2(\tsi)$, $\delta_3(\tsi)$, $\delta_4(\tsi) \in \tSi(3)$ described by
\begin{align*}
    & I'_{\iota(\delta_0(\tsi))} = \{i_2(\tsi), i_3(\tsi), R+2\}, && I_{\delta_4(\tsi)}' = \{i_2(\tsi), i_3(\tsi), R+1\},\\
    & I_{\delta_2(\tsi)}' = \{i_3(\tsi), R+1, R+2\}, && I_{\delta_3(\tsi)}' = \{i_2(\tsi), R+1, R+2\}.
\end{align*}
The structure of the stabilizer group $G_{\tsi}$ is described by \cite[Lemma 2.4]{LY22} as follows.

\begin{lemma}[\cite{LY22}]\label{lem:ExtraStab}
For any $\tsi \in \tSi(4) \setminus \iota(\Sigma(3))$, we have:
\begin{itemize}
    \item $G_{\tsi}$ is a cyclic group of order
        $$
            |G_{\tsi}| = (\fa n_{i_2(\tsi)} - \fb m_{i_2(\tsi)}) - (\fa n_{i_3(\tsi)} - \fb m_{i_3(\tsi)}).
        $$
    \item Elements of age at most 1 are precisely those contained in the cyclic subgroup
    $$
        G_{\delta_0(\tsi)} \cong \mu_{\gcd(|m_{i_2(\tsi)} - m_{i_3(\tsi)}|, |n_{i_2(\tsi)} - n_{i_3(\tsi)}|)}.
    $$

    \item Elements in $G_{\tsi} \setminus G_{\delta_0(\tsi)}$ have age 2. 

\end{itemize}
\end{lemma}

In particular, let $\tsi_0 \in \tSi(4) \setminus \iota(\Sigma(3))$ denote the 4-cone with $$
    I_{\tsi_0}' = \{2, 3, R+1, R+2\},
$$
that is, $\delta_0(\tsi_0) = \tau_0$, $i_2(\tsi_0) = 2$, and $i_3(\tsi_0) = 3$. Then $G_{\tsi_0}$ is a cyclic group of order $\fa \fm$ and the subgroup of elements of age at most 1 is $G_{\tau_0} \cong \mu_{\fm}$. 

Let 
$$
    S:= |\tSi(4) \setminus \iota(\Sigma(3))|
$$
and consider the total ordering 
\begin{equation}\label{eqn:ConeList}
  \tsi^1, \dots, \tsi^S
\end{equation}
of the cones in $\tSi(4) \setminus \iota(\Sigma(3))$ under which $b_{i_2(\tsi^1)}, \dots, b_{i_2(\tsi^S)}$ appear on the boundary of $\Delta$ in counterclockwise order. We denote
$$
  \tau^s := \delta_0(\tsi^s), \quad i_2^s := i_2(\tsi^s), \quad i_3^s := i_3(\tsi^s), \qquad \text{for } s = 1, \dots, S
$$
so that $I'_{\tau^s} = \{i_2^s, i_3^s\}$. Then, 
$$
  i_3^s = i_2^{s+1} \qquad \text{for } s = 1, \dots, S-1.
$$
We define
$$
    c_0 := \fa n_{i_2^1} - \fb m_{i_2^1}, \qquad c_s := \fa n_{i_3^s} - \fb m_{i_3^s} \qquad \text{for } s = 1, \dots, S.
$$
Then the integers $c_0, \dots, c_S$ are strictly monotone decreasing and define a subdivision of the interval
\begin{equation}\label{eqn:PolyDelta0}
    \Delta_0 := [c_S = \fa n_{i_3^S} - \fb m_{i_3^S}, c_0 = \fa n_{i_2^1} - \fb m_{i_2^1}]
\end{equation}
into $S$ subintervals. The interval $\Delta_0$ may be viewed as the image of the polytope $\Delta$ under the projection
$$
    \bR^2 \to \bR, \qquad (m, n) \mapsto \fa n - \fb m,
$$
where the domain $\bR^2$ is identified with the hyperplane of $N \otimes \bR \cong \bR^3$ defined by $\inner{u_3, -} = 1$. See Figure \ref{fig:OrderingExtraCones} for an illustration.

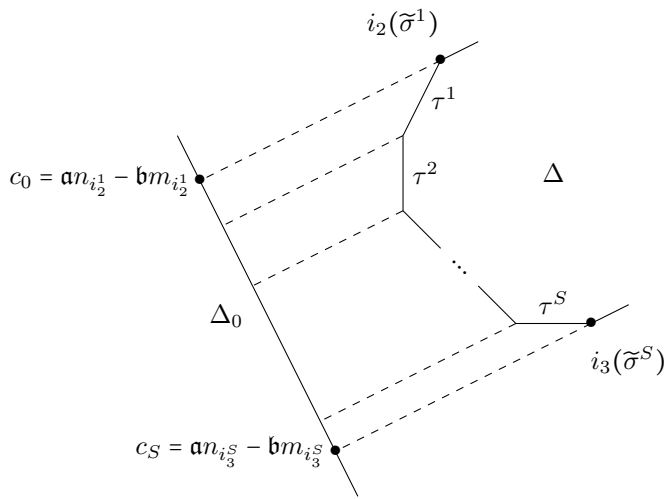
\begin{figure}[h]
\begin{center}
    \begin{tikzpicture}[scale=1]
        
        \coordinate (1) at (0.5, 2);
        \coordinate (2) at (0, 1);        
        \coordinate (3) at (0, 0);
        \coordinate (4) at (1.5, -1.5);        
        \coordinate (5) at (2.5, -1.5);

        \draw (1, 2.25) -- (1) -- (2) -- (3) -- (0.5,-0.5);
        \draw (1, -1) -- (4) -- (5) -- (3, -1.25);

        \node at (0.75, -0.75) {$\ddots$};

        \node at (2, 0.5) {$\Delta$};

        \node[right] at (0.25,1.5) {$\tau^1$};
        \node[right] at (0,0.5) {$\tau^2$};
        \node[above] at (2,-1.5) {$\tau^S$};

        \node at (1) {$\bullet$};
        \node at (5) {$\bullet$};
        \node at (0, 2.5) {$i_2(\tsi^1)$};
        \node at (3, -2) {$i_3(\tsi^S)$};

        \coordinate (min) at (-2.7, 0.4);
        \coordinate (max) at (-0.9, -3.2);

        \draw (-3, 1) -- (-0.6, -3.8);

        \node[left] at (-2, -1.4) {$\Delta_0$};

        \node at (min) {$\bullet$};
        \node at (max) {$\bullet$};

        \node[left] at (min) {$c_0 = \fa n_{i_2^1}- \fb m_{i_2^1}$};
        \node[left] at (max) {$c_S = \fa n_{i_3^S}- \fb m_{i_3^S}$};

        \draw[dashed] (1) -- (min);
        \draw[dashed] (5) -- (max); 

        \draw[dashed] (2) -- (-2.4, -0.2);
        \draw[dashed] (3) -- (-2, -1);
        \draw[dashed] (4) -- (-1.1, -2.8);

    \end{tikzpicture}
\end{center}

    \caption{The projection relating $\Delta$ and $\Delta_0$.}
    \label{fig:OrderingExtraCones}
\end{figure}

Let $\Vol(\Delta)$, $\Vol(\tDelta)$, $\Vol(\Delta_0)$ be the normalized volumes\footnote{This means that the volume of a standard simplex of any dimension is normalized to 1.} of the polytopes $\Delta$, $\tDelta$, $\Delta_0$ respectively. We have
$$
    \Vol(\tDelta) = \sum_{\tsi \in \tSi(4)} |G_{\tsi}| = \sum_{\sigma \in \Sigma(3)}|G_\sigma| + \sum_{\tsi \in \tSi(4) \setminus \iota(\Sigma(3))} |G_{\tsi}| 
   = \Vol(\Delta) + \sum_{s=1}^S (c_{s-1} - c_s)
   = \Vol(\Delta) + (c_0 - c_S),
$$
that is,
\begin{equation}\label{eqn:PolyVolSum}
    \Vol(\tDelta) = \Vol(\Delta) + \Vol(\Delta_0).
\end{equation}

\subsection{A 2-dimensional geometry}\label{sect:2DGeometry}
The data of the additional 4-cones in $\tSi(4) \setminus \iota(\Sigma(3))$ and specifically the interval $\Delta_0$ with the induced subdivision can be packaged into a toric Calabi-Yau 2-orbifold $\cX_0$ with semi-projective coarse moduli space. In more detail, the stacky fan of $\cX_0$ is a triple $\bSi'_0 = (N_0, \Sigma_0, \alpha'_0)$ where $N_0 \cong \bZ^2$. The fan $\Sigma_0$ contains $S+1$ rays spanned by the lattice points
$$
    (c_0, 1), \dots, (c_S, 1)
$$
in $N_0$ respectively where the coordinates are taken under a $\bZ$-basis. Moreover, $\Sigma_0$ contains $S$ 2-cones each spanned by a pair of neighboring rays, and for $s = 1, \dots, S$, the 2-cone spanned by $(c_{s-1}, 1)$ and $(c_s, 1)$ has a cyclic stabilizer group isomorphic to $G_{\tsi^s}$. 

\begin{example}\label{ex:C3} \rm{
In this paper we will use $\cX = \bC^3$, $\cL$ an outer brane, and $f=1$ as the running example; see Figure \ref{fig:ExC3} for an illustration of the geometries in terms of the respective triangulated polytopes, and \cite[Section 2.6]{LY22} for additional examples. The parameters in this example are $\fa = \fb = \fr = \fm = 1$, $\fs = 0$. The corresponding toric Calabi-Yau 4-fold is $\tcX = \Tot(\cO_{\bP^2}(-2) \oplus \cO_{\bP^2}(-1))$. We have $S = 2$ and the two 4-cones $\tsi^1, \tsi^2$ in the ordering \eqref{eqn:ConeList} are described by
$$
    I'_{\tsi^1} = \{2, 3, 4, 5\}, \qquad I'_{\tsi^2} = \{1, 3, 4, 5\}.
$$
We have $\Delta_0 = [c_2 = -1, c_0 = 1]$, $c_1 = 0$, and the induced toric Calabi-Yau surface is $\cX_0 = \Tot(K_{\bP^1})$. The volumes of the polytopes are
$$
    \Vol(\Delta) = 1, \qquad \Vol(\tDelta) = 3, \qquad \Vol(\Delta_0) = 2.
$$

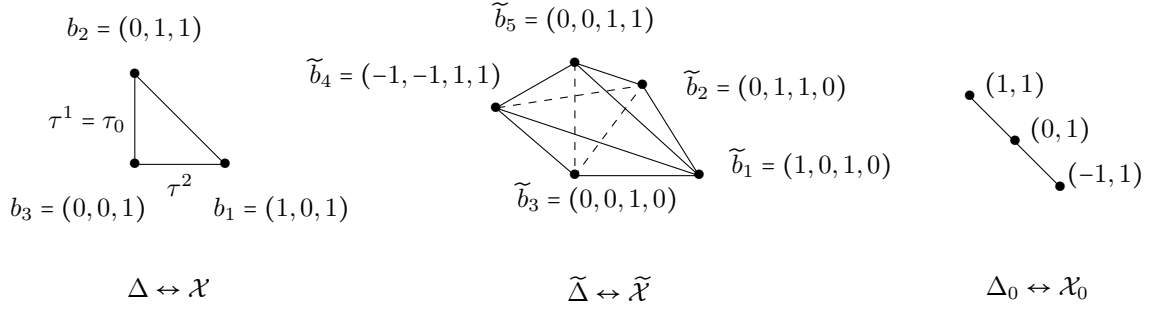
\begin{figure}[h]
\begin{center}
    \begin{tikzpicture}[scale=1.5]
        \coordinate (01) at (-3.1, 0);
        \coordinate (02) at (-3.9, 0.8);        
        \coordinate (03) at (-3.9, 0);
                
        \node at (01) {$\bullet$};
        \node at (02) {$\bullet$};
        \node at (03) {$\bullet$};

        \node at (-2.6, -0.4) {$b_1 = (1, 0, 1)$};
        \node at (-3.9, 1.2) {$b_2 = (0, 1, 1)$};
        \node at (-4.4, -0.4) {$b_3 = (0, 0, 1)$};

        \node[left] at (-3.9, 0.4) {$\tau^1 = \tau_0$};
        \node[below] at (-3.5, 0) {$\tau^2$};

        \draw (01) -- (02) -- (03) -- (01);

        \node at (-3.6, -1.1) {$\Delta \leftrightarrow \cX$};


        \coordinate (1) at (1.1, -0.1);
        \coordinate (2) at (0.6, 0.7);        
        \coordinate (3) at (0, -0.1);
        \coordinate (4) at (-0.7, 0.5);        
        \coordinate (5) at (0, 0.9);

        \node at (1) {$\bullet$};
        \node at (2) {$\bullet$};
        \node at (3) {$\bullet$};
        \node at (4) {$\bullet$};
        \node at (5) {$\bullet$};

        \node at (2.1, 0) {$\tb_1 = (1, 0, 1, 0)$};
        \node at (1.7, 0.7) {$\tb_2 = (0, 1, 1, 0)$};
        \node at (0.2, -0.3) {$\tb_3 = (0, 0, 1, 0)$};
        \node at (-1.5, 0.8) {$\tb_4 = (-1, -1, 1, 1)$};
        \node at (0, 1.3) {$\tb_5 = (0, 0, 1, 1)$};

        \draw (1) -- (3) -- (4) -- (5) -- (2) -- (1) -- (4);
        \draw (1) -- (5);
        \draw[dashed] (5) -- (3) -- (2) -- (4);

        \node at (0.3, -1.1) {$\tDelta \leftrightarrow \tcX$};

        \coordinate (11) at (4.3, -0.2);
        \coordinate (12) at (3.5, 0.6);        
        \coordinate (13) at (3.9, 0.2);
                
        \node at (11) {$\bullet$};
        \node at (12) {$\bullet$};
        \node at (13) {$\bullet$};

        \node at (4.7, -0.1) {$(-1, 1)$};
        \node at (3.9, 0.7) {$(1, 1)$};
        \node at (4.3, 0.3) {$(0, 1)$};

        \draw (11) -- (13) -- (12);

        \node at (4.1, -1.1) {$\Delta_0 \leftrightarrow \cX_0$};

    \end{tikzpicture}
\end{center}

    \caption{Triangulated polytopes in the example of $\cX = \bC^3$ and $f=1$.}
    \label{fig:ExC3}
\end{figure}
}\end{example}

\begin{remark}\label{rem:MultiPhases} \rm{
We note that we will arrive at the same 4-dimensional geometry $\tcX$ and 2-dimensional geometry $\cX_0$ if we start with $\cX$ and choose an Aganagic-Vafa outer brane corresponding to any of the 2-cones $\tau^1, \dots, \tau^S$, with appropriate framing. To be more specific, for any $s = 1, \dots, S$, we may instead consider the Aganagic-Vafa outer brane $\cL^s$ that intersects $\fl_{\tau^s}$ and pick the framing $f_s \in \bQ$ in a way that the framing 1-torus is the same as the torus $T_f$ defined by $(\cL, f)$. To explicitly compute $f_s$, let $\sigma^s \in \Sigma(3)$ be the unique 3-cone that contains $\tau^s$ and write $I'_{\sigma^s} = I'_{\tau^s} \sqcup \{i_1^s\}$. We then consider the $\bZ$-basis $\{v_1^s, v_2^s, v_3^s\}$ of $N$ under which we have the coordinates:
$$
    b_{i_1^s} = (\fr(\tau^s, \sigma^s), -\fs(\tau^s, \sigma^s), 1), \qquad b_{i_2^s} = (0, \fm(\tau^s, \sigma^s), 1), \qquad b_{i_3^s} = (0, 0, 1)
$$
where $\fs(\tau^s, \sigma^s) \in \{0, \dots, \fr(\tau^s, \sigma^s)-1\}$. Under the $\bZ$-basis $\{v_1^s, v_2^s, v_3^s, v_4\}$ of $\tN$, the vector $\tb_{R+1}$ now has coordinates
$
    (-\fa_s, -\fb_s, 1, 1)
$
for some $\fa_s \in \bZ_{>0}, \fb_s \in \bZ$. Then we have
$$
    f_s = \frac{\fb_s}{\fa_s}.
$$
In the example of $\cX = \bC^3$ and $f > 0$, if $(\cL^2, f_2)$ is the framed brane corresponding to the 2-cone $\tau^2$ (with $I'_{\tau^2} = \{1, 3\}$), we have
$$
  f_2 = \frac{-f-1}{f}.
$$
}\end{remark}

%% file: hyper.tex

\section{Hypergeometric functions}\label{sect:Hyper}
In this section, we recall the results from \cite[Section 6]{LY22} on the open/closed correspondence of B-model hypergeometric functions. We set up additional definitions along the way.

\subsection{Chen-Ruan cohomology and preferred bases}
We start by describing the Chen-Ruan cohomology of $\cX$ and $\tcX$ and fixing preferred choices of bases and their equivariant lifts.

\subsubsection{Chen-Ruan cohomology}
As in \cite[Section 2.5.4]{LY22}, we may use the results of \cite{BCS05,BH06,Jiang08,JT08} to obtain the following description of the Chen-Ruan cohomology of the semi-projective toric orbifolds $\cX$ and $\tcX$: For $\bF = \bQ$ or $\bC$, as graded $\bF$-algebras, $H^*_{\CR}(\cX; \bF)$ is generated by the divisor classes
$$
    \cD_i := [\cV(\rho_i)] \in H^2_{\CR}(\cX; \bZ), \qquad i = 1, \dots, R'
$$
and the units $\{\one_j : j \in \Box(\cX)\}$, while $H^*_{\CR}(\tcX; \bF)$ is generated by the divisor classes 
$$
    \tcD_i := [\cV(\trho_i)] \in H^2_{\CR}(\tcX; \bZ), \qquad i = 1, \dots, R', R+1, R+2
$$
and the units $\{\one_j : j \in \Box(\tcX)\}$. The homogenous $\bF$-algebra homomorphism $\iota^*: H^*_{\CR}(\tcX; \bF) \to H^*_{\CR}(\cX; \bF)$ induced by the inclusion $\iota: \cX \to \tcX$ is given by
$$
        \tcD_i \mapsto \cD_i \quad \text{ for } i = 1, \dots, R', \qquad \tcD_{R+1}, \tcD_{R+2} \mapsto 0,
$$
$$
    \one_j \mapsto \begin{cases}
        \one_j & \text{if } j \in \Box(\cX),\\
        0 &  \text{if } j \in \Box(\tcX) \setminus \Box(\cX).
        \end{cases}
$$

\subsubsection{Second cohomology and divisor classes}
Now we focus on second cohomology and consider the commutative diagram
$$
    \xymatrix{
        0 \ar[r] & \tM \ar[r]^{\talpha^\vee} \ar[d] & \bZ^{R+2} \ar[r]^{\tpsi^\vee} \ar[d] & \tbL^\vee \ar[r] \ar[d] & 0\\
        0 \ar[r] & M \ar[r]^{\alpha^\vee} & \bZ^R \ar[r]^{\psi^\vee} & \bL^\vee \ar[r] & 0
    }
$$
obtained by applying $\Hom(-, \bZ)$ to \eqref{eqn:tXSES}, where the rows are exact. We define
\begin{align*}
    & D_i := \psi^\vee(e_i^\vee) \in \bL^\vee, \qquad \text{for } i = 1, \dots, R;\\
    & \tD_i := \tpsi^\vee(e_i^\vee) \in \tbL^\vee, \qquad \text{for } i = 1, \dots, R+2,
\end{align*}
where $\{e_1^\vee, \dots, e_R^\vee\}$ (resp. $\{e_1^\vee, \dots, e_{R+2}^\vee\}$) is the basis of $\bZ^R$ (resp. $\bZ^{R+2}$) dual to $\{e_1, \dots, e_R\}$ (resp. $\{e_1, \dots, e_{R+2}\}$). There is a canonical identification $H^2_{\CR}(\cX;\bQ) \cong \bL^\vee_{\bQ} := \bL^\vee \otimes \bQ$ under which $\cD_i$ is identified with $D_i$ for each $i = 1, \dots, R'$, and $\one_{j(i)}$ with $D_i$ if $b_i \in N$ is a representative of $j(i) \in \Box(\cX)$. Similarly, there is a canonical identification $H^2_{\CR}(\tcX;\bQ) \cong \tbL^\vee_{\bQ} := \tbL^\vee \otimes \bQ$ under which $\tcD_i$ is identified with $\tD_i$ for each $i = 1, \dots, R', R+1, R+2$, and $\one_{j(i)}$ with $\tD_i$ if $\tb_i \in \tN$ is a representative of $j(i) \in \Box(\tcX)$. The projection $\tbL^\vee_{\bQ} \to \bL^\vee_{\bQ}$ is canonically identified with $\iota^*: H^2_{\CR}(\tcX; \bQ) \to H^2_{\CR}(\cX; \bQ)$, under which $\tD_i$ projects to $D_i$ for $i = 1, \dots, R$. Moreover, we have
$$
    \tD_{R+1} = -\tD_{R+2}  \qquad \in \bZ_{\neq 0} \epsilon_{R-2}^\vee
$$
and both project to $0 \in \bL^\vee$, where $\{\epsilon_1^\vee, \dots, \epsilon_{R-2}^\vee\}$ is the basis of $\tbL^\vee$ dual to $\{\epsilon_1, \dots, \epsilon_{R-2}\}$.

Moreover, recall that $\cX$ is a GIT quotient of $\bC^R$ by the action of $\bL \otimes \bC^* \cong (\bC^*)^{R-3}$ and $\tcX$ is a GIT quotient of $\bC^{R+2}$ by the action of $\tbL \otimes \bC^* \cong (\bC^*)^{R-2}$. We have the following surjective \emph{Kirwan maps}
\begin{align*}
    & \kappa: \bL^\vee \cong H^2_{(\bC^*)^{R-3}}(\bC^R; \bZ) \to H^2(\cX; \bZ),\\
    & \tkappa: \tbL^\vee \cong H^2_{(\bC^*)^{R-2}}(\bC^{R+2}; \bZ) \to H^2(\tcX; \bZ).
\end{align*}

We will also need equivariant versions of the divisor classes. For $i = 1, \dots, R'$, we set
$$
    \cD_i^{T'} := [\cV(\rho_i)] \in H^2_{\CR,T'}(\cX; \bQ)
$$
whose non-equivariant limit is $\cD_i$. For $i = 1, \dots, R', R+1, R+2$, we set
$$
    \tcD_i^{\tT'} := [\cV(\trho_i)] \in H^2_{\CR,\tT'}(\tcX; \bQ)
$$
whose non-equivariant limit is $\tcD_i$.

\subsubsection{Extended nef cones}
For $\sigma \in \Sigma(3)$, the extended $\sigma$-nef cone is defined as
$$
    \tNef(\sigma) := \sum_{i \in I_\sigma} \bR_{\ge 0} D_i
$$
which is a top-dimensional cone in $\bL^\vee_{\bR} := \bL^\vee \otimes \bR$. The \emph{extended nef cone} of $\cX$ is defined as
$$
    \tNef(\cX) := \bigcap_{\sigma \in \Sigma(3)} \tNef(\sigma).
$$
Similarly, for $\tsi \in \tSi(4)$, the extended $\tsi$-nef cone is defined as
$$
    \tNef(\tsi) := \sum_{i \in I_\sigma} \bR_{\ge 0} \tD_i
$$
which is a top-dimensional cone in $\tbL^\vee_{\bR} := \tbL^\vee \otimes \bR$. The \emph{extended nef cone} of $\tcX$ is defined as
$$
    \tNef(\tcX) := \bigcap_{\tsi \in \tSi(4)} \tNef(\tsi).
$$
It is checked in \cite[Lemma 6.2]{LY22} that $\tD_{R+1} \in \tNef(\tcX)$.

\subsection{Preferred bases and equivariant lifts}\label{sect:Basis} 
As in \cite[Section 6.2]{LY22}, we fix a choice of elements
$$
    H_1, \dots, H_{R-3} \in \bL^\vee \cap \tNef(\cX), \qquad \tH_1, \dots, \tH_{R-2} \in \tbL^\vee \cap \tNef(\tcX)
$$
for the rest of the paper. The elements are required to satisfy the following conditions:
\begin{itemize}
    \item $\{H_1, \dots, H_{R-3}\}$ is a $\bQ$-basis of $\bL^\vee_{\bQ}$ and $\{\tH_1, \dots, \tH_{R-2}\}$ is a $\bQ$-basis of $\tbL^\vee_{\bQ}$.
    
    \item The images of $H_1, \dots, H_{R'-3}$ under the Kirwan map $\kappa$ form a $\bQ$-basis of $H^2(\cX; \bQ)$ and the images of $\tH_1, \dots, \tH_{R'-3}, \tH_{R-2}$ under the Kirwan map $\tkappa$ form a $\bQ$-basis of $H^2(\tcX; \bQ)$.
    
    \item For $a = 1, \dots, R'-3$, $\tH_a$ is the lift of $H_a$ under the projection $\tbL^\vee \to \bL^\vee$ that is contained in the $\bZ$-span of $\tD_1, \dots, \tD_{R'}$.

    \item For $a = R'-2, \dots, R-3$, we choose $H_a = D_{3+a}$ and $\tH_a = \tD_{3+a}$.

    \item We choose $\tH_{R-2} = \fa \tD_{R+1}$.
\end{itemize}

For $a = 1, \dots, R-3$, we write $H_a$ in terms of the $\bQ$-basis $\{D_4, \dots, D_R\}$ of $\bL^\vee_{\bQ}$ as
$$
  H_a = \sum_{i = 4}^R s_{ai}D_i
$$
for $s_{ai} \in \bQ_{\ge 0}$. By rescaling if necessary, we may assume that $\{H_a\}$ is chosen such that $s_{ai} \in \bZ_{\ge 0}$ for any $a = 1, \dots, R-3$, $i = 4, \dots, R$. By construction, we can write $\{\tH_a\}$ in terms fo the $\bQ$-basis $\{\tD_4, \dots, \tD_R, \tD_{R+1}\}$ of $\tbL^\vee_{\bQ}$ as
$$
  \tH_a = \sum_{i = 4}^R s_{ai}\tD_i, \quad a = 1, \dots, R-3,\qquad  \tH_{R-2} = \fa \tD_{R+1}.
$$

We use the above elements to define the following second cohomology classes:
\begin{align*}
    & u_a := \kappa(H_a), && a = 1, \dots, R'-3;\\
    & \tu_a := \tkappa(\tH_a), && a = 1, \dots, R'-3, R-2;\\
    & u_a = \tu_a := \one_{j(3+a)}, && a = R'-2, \dots, R-3
\end{align*}
where $j(3+a) \in \Box(\cX) \subseteq \Box(\tcX)$ is the age-1 element represented by $b_{3+a}$ or $\tb_{3+a}$. Then $\{u_1, \dots, u_{R-3}\}$ is a $\bQ$-basis of $H^2_{\CR}(\cX; \bQ)$ and $\{\tu_1, \dots, \tu_{R-2}\}$ is a $\bQ$-basis of $H^2_{\CR}(\tcX; \bQ)$.

Moreover, we fix the following equivariant lifts $\{u_1^{T'}, \dots, u_{R-3}^{T'}\}$, $\{\tu_1^{\tT'}, \dots, \tu_{R-2}^{\tT'}\}$ of the bases elements above:
\begin{itemize}
    \item For $a = 1, \dots, R'-3$, let $u_a^{T'} \in H^2_{\CR, T'}(\cX;\bQ)$ be the unique lift of $u_a$ such that $\iota_{\sigma_0}^*(u_a^{T'}) = 0$, and $\tu_a^{\tT'} \in H^2_{\CR, \tT'}(\tcX;\bQ)$ be the unique lift of $\tu_a$ such that $\iota_{\iota(\sigma_0)}^*(\tu_a^{\tT'}) = \iota_{\tsi_0}^*(\tu_a^{\tT'}) = 0$.
    
    \item Let $\tu_{R-2}^{\tT'} \in H^2_{\CR, \tT'}(\tcX;\bQ)$ be the unique lift of $\tu_{R-2}$ such that $\iota_{\tsi_0}^*(\tu_{R-2}^{\tT'}) = 0$.
    
    \item For $a = R'-2, \dots, R-3$, let $u_a^{T'} = \tu_a^{\tT'} := \one_{j(3+a)}$.
\end{itemize}

\subsection{Extended Mori cones and effective classes}
For $\bF = \bQ$ or $\bR$, let $\bL_{\bF} := \bL \otimes \bF$ and $\tbL_{\bF} := \tbL \otimes \bF$. Let $\inner{-,-}$ denote the pairings between $\bL^\vee$ and $\bL$ and between $\tbL^\vee$ and $\tbL$ as well as extensions over $\bF$.

\subsubsection{Extended Mori cones}
For $\sigma \in \Sigma(3)$, the extended $\sigma$-Mori cone is the dual cone of $\tNef(\sigma)$:
$$
    \tNE(\sigma) := \{\beta \in \bL_{\bR} : \inner{D, \beta} \ge 0 \text{ for all } D \in \tNef(\sigma) \}.
$$
The \emph{extended Mori cone} of $\cX$ is defined as
$$
    \tNE(\cX) := \bigcup_{\sigma \in \Sigma(3)} \tNE(\sigma).
$$
Similarly, for $\tsi \in \tSi(4)$, the extended $\tsi$-Mori cone is the dual cone of $\tNef(\tsi)$:
$$
    \tNE(\tsi) := \{\tbeta \in \tbL_{\bR} : \inner{\tD, \tbeta} \ge 0 \text{ for all } \tD \in \tNef(\tsi) \}.
$$
The \emph{extended Mori cone} of $\tcX$ is defined as
$$
    \tNE(\tcX) := \bigcup_{\tsi \in \tSi(4)} \tNE(\tsi).
$$

\subsubsection{Effective classes}
For $\sigma \in \Sigma(3)$, consider
$$
    \bK^\vee_\sigma := \sum_{i \in I_\sigma} \bZ D_i
$$
which is a sublattice of $\bL^\vee$ of finite index. The dual lattice
$$
    \bK_\sigma := \{ \beta \in \bL_{\bQ} : \inner{D, \beta} \in \bZ \text{ for all } D \in \bK^\vee_\sigma \}
$$
is an overlattice of $\bL$ in $\bL_{\bQ}$. The map
$$
    v: \bK_\sigma \to N, \qquad \beta \mapsto \sum_{i = 1}^R \ceil{\inner{D_i, \beta}} b_i
$$
induces a bijection $\bK_\sigma / \bL \to \Box(\sigma)$. We set
$$
    \bK_{\eff,\sigma} := \bK_\sigma \cap \tNE(\sigma).
$$
Taking into account of all maximal cones, we set
$$
    \bK_{\eff}(\cX) := \bigcup_{\sigma \in \Sigma(3)} \bK_{\eff, \sigma}.
$$
Similarly, for $\tsi \in \tSi(4)$, consider
$$
    \bK^\vee_{\tsi} := \sum_{i \in I_{\tsi}} \bZ \tD_i
$$
which is a sublattice of $\tbL^\vee$ of finite index. The dual lattice
$$
    \bK_{\tsi} := \{ \tbeta \in \tbL_{\bQ} : \inner{\tD, \tbeta} \in \bZ \text{ for all } \tD \in \bK^\vee_{\tsi} \}
$$
is an overlattice of $\tbL$ in $\tbL_{\bQ}$. The map
$$
    \tv: \bK_{\tsi} \to \tN, \qquad \tbeta \mapsto \sum_{i = 1}^{R+2} \ceil{\inner{\tD_i, \tbeta}} \tb_i
$$
induces a bijection $\bK_{\tsi} / \tbL \to \Box(\tsi)$. We set
$$
    \bK_{\eff,\tsi} := \bK_{\tsi} \cap \tNE(\tsi).
$$
Taking into account of all maximal cones, we set
$$
    \bK_{\eff}(\tcX) := \bigcup_{\tsi \in \tSi(4)} \bK_{\eff, \tsi}.
$$

\subsection{Small $I$-functions}\label{sect:IFunction}
We introduce formal variables $z$ and
$$
    q = (q_1, \dots, q_{R-3}), \qquad \tq = (\tq_1, \dots, \tq_{R-2}).
$$
For $\beta \in \bK(\cX)$, $\tbeta \in \bK(\tcX)$ we set
$$
    q^\beta := q_1^{\inner{H_1, \beta}} \cdots q_{R-3}^{\inner{H_{R-3}, \beta}}, \qquad \tq^{\tbeta} := \tq_1^{\inner{\tH_1, \tbeta}} \cdots \tq_{R-2}^{\inner{\tH_{R-2}, \tbeta}}.
$$
The condition that $s_{ai} \in \bZ_{\ge 0}$ for all $a, i$ implies that the above are monomials.

Now we define the \emph{small $I$-functions}. The non-equivariant small $I$-function of $\cX$ is
\begin{align*}
    I_{\cX}&(q, z) := e^{\frac{1}{z}\left(\sum_{a \in \{1, \dots, R'-3\}}u_a \log q_a \right)} \\
    & \cdot \sum_{\beta \in \bK_{\eff}(\cX)} q^{\beta} \prod_{i \in \{1, \dots, R'\}} \frac{\prod_{m = \ceil{\inner{D_i, \beta}}}^\infty (\frac{\cD_i}{z} + \inner{D_i, \beta} - m)}{\prod_{m = 0}^\infty (\frac{\cD_i}{z} + \inner{D_i, \beta} - m)}
    \cdot \prod_{i = R'+1}^R \frac{\prod_{m = \ceil{\inner{D_i, \beta}}}^\infty (\inner{D_i, \beta} - m)}{\prod_{m = 0}^\infty (\inner{D_i, \beta} - m)} \frac{\one_{v(\beta)}}{z^{\age(v(\beta))}}.
\end{align*}
Similarly, the non-equivariant small $I$-function of $\tcX$ is
\begin{align*}
    I_{\tcX}&(\tq, z) := e^{\frac{1}{z}\left(\sum_{a \in \{1, \dots, R'-3, R-2\}}\tu_a \log \tq_a \right)} \\
    & \cdot \sum_{\tbeta \in \bK_{\eff}(\tcX)} \tq^{\tbeta} \prod_{i \in \{1, \dots, R', R+1, R+2\}} \frac{\prod_{m = \ceil{\inner{\tD_i, \tbeta}}}^\infty (\frac{\tcD_i}{z} + \inner{\tD_i, \tbeta} - m)}{\prod_{m = 0}^\infty (\frac{\tcD_i}{z} + \inner{\tD_i, \tbeta} - m)}
    \cdot \prod_{i = R'+1}^R \frac{\prod_{m = \ceil{\inner{\tD_i, \tbeta}}}^\infty (\inner{\tD_i, \tbeta} - m)}{\prod_{m = 0}^\infty (\inner{\tD_i, \tbeta} - m)} \frac{\one_{\tv(\tbeta)}}{z^{\age(\tv(\tbeta))}}.
\end{align*}
Moreover, the $\tT'$-equivariant small $I$-function of $\tcX$ is
\begin{align*}
    I_{\tcX}^{\tT'}&(\tq, z) := e^{\frac{1}{z}\left(\sum_{a \in \{1, \dots, R'-3, R-2\}}\tu_a^{\tT'} \log \tq_a \right)} \\
    & \cdot \sum_{\tbeta \in \bK_{\eff}(\tcX)} \tq^{\tbeta} \prod_{i \in \{1, \dots, R', R+1, R+2\}} \frac{\prod_{m = \ceil{\inner{\tD_i, \tbeta}}}^\infty (\frac{\tcD_i^{\tT'}}{z} + \inner{\tD_i, \tbeta} - m)}{\prod_{m = 0}^\infty (\frac{\tcD_i^{\tT'}}{z} + \inner{\tD_i, \tbeta} - m)}
    \cdot \prod_{i = R'+1}^R \frac{\prod_{m = \ceil{\inner{\tD_i, \tbeta}}}^\infty (\inner{\tD_i, \tbeta} - m)}{\prod_{m = 0}^\infty (\inner{\tD_i, \tbeta} - m)} \frac{\one_{\tv(\tbeta)}}{z^{\age(\tv(\tbeta))}}.
\end{align*}

\subsection{The B-model disk function}\label{sect:DiskFunction}
Recall that the Aganagic-Vafa outer brane $\cL$ in $\cX$ determines a preferred flag $(\tau_0, \sigma_0) \in F(\Sigma)$. Define the quantities
$$
    w_0 := \frac{1}{\fr}, \qquad w_2 := \frac{\fs+ \fr f}{\fr\fm}, \qquad w_3:= -\frac{\fm + \fs + \fr f}{\fr\fm}
$$
which are obtained from tangent $\tT'$-weights at the fixed point $\fp_{\sigma_0}$ as in \cite[Section 3.4.1]{LY22}. Let
$$
    \bK_{\eff}(\cX, \cL) := \{(\beta, d) \in \bK_{\eff, \sigma_0} \times \bZ_{\neq 0} : \inner{D_1, \beta} + dw_0 \in \bZ_{\ge 0}\}.
$$
The proof of \cite[Lemma 6.7]{LY22} shows that $\bK_{\eff}(\cX, \cL)$ is in bijection with $\bK_{\eff, \tsi_0} \setminus \bL_{\bQ}$.

Consider the surjective group homomorphisms
$$
    h: \bZ \times G_{\tau_0} \cong \pi_1(\fo_{\tau_0}) \to \pi_1(\fu_{(\tau_0, \sigma_0)}) \cong G_{\sigma_0},
$$
$$
    \th: \bZ \times G_{\tau_0} \cong \pi_1(\fo_{\iota(\tau_0)}) \to \pi_1(\fu_{(\iota(\tau_0), \tsi_0)}) \cong G_{\tsi_0} \cong \mu_{\fa\fm}
$$
defined by the induced maps on fundamental groups as in \cite[Section 2.1.3]{LY22}, where $\fo_{\tau_0} = \fo_{\iota(\tau_0)} \cong \bC^* \times \cB G_{\tau_0}$ denotes the (open) torus orbit corresponding to the cone $\tau_0$ or $\iota(\tau_0)$, and $\fu_{(\tau_0, \sigma_0)} = \fo_{\tau_0} \cup \fp_{\sigma_0}$, $\fu_{(\iota(\tau_0), \tsi_0)} = \fo_{\iota(\tau_0)} \cup \fp_{\tsi_0}$. For $(\beta, d) \in \bK_{\eff}(\cX, \cL)$, there is a unique element $\lambda(\beta, d) \in G_{\tau_0}$ such that
$$
    h(d, \lambda(\beta, d)) = v(\beta).
$$
Let $\epsilon_2(\beta, d), \epsilon_3(\beta, d) \in \bQ \cap [0,1)$ such that $h(d, \lambda(\beta, d))$ acts on $T_{\fp_{\sigma_0}}\fl_{\tau_2}$ and $T_{\fp_{\sigma_0}}\fl_{\tau_3}$ (see \eqref{eqn:OtherFacets}) by multiplication by $e^{2\pi\sqrt{-1}\epsilon_2(\beta, d)}$ and $e^{2\pi\sqrt{-1}\epsilon_3(\beta, d)}$ respectively. Moreover, we define
$$
    \tk(\beta, d) := \th(d, \lambda(\beta, d)) \in G_{\tsi_0} \cong \mu_{\fa\fm}.
$$

Now we define the \emph{B-model disk function} of $(\cX, \cL, f)$ following \cite{FL13,FLT12,LY22}. Let $x$ be a formal variable for the open sector. For $\tk \in \mu_{\fa\fm}$, define
$$
    W^{\cX, (\cL, f)}_{\tk}(q, x) = \sum_{\substack{(\beta, d) \in \bK_{\eff}(\cX, \cL)\\ \tk = \tk(\beta,d)}} q^\beta x^d \frac{(-1)^{\floor{dw_3-\epsilon_3(\beta, d)}+ \ceil{\frac{d}{\fa}}}}{\fm d (\inner{D_1, \beta} + dw_0)! \prod_{i = 4}^R \inner{D_i, \beta}!} \cdot \frac{\prod_{m = 1}^\infty (-\inner{D_3, \beta} - dw_3 - m)}{\prod_{m = 0}^\infty (\inner{D_2, \beta} + dw_2 - m)}.
$$

\subsection{Open/closed correspondence of hypergeometric functions}
The main result of \cite{LY22} on the B-model open/closed correspondence is that the B-model disk function of the open geometry $(\cX, \cL, f)$ can be retrieved from the equivariant $I$-function of the corresponding closed geometry $\tcX$ by taking the equivariant Poincar\'e pairing with certain cohomology classes. For $\tk \in G_{\tsi_0} \cong \mu_{\fa \fm}$, we define 
\begin{equation}\label{eqn:ExtraInsertionOuter}
    \tgamma_{\tk} := \begin{cases}
        \frac{\tcD_2^{\tT'}\tcD_3^{\tT'}\tcD_{R+1}^{\tT'}}{\frac{f}{\fm}\su_1 - \frac{1}{\fm}\su_2 - \su_4} & \text{if } \tk = 1,\\\
        \tcD_{R+1}^{\tT'} \one_{\tk^{-1}} & \text{if } \age(\tk) = 1,\\
        \one_{\tk^{-1}} & \text{if } \age(\tk) = 2
    \end{cases}
    \qquad \in H^4_{\CR, \tT'}(\tcX; \bQ).
\end{equation}

\begin{notation}\label{not:ZCoefficient}\rm{
Let $\bS$ be a commutative ring and consider a power series
$$
    h(z) = \sum_{n = 0}^\infty h_nz^{-n}    \qquad \in \bS\llbracket z^{-1} \rrbracket
$$
in $z^{-1}$ with coefficients $h_n \in \bS$. Then for $n \in \bZ_{\ge 0}$ we set
$$
    [z^{-n}]h(z) := h_n.
$$
}
\end{notation}

Now we state the correspondence of hypergeometric functions obtained in \cite[Theorem 1.4]{LY22}.

\begin{theorem}[\cite{LY22}]\label{thm:HyperCorr}
For $\tk \in G_{\tsi_0} \cong \mu_{\fa \fm}$, we have
$$
    W^{\cX, (\cL, f)}_{\tk}(q, x) = [z^{-2}] \left(I_{\tcX}^{\tT'}(\tq, z), \tgamma_{\tk} \right)_{\tcX}^{\tT'} \bigg|_{\su_4 = 0, \su_2 - f\su_1 = 0}
$$
under the change of variables $\tq_a = q_a$ for $a = 1, \dots, R-3$ and $\tq_{R-2} = x$.
\end{theorem}

Here, $(-,-)_{\tcX}^{\tT'}$ is the $\tT'$-equivariant Poincar\'e pairing of $\tcX$ (see \eqref{eqn:EquivPairing}) and $\big|_{\su_4 = 0, \su_2 - f\su_1 = 0}$ denotes the weight restriction to the subtorus $T_f \subset \tT'$.

\begin{remark}\label{rem:Mirror}\rm{
As discussed in Section \ref{sect:Overview}, Theorem \ref{thm:HyperCorr} has an A-model counterpart in terms of generating functions of Gromov-Witten invariants, that is, the the A-model disk function of $(\cX, \cL, f)$ can be retrieved from the equivariant $J$-function of $\tcX$ by taking the Poincar\'e pairing with the classes $\tgamma_{\tk}$. We refer to \cite{LY22} for details.
}\end{remark}


%% file: picardfuchs.tex

\section{Extended Picard-Fuchs system and solutions from open strings}\label{sect:PicardFuchs}
In this section, we compare the Picard-Fuchs systems associated to $\cX$, $\tcX$ and show that the latter is an extension of the former (Proposition \ref{prop:SolnSubspace}), following \cite{Mayr01,LM01}. We further use the open/closed correspondence of hypergeometric functions (Theorem \ref{thm:HyperCorr}) to describe the additional solutions to the extended system in terms of the open string data on $\cX$ (Propositions \ref{prop:OpenMirrorMapSoln}, \ref{prop:DiskFnSoln}).

\subsection{Extended Picard-Fuchs system}\label{sect:PicardFuchsDef}
We start by defining the Picard-Fuchs systems. We express $D_1, \dots, D_R \in \bL^\vee$ (resp. $\tD_1, \dots, \tD_{R+2} \in \tbL^\vee$) in terms of the preferred basis $\{H_1, \dots, H_{R-3}\}$ (resp. $\{\tH_1, \dots, \tH_{R-2}\}$) as follows:
$$
  D_i = \sum_{a = 1}^{R-3} m_i^{(a)}H_a, \quad i = 1, \dots, R; \qquad \qquad
  \tD_i = \sum_{a = 1}^{R-2} \tm_i^{(a)}\tH_a, \quad i = 1, \dots, R+2.
$$
(See Section \ref{sect:Basis} for the definition of the bases.) The structural coefficients $m_i^{(a)}$, $\tm_i^{(a)}$ are described by the following slight extension of \cite[Lemma 6.11]{LY22}.

\begin{lemma}[\cite{LY22}]\label{lem:Charges}
We have
\begin{enumerate}[label=(\roman*)]
  \item $(\tm_{R+1}^{(1)}, \dots, \tm_{R+1}^{(R-2)}) = -(\tm_{R+2}^{(1)}, \dots, \tm_{R+2}^{(R-2)}) = \left(0, \dots, 0, \frac{1}{\fa} \right)$; 
  \item $(\tm_1^{(R-2)}, \dots, \tm_{R+2}^{(R-2)}) = \left(w_0, w_2, w_3, 0, \dots, 0, \frac{1}{\fa}, -\frac{1}{\fa} \right)$;
  \item $m_i^{(a)} = \tm_i^{(a)}$ for $i = 1, \dots, R$ and $a = 1, \dots, R-3$.
\end{enumerate}
\end{lemma}

\begin{proof}
Items (i) and (ii) are \cite[Lemma 6.11]{LY22}. Item (iii) follows from (ii) and the construction that under the projection $\tbL^\vee \to \bL^\vee$, $\tD_i$ projects to $D_i$ for $i = 1, \dots, R$, $\tD_{R+1}$ and $\tD_{R+2}$ projects to $0$, and $\tH_a$ projects to $H_a$ for $a = 1, \dots, R-3$.
\end{proof}

Indeed, when $\cX$ is smooth and $f \in \bZ$ ($\fa = 1$), item (ii) above has integer entries and is the extended charge vector obtained in \cite{Mayr01,LM01} from the open string sector. In general the extended charge vector is a scalar multiple.

For $\beta \in \bL$, we define the differential operator
$$
  \cP_\beta := q^\beta \prod_{\substack{i \in \{1, \dots, R\} \\ \inner{D_i, \beta} < 0}} \prod_{m=0}^{-\inner{D_i, \beta}-1} (\cP_i - m) - \prod_{\substack{i \in \{1, \dots, R\} \\ \inner{D_i, \beta} > 0}} \prod_{m=0}^{-\inner{D_i, \beta}-1} (\cP_i - m)
$$ 
where for $i = 1, \dots, R$,
$$
  \cP_i := \sum_{a = 1}^{R-3} m_i^{(a)}q_a\frac{\partial}{\partial q_a}.
$$
Similarly, for $\tbeta \in \tbL$, we define the differential operator
$$
  \tcP_{\tbeta} := \tq^{\tbeta} \prod_{\substack{i \in \{1, \dots, R+2\} \\ \inner{\tD_i, \tbeta} < 0}} \prod_{m=0}^{-\inner{\tD_i, \tbeta}-1} (\tcP_i - m) - \prod_{\substack{i \in \{1, \dots, R+2\} \\ \inner{\tD_i, \tbeta} > 0}} \prod_{m=0}^{-\inner{\tD_i, \tbeta}-1} (\tcP_i - m)
$$ 
where for $i = 1, \dots, R+2$,
$$
  \tcP_i := \sum_{a = 1}^{R-2} \tm_i^{(a)}\tq_a\frac{\partial}{\partial \tq_a}.
$$

The \emph{Picard-Fuchs} systems associated to $\cX$ and $\tcX$ are given by
$$
  \cP := \{\cP_\beta\}_{\beta \in \bL}, \qquad \tcP := \{\tcP_{\tbeta}\}_{\tbeta \in \tbL}
$$
respectively. We refer to $\tcP$ as the \emph{extended} Picard-Fuchs system following \cite{Mayr01,LM01} since, as shown below, every solution to $\cP$ is also a solution to $\tcP$.

\begin{proposition}\label{prop:SolnSubspace}
Suppose $F(q_1, \dots, q_{R-3})$ is such that
$
  \cP_\beta F(q_1, \dots, q_{R-3}) = 0
$
for all $\beta \in \bL$. Then
$
  \tcP_{\tbeta} F(\tq_1, \dots, \tq_{R-3}) = 0
$
for all $\tbeta \in \tbL$.
\end{proposition}

\begin{proof}
Under the identification $\tq_a = q_a$ for $a = 1, \dots, R-3$, we have by Lemma \ref{lem:Charges} that
$$
  \tcP_1 = \cP_1 + w_0\tq_{R-2}\frac{\partial}{\partial \tq_{R-2}}, \quad \tcP_2 = \cP_2 + w_2\tq_{R-2}\frac{\partial}{\partial \tq_{R-2}}, \quad \tcP_3 = \cP_3 + w_3\tq_{R-2}\frac{\partial}{\partial \tq_{R-2}},
$$
$$
  \tcP_i = \cP_i \qquad \text{for } i = 4, \dots, R,
$$
$$
  \tcP_{R+1} = \frac{1}{\fa} \tq_{R-2}\frac{\partial}{\partial \tq_{R-2}}, \quad \tcP_{R+2} = -\frac{1}{\fa}\tq_{R-2}\frac{\partial}{\partial \tq_{R-2}}.
$$
Therefore, for any function $G=G(q_1, \dots, q_{R-3})$ that only depends on the first $R-3$ variables, we have
\begin{equation}\label{eqn:PiActionSame}
  \tcP_iG(\tq_1, \dots, \tq_{R-3}) = \cP_iG(q_1, \dots, q_{R-3}) \qquad \text{for } i = 1, \dots, R,
\end{equation}
\begin{equation}\label{eqn:ExtraPiVanish}
  \tcP_{R+1}G = \tcP_{R+2}G  = 0.
\end{equation}

Now we consider the operator $\tcP_{\tbeta}$ corresponding to some $\tbeta \in \tbL$. If $\tbeta \in \bL$, observe that $\inner{\tD_i, \tbeta} = \inner{D_i, \tbeta}$ for $i = 1, \dots, R$ and $\inner{\tD_{R+1}, \tbeta} = -\inner{\tD_{R+2}, \tbeta} = 0$. This together with \eqref{eqn:PiActionSame} implies that
$$
  \tcP_{\tbeta}F(\tq_1, \dots, \tq_{R-3}) = \cP_{\tbeta}F(q_1, \dots, q_{R-3}),
$$
which is 0 by assumption. On the other hand, if $\tbeta \not \in \bL$, then $\inner{\tD_{R+1}, \tbeta} = -\inner{D_{R+2}, \tbeta} \neq 0$. It follows that one of the two products in $\tcP_{\tbeta}$ contains $\tcP_{R+1}$ as a multiplicative factor, and the other contains $\tcP_{R+2}$ as a multiplicative factor. Then \eqref{eqn:ExtraPiVanish} implies that $\tcP_{\tbeta}F = 0$.
\end{proof}

Given Proposition \ref{prop:SolnSubspace}, our goal is to understand the additional solutions to the extended system $\tcP$ that do not come from solutions to $\cP$. By Iritani \cite{Iritani09}, we have the following characterization of the dimensions of solution spaces (over $\bC$).

\begin{theorem}[\cite{Iritani09}]\label{thm:PFSolnDimension}
The dimension of the solution space to the system $\cP$ is
$$
  \dim_{\bC} H^*_{\CR}(\cX; \bC) = \Vol(\Delta)
$$
and the dimension of the solution space to the extended system $\tcP$ is
$$
  \dim_{\bC} H^*_{\CR}(\tcX; \bC) = \Vol(\tDelta).
$$
\end{theorem}

By \eqref{eqn:PolyVolSum}, the difference in the dimensions of solutions spaces to $\cP$ and $\tcP$ is
\begin{equation}\label{eqn:DimDiff}
  \Vol(\tDelta) - \Vol(\Delta) = \Vol(\Delta_0) = \sum_{s=1}^S |G_{\tsi^s}|.
\end{equation}

\subsection{Solutions as coefficients of non-equivariant \texorpdfstring{$I$}{I}-functions}

To characterize the additional solutions to the extended system $\tcP$, we use the description due to Givental \cite{Givental98} of the solutions to Picard-Fuchs systems as coefficients of the non-equivariant small $I$-functions
$$
  I_{\cX}(q, z), \qquad I_{\tcX}(\tq, z)
$$
of $\cX$ and $\tcX$. (See also \cite[Lemma 4.6]{Iritani09}; see \cite[Section 3.9--3.11]{FLZ20} for a discussion for toric Calabi-Yau 3-orbifolds.) The $I$-functions take value in $H^*_{\CR}(\cX; \bQ)$, $H^*_{\CR}(\tcX; \bQ)$ respectively. Under the inclusion $\iota: \cX \to \tcX$, we have
\begin{equation}\label{eqn:IPullback}
  \iota^*\left(I_{\tcX}(\tq, z)\right) = I_{\cX}(q, z)
\end{equation}
with $\tq_a = q_a$, $a = 1, \dots, R-3$. (Here, all terms involving $\tq_{R-2}$ pull back to 0.) To extract coefficients of the $I$-functions, we use the perfect pairing \eqref{eqn:CpctPairing}. There is a pushforward map
$$
  \iota_*: H^*_{\CR,c}(\cX; \bC) \to H^{*+2}_{\CR,c}(\tcX; \bC), \qquad \gamma_c \mapsto \gamma_c\tcD_{R+2}.
$$
For any $\tgamma \in H^*_{\CR}(\tcX; \bC)$ and $\gamma_c \in H^*_{\CR,c}(\cX; \bC)$, we have
\begin{equation}\label{eqn:CohPairingEqual}
  \left(\iota^*(\tgamma), \gamma_c \right)_{\cX} = \left(\tgamma, \iota_*(\gamma_c) \right)_{\tcX} = \left(\tgamma, \gamma_c\tcD_{R+2} \right)_{\tcX}.
\end{equation}

\begin{theorem}[\cite{Givental98}]\label{thm:SolnICoeff}
A basis of the solution space to the Picard-Fuchs system $\cP$ consists of
$$
  [z^{-\deg(\gamma_c)/2}]\left(I_{\cX}(q, z), \gamma_c \right)_{\cX}
$$
where $\gamma_c$ ranges through any homogeneous basis of $H^*_{\CR,c}(\cX; \bC)$. A basis of the solution space to the Picard-Fuchs system $\tcP$ consists of
$$
  [z^{-\deg(\tgamma_c)/2}]\left(I_{\tcX}(\tq, z), \tgamma_c \right)_{\tcX}
$$
where $\tgamma_c$ ranges through any homogeneous basis of $H^*_{\CR,c}(\tcX; \bC)$.
\end{theorem}

See Notation \ref{not:ZCoefficient} for the notation $[z^{-n}]$. We now consider the expansion of the $I$-functions in powers of $z^{-1}$.

\subsubsection{Zeroth-order terms}
The zeroth-order terms in the expansion are
$$
  [z^0]I_{\cX}(q, z) = 1 = \left(I_{\cX}(q, z), [\pt] \right)_{\cX}, \qquad [z^0]I_{\tcX}(\tq, z) = 1 = \left(I_{\tcX}(\tq, z), [\pt] \right)_{\tcX}.
$$

\subsubsection{First-order terms}
Under the bases $\{u_1, \dots, u_{R-3}\}$, $\{\tu_1, \dots, \tu_{R-2}\}$ defined in Section \ref{sect:Basis}, the first-order terms in the expansion can be written as
$$
  [z^{-1}]I_{\cX}(q, z) = \sum_{a = 1}^{R-3} \tau_a(q) u_a, \qquad 
  [z^{-1}]I_{\tcX}(\tq, z) = \sum_{a = 1}^{R-2} \ttau_a(\tq) \tu_a.
$$
The coefficients $\tau_1(q), \dots, \tau_{R-3}(q)$ are solutions to $\cP$ and $\ttau_1(\tq), \dots, \ttau_{R-2}(\tq)$ are solutions to $\tcP$. (To obtain these solutions from the pairings as in Theorem \ref{thm:SolnICoeff}, we may pair the $I$-functions with the dual bases of $H^4_{\CR,c}(\cX;\bC)$, $H^6_{\CR,c}(\tcX;\bC)$ respectively. These coefficients specify the \emph{(closed) mirror maps} under which the $I$-functions are identified with the $J$-functions from Gromov-Witten theory. They are explicitly calculated in \cite[Sections 6.4, 6.5]{LY22}. In particular, \cite[Proposition 6.14]{LY22} gives the identification 
$$
  \tau_a(q) = \ttau_a(\tq), \qquad a = 1, \dots, R-3
$$
of solutions, and further identifies the additional solution $\ttau_{R-2}(\tq)$ to $\tcP$ with the \emph{open mirror map} under which the $B$-model disk function is identified with the $A$-model disk function. The open mirror map is explicitly calculated in \cite[Section 6.5]{LY22}. We will return to this point in Proposition \ref{prop:OpenMirrorMapSoln} below.

\subsubsection{Second-order terms}
Now we turn to the second-order terms in the expansion. Note that there are no higher-order terms since $H^{\ge 6}_{\CR}(\cX;\bC) = H^{\ge 6}_{\CR}(\tcX;\bC) = 0$. For the remainder of this section, let $A :=  \dim_{\bC} H^4_{\CR}(\cX;\bC)$. Let $\{u_{1,c}, \dots, u_{A,c}\}$ be a basis of $H^2_{\CR,c}(\cX;\bC)$. Then the coefficients of the dual basis of $H^4_{\CR}(\cX;\bC)$ in $I_{\cX}(q, z)$ are
$$
  \left(I_{\cX}(q, z), u_{a,c} \right)_{\cX}, \qquad a = 1, \dots, A,
$$
which are the remaining solutions to $\cP$.

From the above basis, we can construct the following basis of $H^4_{\CR,c}(\tcX;\bC) \cong H^4_{\CR}(\tcX;\bC)$:
\begin{equation}\label{eqn:CompactH4Basis}
\begin{aligned}
  &\{\iota_*(u_{1,c}), \dots, \iota_*(u_{A,c})\} \sqcup \{\tcD_{i_2(\tsi^s)}\tcD_{R+2} : s = 2, \dots, S \} \\
  & \sqcup \{\tcD_{R+2} \one_j : j \in \Box(\tau^s), s=1, \dots, S, j \neq \vzero\}
  \sqcup \{\one_j : j \in \Box(\tcX) \setminus \Box(\cX)\}.
\end{aligned}
\end{equation}
Here, observe that:
\begin{itemize}
  \item $\{i_2(\tsi^2), \dots, i_2(\tsi^S)\}$ is the set of indices $i$ such that the divisor $\cV(\rho_i)$ of $\cX$ is non-compact but codimension-2 $\tT$-invariant closed substack $\cV(\iota(\rho_i))$ of $\tcX$ is compact. In terms of fans, $\rho_i$ lies on the boundary of the support of $\Sigma$ while $\trho_i$ lies in the interior of the support of $\tSi$ (cf. Figure \ref{fig:OrderingExtraCones}). 
  
  \item $\bigcup_{s=1}^S \Box(\tau^s) \setminus \{\vzero\}$ is the set of age-1 elements $j \in \Box(\cX) \subseteq \Box(\tcX)$ such that if $j \in \Box(\tau^s)$, then the line $\fl_{\tau^s}$ in the twisted sector $\cX_j$ is non-compact but the line $\fl_{\iota(\tau^s)}$ in $\tcX_j$ is compact.
  
  \item Recall from Lemma \ref{lem:ExtraStab} that any element $j \in \Box(\tcX) \setminus \Box(\cX)$ has age 2 and thus $\one_j$ is a degree-4 class.
\end{itemize}
By \eqref{eqn:IPullback} and \eqref{eqn:CohPairingEqual},
$$
  \left(I_{\tcX}(\tq, z), \iota_*(u_{a,c}) \right)_{\tcX} = \left(I_{\cX}(q, z), u_{a,c} \right)_{\cX}
$$
for $a = 1, \dots, A$ are the coefficients of $I_{\tcX}(\tq, z)$ corresponding to those of $I_{\cX}(q, z)$. The additional solutions result from pairing $I_{\tcX}$ with the additional basis elements and will be described in Proposition \ref{prop:DiskFnSoln} below. We check that the number of such solutions is equal to
$$
  \left| \bigcup_{s=1}^S \Box(\tau^s) \setminus \{\vzero\} \right| + \left| \Box(\tcX) \setminus \Box(\cX) \right| = \sum_{s=1}^S |G_{\tsi^s}| - 1,
$$
which is consistent with the discussion in Section \ref{sect:PicardFuchsDef} and specifically \eqref{eqn:DimDiff} above.

\subsection{Additional solutions from open strings}
In this subsection, we characterize the additional solutions to the extended Picard-Fuchs system $\tcP$ by the open geometry. First, as mentioned above, \cite[Proposition 6.14]{LY22} directly translates into the following result.

\begin{proposition}\label{prop:OpenMirrorMapSoln}
The additional solution $\ttau_{R-2}(\tq)$ is given by the open mirror map
$$
  \log \sX(\tq_1, \dots, \tq_{R-3}, \tq_{R-2})
$$
of $(\cX, \cL, f)$. This solution has form
$$
  \log \tq_{R-2} + \text{power series in }\tq_1, \dots, \tq_{R-3}.
$$
\end{proposition}

Now, we describe the additional solutions resulting from pairing the $I$-function of $\tcX$ with degree-4 classes, using our open/closed correspondence for hypergeometric functions (Theorem \ref{thm:HyperCorr}). Here, we focus on solutions that are related to the distinguished cone $\tsi_0$ determined by the brane $\cL$. To describe the solutions related to another cone $\tsi^s \in \tSi(4) \setminus \iota(\Sigma(3))$ and thereby recover the full solution space, it suffices to adapt the discussion below to an outer brane $\cL^s$ that corresponds to $\tau^s$ and is framed by $f_s \in \bQ$ chosen such that the same corresponding closed geometry $\tcX$ arises; see Remark \ref{rem:MultiPhases}.

For $\tsi_0$, we show that when the ray $\trho_2$ or $\trho_3$ is shared with another cone in $\tSi(4) \setminus \iota(\Sigma(3))$, which happens for certain framings $f$, the power series part of the corresponding solution has a description involving the B-model disk function $W^{\cX, (\cL, f)}$ in the untwisted sector. Moreover, for any non-trivial $j \in \Box(\tsi_0)$, we show that the corresponding solution can be described in terms of $W^{\cX, (\cL, f)}$ in the corresponding twisted sector.

We introduce some notations. For the rest of this section, let $s \in \{1, \dots, S\}$ be the order of $\tsi_0$ in the list \eqref{eqn:ConeList}, i.e. $\tsi_0 = \tsi^s$. When $s>1$, $\tsi_0$ shares the ray $\trho_2$ with $\tsi^{s-1}$ (i.e. $i_3(\tsi^{s-1}) = 2$) and $\tcD_2\tcD_{R+2}$ is an element of the basis \eqref{eqn:CompactH4Basis} of $H^4_{\CR,c}(\tcX;\bC)$. In this case we define
$$
  \tgamma_1^- :=  \frac{\tcD_2^{\tT'}\tcD_{i_2(\tsi^{s-1})}^{\tT'}\tcD_{R+1}^{\tT'}}{\tbw(\delta_3(\tsi^{s-1}), \tsi^{s-1})} \in H^4_{\CR, \tT'}(\tcX; \bQ),
$$
which is the analog of the class $\tgamma_1$ \eqref{eqn:ExtraInsertionOuter} in the untwisted sector $\tk = 1$ but defined for the cone $\tsi^{s-1}$. Let $(\cL^{s-1}, f_{s-1})$ be a framed outer brane that corresponds to $\tau^{s-1}$ and gives rise to the same closed geometry $\tcX$. Then applying Theorem \ref{thm:HyperCorr} gives
$$
    W^{\cX, (\cL^{s-1}, f_{s-1})}_1(q, x) = [z^{-2}] \left(I_{\tcX}^{\tT'}(\tq, z), \tgamma_1^- \right)_{\tcX}^{\tT'} \bigg|_{\su_4 = 0, \su_2 - f\su_1 = 0}.
$$
Similarly, when $s<S$, $\tsi_0$ shares the ray $\trho_3$ with $\tsi^{s+1}$ (i.e. $i_2(\tsi^{s+1}) = 3$) and $\tcD_3\tcD_{R+2}$ is an element of the basis \eqref{eqn:CompactH4Basis}. In this case we define
$$
  \tgamma_1^+ :=  \frac{\tcD_3^{\tT'}\tcD_{i_3(\tsi^{s+1})}^{\tT'}\tcD_{R+1}^{\tT'}}{\tbw(\delta_3(\tsi^{s+1}), \tsi^{s+1})} \in H^4_{\CR, \tT'}(\tcX; \bQ)
$$
which is the analog of the class $\tgamma_1$ \eqref{eqn:ExtraInsertionOuter} in the untwisted sector but defined for the cone $\tsi^{s+1}$. Let $(\cL^{s+1}, f_{s+1})$ be a framed outer brane that corresponds to $\tau^{s+1}$ and gives rise to the same closed geometry $\tcX$. Then applying Theorem \ref{thm:HyperCorr} gives
$$
    W^{\cX, (\cL^{s+1}, f_{s+1})}_1(q, x) = [z^{-2}] \left(I_{\tcX}^{\tT'}(\tq, z), \tgamma_1^+ \right)_{\tcX}^{\tT'} \bigg|_{\su_4 = 0, \su_2 - f\su_1 = 0}.
$$

\begin{proposition}\label{prop:DiskFnSoln}
We have the following description of the additional solutions to $\tcP$:
\begin{enumerate}[label=(\roman*)]
    \item When $s>1$, we have
    \begin{align*}
      [z^{-2}]\left(I_{\tcX}(\tq, z), \tcD_2\tcD_{R+2} \right)_{\tcX} = & c^-(\log \tq_{R-2})^2 + (\log \tq_{R-2})G_1^-(\tq_1, \dots, \tq_{R-3}) + G_2^-(\tq_1, \dots, \tq_{R-3})\\
      & - W^{\cX, (\cL, f)}_1(\tq_1, \dots, \tq_{R-3}, \tq_{R-2})
      + W^{\cX, (\cL^{s-1}, f_{s-1})}_1(\tq_1, \dots, \tq_{R-3}, \tq_{R-2})
    \end{align*}
    for some constant $c^- \in \bQ$ and functions $G_1^-, G_2^-$ in $\tq_1, \dots, \tq_{R-3}$. 

    \item When $s<S$, we have
    \begin{align*}
      [z^{-2}]\left(I_{\tcX}(\tq, z), \tcD_3\tcD_{R+2} \right)_{\tcX} = & c^+(\log \tq_{R-2})^2 + (\log \tq_{R-2})G_1^+(\tq_1, \dots, \tq_{R-3}) + G_2^+(\tq_1, \dots, \tq_{R-3})\\
      & + W^{\cX, (\cL, f)}_1(\tq_1, \dots, \tq_{R-3}, \tq_{R-2})
      - W^{\cX, (\cL^{s+1}, f_{s+1})}_1(\tq_1, \dots, \tq_{R-3}, \tq_{R-2})
    \end{align*}
    for some constant $c^+ \in \bQ$ and functions $G_1^+, G_2^+$ in $\tq_1, \dots, \tq_{R-3}$. 

    \item For any $j \in \Box(\tau_0)$, $j \neq \vzero$, we have
    $$
      [z^{-2}]\left(I_{\tcX}(\tq, z), \tcD_{R+2} \one_j  \right)_{\tcX} = G_j(\tq_1, \dots, \tq_{R-3})- W^{\cX, (\cL, f)}_{\tk}(\tq_1, \dots, \tq_{R-3}, \tq_{R-2})
    $$
    for some function $G_j$ in $\tq_1, \dots, \tq_{R-3}$, where $\tk^{-1} \in G_{\tau_0}$ corresponds to $j$.

    \item For any $j \in \Box(\tsi_0) \setminus \Box(\tau_0)$, we have
    $$
      [z^{-2}]\left(I_{\tcX}(\tq, z), \one_j \right)_{\tcX} = G_j(\tq_1, \dots, \tq_{R-3}) + W^{\cX, (\cL, f)}_{\tk}(\tq_1, \dots, \tq_{R-3}, \tq_{R-2})
    $$
    for some function $G_j$ in $\tq_1, \dots, \tq_{R-3}$, where $\tk^{-1} \in G_{\tsi_0} \setminus G_{\tau_0}$ corresponds to $j$.

\end{enumerate}
\end{proposition}

\begin{proof}
We first prove (i), and (ii) will follow from a similar proof. We consider $\left(I_{\tcX}(\tq, z), \tcD_2\tcD_{R+2} \right)_{\tcX}$ as the non-equivariant limit of the $\tT'$-equivariant Poincar\'e pairing
$$
  \left(I_{\tcX}^{\tT'}(\tq, z), \tcD_2^{\tT'}\tcD_{R+2}^{\tT'} \right)_{\tcX}^{\tT'},
$$
and we may consider the weight restriction $\su_4 = 0, \su_2 - f\su_1 = 0$. As we focus on the power series part of the solution, it suffices to consider the pairing of $\tcD_2^{\tT'}\tcD_{R+2}^{\tT'}$ with
$$
  \sum_{\tbeta \in \bK_{\eff}(\tcX)} \tq^{\tbeta} \prod_{i \in \{1, \dots, R', R+1, R+2\}} \frac{\prod_{m = \ceil{\inner{\tD_i, \tbeta}}}^\infty (\frac{\tcD_i^{\tT'}}{z} + \inner{\tD_i, \tbeta} - m)}{\prod_{m = 0}^\infty (\frac{\tcD_i^{\tT'}}{z} + \inner{\tD_i, \tbeta} - m)}
  \cdot \prod_{i = R'+1}^R \frac{\prod_{m = \ceil{\inner{\tD_i, \tbeta}}}^\infty (\inner{\tD_i, \tbeta} - m)}{\prod_{m = 0}^\infty (\inner{\tD_i, \tbeta} - m)} \frac{\one_{\tv(\tbeta)}}{z^{\age(\tv(\tbeta))}}.
$$
Let $I_{\tbeta}(\tq, z)$ denote the summand corresponding to $\tbeta$ in the above. Moreover, our focus is on the untwisted sector and on terms that depend on $\tq_{R-2}$. Since $\tq^{\tbeta}$ is independent of $\tq_{R-2}$ for any $\tbeta \in \bL$, it suffices to consider the pairing of $\tcD_2^{\tT'}\tcD_{R+2}^{\tT'}$ with each $I_{\tbeta}(\tq, z)$ for $\tbeta$ in
$$
  \{\tbeta \in \bK_{\eff}(\tcX) \setminus \bL : \tv(\tbeta) = \vzero\} = \bK_{\eff}(\tcX) \cap (\tbL \setminus \bL).
$$
Fix such a $\tbeta$. We may assume that the degree of $I_{\tbeta}(\tq, z)$ is 4 as the pairing is otherwise 0. Since $\inner{\tD_{R+1}, \tbeta} = -\inner{\tD_{R+2}, \tbeta} \neq 0$, $[z^{-2}]I_{\tbeta}(\tq, z)$ has form
$$
  c\tq^{\tbeta}\tcD_i^{\tT'}\tcD_{R+1}^{\tT'} \qquad \text{or} \qquad c\tq^{\tbeta}\tcD_i^{\tT'}\tcD_{R+2}^{\tT'}
$$
for some $i \in \{1, \dots, R'\}$ and $c \in \bQ$, $c \neq 0$. Note that for any $\tsi \in \iota(\Sigma(3))$, $\iota_{\tsi}^*(\tcD_{R+1}^{\tT'}) = 0$ and $\iota_{\tsi}^*(\tcD_{R+2}^{\tT'}) = \su_4$. Moreover, the only cones in $\tSi(4) \setminus \iota(\Sigma(3))$ that contain the ray $\trho_2$ are $\tsi_0$ and $\tsi^{s-1}$. It then follows that
\begin{align*}
\left(I_{\tbeta}(\tq, z), \tcD_2^{\tT'}\tcD_{R+2}^{\tT'} \right)_{\tcX}^{\tT'} \bigg|_{\su_4 = 0} &= \sum_{\tsi \in \tSi(4)} \frac{\iota_{\tsi}^*\left(\tcD_2^{\tT'}\tcD_{R+2}^{\tT'}I_{\tbeta}(\tq, z)\right)}{|G_{\tsi}|e_{\tT'}(T_{\fp_{\tsi}}\tcX)} \bigg|_{\su_4 = 0}\\
& =  \frac{\iota_{\tsi_0}^*\left(\tcD_2^{\tT'}\tcD_{R+2}^{\tT'}I_{\tbeta}(\tq, z)\right)}{|G_{\tsi_0}|e_{\tT'}(T_{\fp_{\tsi_0}}\tcX)} \bigg|_{\su_4 = 0} + \frac{\iota_{\tsi^{s-1}}^*\left(\tcD_2^{\tT'}\tcD_{R+2}^{\tT'}I_{\tbeta}(\tq, z)\right)}{|G_{\tsi^{s-1}}|e_{\tT'}(T_{\fp_{\tsi^{s-1}}}\tcX)} \bigg|_{\su_4 = 0}\\
& = \left(I_{\tbeta}(\tq, z), -\tgamma_1\right)_{\tcX}^{\tT'} \bigg|_{\su_4 = 0} + \left(I_{\tbeta}(\tq, z), \tgamma_1^-\right)_{\tcX}^{\tT'} \bigg|_{\su_4 = 0}.
\end{align*}
Here, the last equality follows from that
$$
  \iota_{\tsi_0}^*\left(\tcD_2^{\tT'}\tcD_{R+2}^{\tT'}\right) \big|_{\su_4 = 0} = -\iota_{\tsi_0}^*\left(\tgamma_1\right) \big|_{\su_4 = 0}, \qquad \iota_{\tsi^{s-1}}^*\left(\tcD_2^{\tT'}\tcD_{R+2}^{\tT'}\right) \big|_{\su_4 = 0} = \iota_{\tsi^{s-1}}^*\left(\tgamma_1^-\right) \big|_{\su_4 = 0}
$$
and that $\iota_{\tsi}^*\left(\tgamma_0\right) = 0$ for any $\tsi \neq \tsi_0$, $\iota_{\tsi}^*\left(\tgamma_1^-\right) = 0$ for any $\tsi \neq \tsi^{s-1}$. Statement (i) of the proposition then follows from Theorem \ref{thm:HyperCorr}.

Now we prove (iii), and (iv) will follow from a similar proof. Again we consider $\left(I_{\tcX}(\tq, z), \one_j\tcD_{R+2} \right)_{\tcX}$ as the non-equivariant limit of the $\tT'$-equivariant Poincar\'e pairing
$$
  \left(I_{\tcX}^{\tT'}(\tq, z), \tcD_{R+2}^{\tT'} \one_j \right)_{\tcX}^{\tT'},
$$
and we may consider the weight restriction $\su_4 = 0, \su_2 - f\su_1 = 0$. Since we are considering the twisted sector specified by $j$, the above pairing is equal to the pairing of $\one_j\tcD_{R+2}^{\tT'}$ with
$$
  \sum_{\substack{\tbeta \in \bK_{\eff}(\tcX) \\ \tv(\tbeta) = -j}} I_{\tbeta}(\tq, z).
$$
Again, since we focus on terms that depend on $\tq_{R-2}$, we consider the pairing of $\one_j\tcD_{R+2}^{\tT'}$ with each $I_{\tbeta}(\tq, z)$ for $\tbeta$ in
$$
  \{\tbeta \in \bK_{\eff}(\tcX) \setminus \bL_{\bQ} : \tv(\tbeta) = -j\}.
$$
Fix such a $\tbeta$. We may assume that the degree of $I_{\tbeta}(\tq, z)$ is 4, as the pairing is otherwise 0. Since $\inner{\tD_{R+1}, \tbeta} = -\inner{\tD_{R+2}, \tbeta} \neq 0$, $[z^{-2}]I_{\tbeta}(\tq, z)$ has form
$$
  c\tq^{\tbeta}\tcD_{R+1}^{\tT'}\one_j \qquad \text{or} \qquad c\tq^{\tbeta}\tcD_{R+2}^{\tT'}\one_j
$$
for some $c \in \bQ$, $c \neq 0$. Note that $j$ only belongs to the $\Box$ of the 4-cones $\iota(\sigma_0)$ and $\tsi_0$, and $\iota_{\iota(\sigma_0)}^*(\tcD_{R+1}^{\tT'}) = 0$, $\iota_{\iota(\sigma_0)}^*(\tcD_{R+2}^{\tT'}) = \su_4$. It then follows that
$$
\left(I_{\tbeta}(\tq, z), \tcD_{R+2}^{\tT'} \one_j\right)_{\tcX}^{\tT'} \bigg|_{\su_4 = 0} =  \frac{\iota_{\tsi_0}^*\left(I_{\tbeta}(\tq, z) \tcD_{R+2}^{\tT'} \one_j\right)}{|G_{\tsi_0}|e_{\tT'}(T_{\fp_{\tsi_0}}\tcX_j)} \bigg|_{\su_4 = 0} = \left(I_{\tbeta}(\tq, z), -\tgamma_{\tk} \right)_{\tcX}^{\tT'} \bigg|_{\su_4 = 0}.
$$
Here, the last equality follows from that
$$
  \iota_{\tsi_0}^*\left(\tcD_{R+2}^{\tT'} \one_j\right) \big|_{\su_4 = 0} = -\iota_{\tsi_0}^*\left(\tgamma_{\tk} \right) \big|_{\su_4 = 0}
$$
and that $\iota_{\iota(\sigma_0)}^*\left(\tgamma_{\tk}\right) = 0$. Statement (iii) of the proposition then follows from Theorem \ref{thm:HyperCorr}.
\end{proof}

\begin{example}\label{ex:C3PicardFuchs}
\rm{
Let $\cX = \bC^3$, $\cL$ be an outer brane, and $f=1$, as in Example \ref{ex:C3}. The non-equivariant $I$-function of $\tcX = \Tot(\cO_{\bP^2}(-2) \oplus \cO_{\bP^2}(-1))$ is
$$
  I_{\tcX}(\tq_1, z) = e^{\frac{\tH_1\log \tq_1}{z}} \left(1 + \sum_{d \in \bZ_{> 0}} \tq_1^d \frac{2\tH_1^2}{z^2} \frac{(-1)^d(2d-1)!}{d \cdot (d!)^2} \right)
  = 1 + \frac{\tH_1}{z} \log \tq_1 + \frac{\tH_1^2}{z^2}\left(\frac{1}{2} (\log \tq_1)^2 + 2W^{\cX, (\cL, 1)}_{1}(\tq_1) \right).
$$
The three coefficients give a basis of solutions to $\tcP$. Among them, the constant $1$ is the only solution to $\cP$. The solution $\log \tq_1$ in the $[z^{-1}]$ part is the open mirror map (Proposition \ref{prop:OpenMirrorMapSoln}). The power series part of the solution in the $[z^{-2}]$ part is a scalar multiple of the B-model disk function $W^{\cX, (\cL, 1)}_{1}$ (Proposition \ref{prop:DiskFnSoln}). 
In this example, the disk function $W^{\cX, (\cL^2, -2)}_{1}$ of the neighboring framed outer brane $(\cL^2, f_2 = -2)$ (as in Remark \ref{rem:MultiPhases}) happens to be equal to $-W^{\cX, (\cL, 1)}_{1}$, which then completely describes the $\tq_{R-2}$-dependence of the power series part of the solution. This phenomenon arises in certain other examples as well, including the case $\cX = \Tot(K_{\bP^2})$ as considered in \cite{Mayr01,LM01}.
}
\end{example}

%% file: cycles.tex

\section{Integral cycles and periods}\label{sect:Cycles}
In this section, we study the open/closed correspondence from the perspective of Hori-Vafa mirror families, contextualizing the results on hypergeometric functions and Picard-Fuchs systems in the previous sections. We establish a correspondence between integral relative 3-cycles on the Hori-Vafa mirror of $\cX$ and integral 4-cycles on that of $\tcX$ under which the periods are preserved (Theorem \ref{thm:CycleCorr}).

\subsection{Laurent polynomials}\label{sect:Laurent}
We start by introducing the Laurent polynomials that will be used to define the mirror families. Consider the integer coefficients $s_{ai} \in \bZ_{\ge 0}$ defined in Section \ref{sect:Basis}. For $i = 4, \dots, R$, define
$$
  s_i(q) := \prod_{a = 1}^{R-3} q_a^{s_{ai}}
$$
which is a monomial in $q_1, \dots, q_{R-3}$. For convenience we set $s_1(q) = s_2(q) = s_3(q) := 1$. We will write $s_i(\tq)$ as the same monomial but in $\tq_1, \dots, \tq_{R-3}$. We define the following Laurent polynomials parameterized by $q$ and $\tq$:
\begin{align*}
  &H(X, Y, q) := X^\fr Y^{-\fs} +  Y^{\fm} + 1 + \sum_{i = 4}^R s_i(q)X^{m_i}Y^{n_i} && \in \bC[X^{\pm 1}, Y^{\pm 1}],\\
  &\tH(X, Y, Z, \tq) := X^\fr Y^{-\fs} +  Y^{\fm} + 1 + \sum_{i = 4}^R s_i(\tq)X^{m_i}Y^{n_i} + \tq_{R-2}^{\fa}X^{-\fa}Y^{-\fb}Z + Z && \in \bC[X^{\pm 1}, Y^{\pm 1}, Z^{\pm 1}].
\end{align*}
Note that under $\tq_a = q_a$ for $a = 1, \dots, R-3$, we have
\begin{equation}\label{eqn:DefEqnRelation}
  \tH(X, Y, Z, \tq) = H(X,Y,q) + \tq_{R-2}^{\fa}X^{-\fa}Y^{-\fb}Z + Z,
\end{equation}
where the two additional terms correspond to the two additional vectors $\tb_{R+1} = (-\fa, -\fb, 1, 1)$, $\tb_{R+2} = (0, 0, 1, 1)$. Moreover, we use the open parameter $x$ to define the following restriction of $H(X, Y, q)$:
\begin{align*}
  H_0(Y_0, q, x) &:= H(X, Y, q) \big|_{X^{\fa}Y^{\fb} = -x^{\fa}} = H(e^{\pi\sqrt{-1}/\fa}xY_0^{-\fb}, Y_0^{\fa}, q) \\
   & =  e^{\pi\sqrt{-1}\fr/\fa}x^\fr Y_0^{-\fa\fs - \fb\fr} +  Y_0^{\fa\fm} + 1 + \sum_{i = 4}^R s_i(q)e^{\pi\sqrt{-1}m_i/\fa}x^{m_i}Y_0^{\fa n_i-\fb m_i} \qquad \in \bC[Y_0^{\pm 1}].
\end{align*}

We consider the following domains for the parameters $q, x, \tq$:
\begin{definition}\rm{
Let $\epsilon > 0$. Let $U_\epsilon \subset (\bC^*)^{R'-3} \times \bC^{R-R'}$ be the set of $(x_1, \dots, x_{R-3})$ such that
\begin{itemize}
  \item $x_1, \dots, x_{R'-3} \in \bC^*$, $x_{R'-2}, \dots, x_{R-3} \in \bC$;

  \item $|x_a| < \epsilon$ for all $a = 1, \dots, R-3$.
\end{itemize}
Moreover, Let $\tU_\epsilon \subset (\bC^*)^{R'-2} \times \bC^{R-R'}$ be the set of $(x_1, \dots, x_{R-2})$ such that
\begin{itemize}
  \item $(x_1, \dots, x_{R-3}) \in U_\epsilon$;

  \item $x_{R-2} \in \bC^*$, $|x_{R-2}| < \epsilon$.
\end{itemize}
}\end{definition}

Given $q \in U_\epsilon$, $\tq \in \tU_\epsilon$, the Newton polytopes\footnote{The \emph{Newton polytope} of a Laurent polynomial $F(X_1, \dots, X_r) \in \bC[X_1^{\pm1}, \dots, X_r^{\pm1}]$ is the convex hull in $\bR^r$ of lattice points $(a_1, \dots, a_r) \in \bZ^r$ such that the coefficient of $X_1^{a_1}\cdots X_r^{a_r}$ in $F$ is non-zero.} of $H(X, Y, q)$ and $\tH(X, Y, Z, \tq)$ are $\Delta \subset \bR^2$ and $\tDelta \subset \bR^3$ respectively. Moreover, for $\epsilon$ sufficiently small and $(q,x) \in \tU_\epsilon$, the Newton polytope of $H_0(Y_0, q, x)$ is $\Delta_0 \subset \bR$ (see \eqref{eqn:PolyDelta0}). In this paper, we will choose a sufficiently small radius $\epsilon$ for the parameter domains such that the Laurent polynomials satisfy the following regularity condition of Batyrev \cite{Batyrev93} (see also \cite{Dwork62,Dwork64,GKZ89}).

\begin{definition}\label{def:DeltaReg}
\rm{
Let 
$$
  F(X_1, \dots, X_r) = \sum_{a = (a_1, \dots, a_r) \in \bZ^r} c_aX_1^{a_1}\cdots X_r^{a_r} \qquad \in \bC[X_1^{\pm1}, \dots, X_r^{\pm1}]
$$
be a Laurent polynomial whose Newton polytope $\Delta_F \subset \bR^r$ is $r$-dimensional. $F$ is said to be \emph{$\Delta_F$-regular} if for every $0 < d \le r$ and every $d$-dimensional face $\Delta' \subseteq \Delta_F$, the functions
$$
  F^{\Delta'}:= \sum_{a \in \Delta' \cap \bZ^r} c_aX_1^{a_1}\cdots X_r^{a_r}, \qquad X_1\frac{\partial F^{\Delta'}}{\partial X_1}, \quad \dots, \quad  X_r\frac{\partial F^{\Delta'}}{\partial X_r}
$$
have no common zeroes in $(\bC^*)^r$.
}\end{definition}

\begin{assumption}\label{assump:Regular}\rm{
We fix a sufficiently small $\epsilon > 0$ such that for any $(q,x), \tq \in \tU_\epsilon = U_\epsilon \times \{c \in \bC^* : |c|<\epsilon\}$, the following hold:
\begin{itemize}
  \item $H(X, Y, q)$ is $\Delta$-regular;
  \item $\tH(X, Y, Z, \tq)$ is $\tDelta$-regular;
  \item $H_0(Y_0, q, x)$ has Newton polytope $\Delta_0$ and is $\Delta_0$-regular.
\end{itemize}
}\end{assumption}

The existence of such an $\epsilon$ is implied by \cite[Lemma 3.8]{Iritani09}, which studies a closely related non-degeneracy condition of Kouchnirenko \cite{Kouch76}. Assumption \ref{assump:Regular} will only be used in this section to ensure the smoothness of the mirror families, and will be used in full in Section \ref{sect:MHS} when we study variations of mixed Hodge structures.

\subsection{Hori-Vafa mirrors}\label{sect:HoriVafa}
The Hori-Vafa mirror family of $\cX$ is $(\cX^\vee_q, \Omega_q)$ defined over $q \in U_\epsilon$, where
$$
  \cX^\vee_q := \{(u,v, X, Y) \in \bC^2 \times (\bC^*)^2: uv = H(X, Y, q) \}
$$
is a non-compact Calabi-Yau 3-fold and
$$
  \Omega_q := \Res_{\cX^\vee_q} \frac{du \wedge dv \wedge \frac{dX}{X} \wedge \frac{dY}{Y}}{H(X, Y, q) - uv}
$$
is a holomorphic 3-form on $\cX^\vee_q$. Similarly, the Hori-Vafa mirror family of $\tcX$ is $(\tcX^\vee_{\tq}, \tOmega_{\tq})$ defined over $\tq \in \tU_\epsilon$, where
$$
  \tcX^\vee_{\tq} := \{(u,v, X, Y, Z) \in \bC^2 \times (\bC^*)^3: uv = \tH(X, Y, Z, \tq) \}
$$
is a non-compact Calabi-Yau 4-fold and
$$
  \tOmega_{\tq} := \Res_{\tcX^\vee_{\tq}} \frac{du \wedge dv \wedge \frac{dX}{X} \wedge \frac{dY}{Y} \wedge \frac{dZ}{Z}}{\tH(X, Y, Z, \tq) - uv}
$$
is a holomorphic 4-form on $\tcX^\vee_{\tq}$. Note that the projection $Z: \tcX^\vee_{\tq} \to \bC^*$ given by the $Z$-coordinate can be extended over $Z = 0$, and by relation \eqref{eqn:DefEqnRelation}, the fiber over $Z=0$ is isomorphic to $\cX^\vee_{(\tq_1, \dots, \tq_{R-3})}$.

Moreover, we consider a family $(\cY_{q,x}, \Omega_{q,x}^0)$ defined over $(q,x) \in \tU_\epsilon$, where
$$
  \cY_{q,x} := \{(u,v, Y_0) \in \bC^2 \times \bC^*: uv = H_0(Y_0, q, x) \}
$$
is a non-compact Calabi-Yau surface and
$$
  \Omega_{q,x}^0 := \Res_{\cY_{q,x}} \frac{du \wedge dv \wedge \frac{dY_0}{Y_0}}{H_0(Y_0, q, x) - uv}
$$
is a holomorphic 2-form on $\cY_{q,x}$. The family $(\cY_{q,x}, \Omega_{q,x}^0)$ is the pullback of the Hori-Vafa mirror family of $\tcX_0$ to the base $\tU_\epsilon$.

Consider the inclusion
\begin{equation}\label{eqn:Y0Include}
  \bC^* \to (\bC^*)^2, \qquad Y_0 \mapsto \left(e^{\pi\sqrt{-1}/\fa}xY_0^{-\fb}, Y_0^{\fa} \right)
\end{equation}
which is indeed injective since $\fa$, $\fb$ are coprime. Under \eqref{eqn:Y0Include}, for any $q$, $\cY_{q,x}$ can be identified with the family of hypersurfaces 
$$
  \cX^\vee_q \cap \{X^{\fa}Y^{\fb} = -x^{\fa}\}
$$
parameterized by $x$. Moreover, $\cY_{q,x} \times \bC^*$ can be identified with the hypersurface
$$
  \tcX^\vee_{\tq} \cap \{X^{\fa}Y^{\fb} = -\tq_{R-2}^{\fa}\}
$$
given by $((u,v,Y_0),Z) \mapsto (u,v,e^{\pi\sqrt{-1}/\fa}xY_0^{-\fb}, Y_0^{\fa}, Z)$ under $\tq_{a} = q_a$, $a = 1, \dots, R-3$ and $\tq_{R-2} = x$. Under the extension of $Z: \tcX^\vee_{\tq} \to \bC^*$ over $Z = 0$, this recovers the embedding of $\cY_{q,x}$ in $\cX^\vee_q$.

By Assumption \ref{assump:Regular}, in particular the regularity condition applied to the top-dimensional faces of the polytopes themselves, the three families $\cX^\vee_q$, $\tcX^\vee_{\tq}$, and $\cY_{q,x}$ are smooth over $(q,x), \tq \in \tU_\epsilon$.

\subsection{Open/closed correspondence of integral cycles and periods}
It is a well-known fact that period integrals of the holomorphic form over middle-dimensional integral cycles on the Hori-Vafa mirror generate the solution space to the Picard-Fuchs system.

\begin{theorem}[Cf. \cite{Hosono06,KM10,CLT13}]\label{thm:PeriodSoln}
The period integrals
$$
  \int_{\Gamma}  \Omega_q, \qquad \Gamma \in H_3(\cX^\vee_q; \bZ)
$$
give a $\bC$-basis of the solution space to the Picard-Fuchs system $\cP$ as $\Gamma$ ranges through a basis. The period integrals
$$
  \int_{\tGamma}  \tOmega_{\tq}, \qquad \tGamma \in H_4(\tcX^\vee_{\tq}; \bZ)
$$
give a $\bC$-basis of the solution space to the Picard-Fuchs system $\tcP$ as $\tGamma$ ranges through a basis.
\end{theorem}

\begin{remark}\label{rem:SmallPFSystem}\rm{
Similarly, as $\Gamma_0$ ranges through a basis of $H_2(\cY_{q,x}; \bZ)$, the period integrals
$$
  \int_{\Gamma_0}  \Omega_{q,x}^0
$$
give a $\bC$-basis of the solution space of a Picard-Fuchs system defined by the polytope $\Delta_0$. The dimension of the solution space is $\Vol(\Delta_0)$, which is the difference between those of $\tcP$ and $\cP$ (see Theorem \ref{thm:PFSolnDimension} and \eqref{eqn:PolyVolSum}). In particular, Equation \eqref{eqn:PolyVolSum} translates into
\begin{equation}\label{eqn:RankSum}
  \rank(H_4(\tcX^\vee_{\tq}; \bZ)) = \rank(H_3(\cX^\vee_q; \bZ)) + \rank(H_2(\cY_{q,x}; \bZ)).
\end{equation}
}\end{remark}

Our main result of this section is that periods of $\tOmega_{\tq}$ over 4-cycles on $\tcX^\vee_{\tq}$, and thus solutions to the extended system $\tcP$, can be recovered from period of $\Omega_q$ over \emph{relative} 3-cycles on $\cX^\vee_q$ with boundary on the hypersurface $\cY_{q,x}$. This gives the open/closed correspondence on the level of cycles and periods.

\begin{theorem}\label{thm:CycleCorr}
For $(q,x) = \tq \in \tU_\epsilon$, there is an injective map
$$
  \iota: H_3(\cX^\vee_q, \cY_{q,x}; \bZ) \to H_4(\tcX^\vee_{\tq}; \bZ)
$$
such that for any $\Gamma \in H_3(\cX^\vee_q, \cY_{q,x}; \bZ)$, we have
$$
  \int_{\Gamma}  \Omega_q = \frac{1}{2\pi\sqrt{-1}} \int_{\iota(\Gamma)}  \tOmega_{\tq}.
$$
Moreover, $\iota$ is an isomorphism over $\bQ$.
\end{theorem}


\subsection{Conic and circle fibrations}\label{sect:Conic}
Before proving Theorem \ref{thm:CycleCorr}, we first review the structure of Hori-Vafa mirrors as conic fibrations over algebraic tori, following \cite[Section 4.2]{CLT13}; see also \cite[Section 5.1]{DK11}. We detail the construction for $\cX^\vee_q$ (see also \cite[Section 4.4]{FLZ20} for this case) and the constructions for $\tcX^\vee_{\tq}$, $\cY_{q,x}$ are similar.

Consider the following families of affine hypersurfaces in algebraic tori of different dimensions, defined over $(q,x), \tq \in \tU_\epsilon$:
\begin{equation}\label{eqn:MirrorCurveDef}
\begin{aligned}
  & C_q := \{(X,Y) \in (\bC^*)^2 : H(X, Y, q) = 0\},\\
  & S_{\tq} := \{(X,Y,Z) \in (\bC^*)^3 : \tH(X, Y, Z, \tq) = 0 \},\\
  & P_{q,x} := \{Y_0 \in \bC^* : H_0(Y_0, q, x) = 0 \}.
\end{aligned}
\end{equation}
By Assumption \ref{assump:Regular}, $C_q$ (resp. $S_{\tq}$) is a family of smooth algebraic curves (resp. surfaces) and $P_{q,x}$ is a family of $\Vol(\Delta_0)$ distinct points. By \eqref{eqn:DefEqnRelation}, the projection $Z: S_{\tq} \to \bC^*$ can be extended over $Z = 0$ and the fiber over $Z=0$ is isomorphic to $C_{(q_1, \dots, q_{R-3})}$. Under \eqref{eqn:Y0Include}, for any $q$, $P_{q,x}$ can be identified with the family of point sets
$$
  C_q \cap \{X^{\fa}Y^{\fb} = -x^{\fa}\}.
$$
parameterized by $x$. Moreover, $P_{q,x} \times \bC^*$ can be identified with the hypersurface
$$
  S_{\tq} \cap \{X^{\fa}Y^{\fb} = -\tq_{R-2}^{\fa}\}.
$$
Under the extension of $Z: S_{\tq} \to \bC^*$ over $Z = 0$, this recovers the embedding of $P_{q,x}$ in $C_q$.

Consider the following Hamiltonian $U(1)$-action
$$
  e^{\theta\sqrt{-1}} \cdot (u, v, X, Y) := (e^{\theta\sqrt{-1}}u, e^{-\theta\sqrt{-1}}v, X, Y)
$$
on $\cX^\vee_q$ whose moment map is given by
$$
  \mu: \cX^\vee_q \to \bR, \qquad (u, v, X, Y) \mapsto \frac{1}{2} \left(|u|^2-|v|^2\right).
$$
Now consider the conic fibration
$$
  \cX^\vee_q \to (\bC^*)^2, \qquad (u, v, X, Y) \mapsto (X,Y)
$$
whose discriminant locus is the curve $C_q$ \eqref{eqn:MirrorCurveDef}. Restricting the above to the level set
$$
  \mu^{-1}(0) = \{ (u, v, X, Y) \in \cX^\vee_q : |u| = |v|\},
$$
we obtain a circle fibration
$$
  \pi: \mu^{-1}(0) \to (\bC^*)^2
$$
where the fiber degenerates to a point precisely over $(X,Y) \in C_q$. The map $\gamma \mapsto \pi^{-1}(\gamma)$ induces an isomorphism
$$
  \alpha: H_2((\bC^*)^2, C_q; \bZ) \xrightarrow{\sim} H_3(\cX^\vee_q; \bZ).
$$

By identical constructions, the conic fibrations
$$
  \tcX^\vee_{\tq} \to (\bC^*)^3, \qquad \cY_{q,x} \to (\bC^*)
$$
restrict to circle fibrations
$$
  \tpi: \mu^{-1}(0) \subset \tcX^\vee_{\tq} \to (\bC^*)^3, \qquad \pi_0: \mu^{-1}(0) \subset \cY_{q,x} \to (\bC^*)
$$
whose discriminant loci are $S_{\tq}$, $P_{q,x}$ respectively. There are induced isomorphisms
$$
  \talpha: H_3((\bC^*)^3, S_{\tq}; \bZ) \xrightarrow{\sim} H_4(\tcX^\vee_{\tq}; \bZ), \qquad 
  \alpha_0: H_1(\bC^*, P_{q,x}; \bZ) \xrightarrow{\sim} H_2(\cY_{q,x}; \bZ).
$$

The standard holomorphic forms
\begin{equation}\label{eqn:FormsOnTori}
  \omega:= \frac{dX}{X} \wedge \frac{dY}{Y}, \qquad \tomega:= \frac{dX}{X} \wedge \frac{dY}{Y} \wedge \frac{dZ}{Z}, \qquad \omega^0 := \frac{dY_0}{Y_0}
\end{equation}
on $(\bC^*)^2$, $(\bC^*)^3$, $\bC^*$ respectively vanish on the hypersurfaces $C_q$, $S_{\tq}$, $P_{q,x}$ and thus represent relative cohomology classes in
$$
  H^2((\bC^*)^2, C_q; \bC), \qquad H^3((\bC^*)^3, S_{\tq}; \bC), \qquad H^1(\bC^*, P_{q,x}; \bC).
$$
The relative periods of these relative forms are known to match periods of the holomorphic forms on the Hori-Vafa mirrors via the isomorphisms $\alpha$, $\talpha$, $\alpha_0$ on homology.

\begin{theorem}[\cite{DK11,CLT13}]\label{thm:RelPeriod}
We have
\begin{align*}
  & \int_{\gamma} \omega = \frac{1}{2\pi\sqrt{-1}} \int_{\alpha(\gamma)} \Omega_q && \text{for all } \gamma \in H_2((\bC^*)^2, C_q; \bZ);\\
  & \int_{\tgamma} \tomega = \frac{1}{2\pi\sqrt{-1}} \int_{\talpha(\tgamma)} \tOmega_{\tq} && \text{for all } \tgamma \in H_3((\bC^*)^3, S_{\tq}; \bZ);\\
  & \int_{\gamma_0} \omega^0 = \frac{1}{2\pi\sqrt{-1}} \int_{\alpha_0(\gamma_0)} \Omega^0_{q,x} && \text{for all } \gamma_0 \in H_1(\bC^*, P_{q,x}; \bZ).
\end{align*}
\end{theorem}

As pointed out in the proof of \cite[Proposition 14]{CLT13}, the results in \cite{Batyrev93,Stienstra98,KM10} (cf. Section \ref{sect:MHS}) imply that the relative periods of $\omega$, $\tomega$ also generate the solution spaces to the Picard-Fuchs systems $\cP$, $\tcP$ respectively (cf. Theorem \ref{thm:PeriodSoln}). A similar remark holds for relative periods of $\omega^0$ (see Remark \ref{rem:SmallPFSystem}).

\subsection{Decomposing relative 3-cycles}\label{sect:Decompose3Cycles}
Now we start to prove Theorem \ref{thm:CycleCorr}. Fix $(q,x) = \tq \in \tU_\epsilon$. Consider the long exact sequence of integral homology associated to the pair $(\cX^\vee_q, \cY_{q,x})$:
\begin{equation}\label{eqn:RelHomologyLES}
  \cdots \to H_3(\cY_{q,x}; \bZ) \to H_3(\cX^\vee_q; \bZ) \to H_3(\cX^\vee_q, \cY_{q,x}; \bZ) \to H_2(\cY_{q,x}; \bZ) \to H_2(\cX^\vee_q; \bZ) \to \cdots.
\end{equation}
Our goal of this subsection is to show the following.

\begin{proposition}\label{prop:RelHomologySES}
The following sequence induced by \eqref{eqn:RelHomologyLES}
\begin{equation}\label{eqn:RelHomologySES}
  \xymatrix{
    0 \ar[r] & H_3(\cX^\vee_q; \bZ) \ar[r] & H_3(\cX^\vee_q, \cY_{q,x}; \bZ) \ar[r] & H_2(\cY_{q,x}; \bZ) \ar[r] & 0
  }
\end{equation}
is a \emph{split} short exact sequence.
\end{proposition}

We prove Proposition \ref{prop:RelHomologySES} in a series of lemmas below.

\begin{lemma}\label{lem:CycleInjection}
We have
$$
  H_3(\cY_{q,x}; \bZ) = 0.
$$
\end{lemma}

\begin{proof}
Consider the conic fibration $\pi: \cY_{q,x} \to \bC^*$. We write
$$
  \bC^* = A_0 \cup B_0,
$$ 
where $B_0$ is the union of disjoint open disks, one around each point in $P_{q,x}$, and $A_0$ is the complement in $\bC^*$ of disjoint closed disks, one around each point in $P_{q,x}$. Note that $B_0$ deformation retracts to $P_{q,x}$ while $A_0$ is homotopy equivalent to $\bC^* \setminus P_{q,x}$. Their intersection $A_0 \cap B_0$ deformation retracts to a disjoint union of circles, one around each point in $P_{q,x}$.

Then $\cY_{q,x}$ is covered by the sets $A = \pi^{-1}(A_0)$ and $B = \pi^{-1}(B_0)$, and Mayer-Vietoris gives
$$
  \cdots \to H_3(A;\bZ) \oplus H_3(B;\bZ) \to H_3(\cY_{q,x}; \bZ) \to H_2(A \cap B; \bZ) \to H_2(A;\bZ) \oplus H_2(B;\bZ) \to \cdots.
$$
The fiber of $\pi$ over each point in $P_{q,x}$ is $\{uv = 0\} = \bC \cup \bC \subset \bC^2$. Thus $B$ deformation retracts to $P_{q,x}$, which implies that $H_3(B;\bZ) = H_2(B; \bZ) = 0$. Moreover, $A$ is a trivial $\bC^*$-bundle over $A_0$ with a trivialization given by
$$
  A \cong \bC^* \times A_0, \qquad \left(u, v = \frac{H_0(Y_0, q, x)}{u}, Y_0\right) \xleftrightarrow{} (u,Y_0),
$$
and $A \cap B$ is a subbundle over $A_0 \cap B_0$. Thus $H_3(A; \bZ) = 0$ and moreover, the map $H_2(A \cap B) \to H_2(A)$ induced by inclusion is injective by the injectivity of the map $H_1(A_0 \cap B_0) \to H_1(A_0)$ induced by inclusion.
\end{proof}

\begin{lemma}\label{lem:MCH1Surjective}
The map
$$
      H_1(C_q; \bZ) \to H_1((\bC^*)^2; \bZ)
$$
induced by inclusion is surjective.
\end{lemma}

We note that \cite[Section 4.4]{FLZ20} studies the above map and shows that it has finite cokernel. We confirm that the cokernel is indeed trivial.

\begin{proof}
Denote the above map by $\iota_*$. We use the analysis of the family of curves $C_q$ in \cite[Sections 4.4, 5.8]{FLZ20}. First we consider the case $n_1 = \fs = 0$. Given a 1-cycle $\gamma \in H_1(C_q; \bZ)$, we will consider the pairing of $\iota_*(\gamma)$ with the basis of 1-forms $\frac{dX}{2\pi\sqrt{-1}X}, \frac{dY}{2\pi\sqrt{-1}Y} \in H^1((\bC^*)^2; \bZ)$. For any flag $(\tau,\sigma) \in F(\Sigma)$, \cite{FLZ20} constructs an element $\gamma^{(\tau,\sigma)} \in H_1(C_q; \bZ)$ by considering the degeneration of $C_q$ in the limit $q \to 0$ and taking the vanishing cycles on the irreducible component corresponding to $\sigma$ around punctures corresponding to $\tau$. \cite{FLZ20} further computes the pairing of $\iota_*(\gamma^{(\tau,\sigma)})$ with $\frac{dX}{2\pi\sqrt{-1}X}, \frac{dY}{2\pi\sqrt{-1}Y}$. In our present case $\fs = 0$, a direct computation shows that for the cycles $\gamma^{(\tau_0,\sigma_0)}$, $\gamma^{(\tau_2,\sigma_0)}$ corresponding to flags $(\tau_0,\sigma_0)$, $(\tau_2, \sigma_0)$ respectively, we have
\begin{align*}
  &\inner{\iota_*(\gamma^{(\tau_0,\sigma_0)}), \tfrac{dX}{2\pi\sqrt{-1}X}} = 1, \qquad \inner{\iota_*(\gamma^{(\tau_0,\sigma_0)}), \tfrac{dY}{2\pi\sqrt{-1}Y}} = 0, \\
  &\inner{\iota_*(\gamma^{(\tau_2,\sigma_0)}), \tfrac{dX}{2\pi\sqrt{-1}X}} = 0, \qquad \inner{\iota_*(\gamma^{(\tau_2,\sigma_0)}), \tfrac{dY}{2\pi\sqrt{-1}Y}} = 1.
\end{align*}
This implies that $\iota_*(\gamma^{(\tau_0,\sigma_0)})$, $\iota_*(\gamma^{(\tau_2,\sigma_0)})$ form a basis of $H_1((\bC^*)^2; \bZ)$ and the surjectivity of $\iota_*$ follows.

Now we consider the case $n_1 = - \fs <0$. Note that the vector
$$
  (1,0,1) = \frac{1}{\fr}b_1 + \frac{\fs}{\fr\fm}b_2 + \left(1 - \frac{1}{\fr} - \frac{\fs}{\fr\fm}  \right)b_3
$$
always lies inside the cone $\sigma_0$ (recall that $\fs \le \fr-1$). Let $\cX'$ be the crepant partial resolution of $\cX$ whose fan is obtained from $\Sigma$ by taking the star subdivision along a new ray generated by $(1,0,1)$. By \cite[Section 5.1]{Yu20}, the family of curves $\{C_q\}_{q \in U_{\epsilon}}$ and the analogous family defined by the toric Calabi-Yau 3-orbifold $\cX'$ fit into a global family of affine curves in $(\bC^*)^2$ over a base $\cM$. Moreover, $H_1(C_q;\bZ)$ forms a local system of lattices over $\cM$ and maps to the trivial system $H_2((\bC^*)^2;\bZ)$ via inclusion. Now, note that $\cX'$ satisfies the condition $\fs = 0$ of the previous case and $\iota_*$ is surjective for the family of curves defined by $\cX'$. It follows that $\iota_*$ is also surjective for $\{C_q\}_{q \in U_{\epsilon}}$ since surjectivity is preserved under parallel transport over $\cM$.
\end{proof}

\begin{lemma}\label{lem:MCRelH1Surjective}
The map
$$
      H_1(C_q, P_{q,x}; \bZ) \to H_1((\bC^*)^2, P_{q,x}; \bZ)
$$
induced by inclusion is surjective.
\end{lemma}

\begin{proof}
The lemma follows from Lemma \ref{lem:MCH1Surjective} and the Five Lemma applied to
$$
\xymatrix{
  0 \ar[r] \ar[d] & H_1(C_q; \bZ) \ar[r] \ar[d] & H_1(C_q, P_{q,x}; \bZ) \ar[r] \ar[d] &  H_0(P_{q,x}; \bZ) \ar[r] \ar[d] & H_0(C_q; \bZ) \ar[d]\\
  0 \ar[r] & H_1((\bC^*)^2; \bZ) \ar[r]  & H_1((\bC^*)^2, P_{q,x}; \bZ) \ar[r] &  H_0(P_{q,x}; \bZ) \ar[r]  & H_0(\bC^*; \bZ).
}
$$
\end{proof}

\begin{lemma}\label{lem:SectionExists}
There is a section
$$
  s: H_2(\cY_{q,x}; \bZ) \to H_3(\cX^\vee_q, \cY_{q,x}; \bZ) 
$$
of the boundary map $\partial: H_3(\cX^\vee_q, \cY_{q,x}; \bZ) \to H_2(\cY_{q,x}; \bZ)$, i.e. the composition $\partial \circ s$ is the identity on $H_2(\cY_{q,x}; \bZ)$. In particular, $\partial$ is surjective.
\end{lemma}

\begin{proof}
Recall from Section \ref{sect:Conic} that there is an isomorphism $\alpha_0: H_1(\bC^*, P_{q,x}; \bZ) \xrightarrow{\sim} H_2(\cY_{q,x}; \bZ)$ induced by the circle fibration $\pi_0$. Consider the inclusion \eqref{eqn:Y0Include} over which the inclusion $\cY_{q,x} \to \cX^\vee_q$ is the induced inclusion of conic bundles. Note that the induced map $H_1(\bC^*, P_{q,x}; \bZ) \to H_1((\bC^*)^2, P_{q,x}; \bZ)$ is injective. Let
$$
  \gamma_0 \in H_1(\bC^*, P_{q,x}; \bZ) \subseteq H_1((\bC^*)^2, P_{q,x}; \bZ).
$$
By Lemma \ref{lem:MCRelH1Surjective}, $\gamma_0$ is homologous in $(\bC^*)^2$ to a relative 1-cycle $\gamma_0'$ contained in $C_q$ relative to $P_{q,x}$. Now let $\gamma$ be a 2-chain in $(\bC^*)^2$ such that $\partial \gamma = \gamma_0 - \gamma_0'$. Then the circle fibration map
$$
  \pi: \{|u| = |v|: (u,v,X,Y) \in \cX^\vee_q\} \to (\bC^*)^2
$$
lifts $\gamma$ to a relative 3-chain $\pi^{-1}(\gamma)$ in $\cX^\vee_q$ whose boundary is $\partial (\pi^{-1}(\gamma)) = \pi_0^{-1}(\gamma_0)$. (Note that $\gamma_0'$ is contained in the discriminant locus $C_q$ of $\pi$.) We may then define
$$
  s(\pi_0^{-1}(\gamma_0)) := \pi^{-1}(\gamma).
$$
See Figure \ref{fig:SectionExists} for an illustration of the construction. 
\end{proof}

\begin{figure}[h]
\begin{center}
    \begin{tikzpicture}[scale=1]
        
        \node at (-3.5,0) {$\phantom{\cdot}$};

        \coordinate (0) at (-1, -1);
        \coordinate (1) at (-1, 1);        
        
        \draw[dashed] (-1, -2) -- (-1,2);
        \draw (0) -- (1);
        \draw[dashed] (-2, -1.1) -- (-1,-1);
        \draw[dashed] (-2, 1.1) -- (-1,1);
        \draw (0) .. controls (2, -0.5) and (2, 0.5) .. (1);

        \node[left] at (-1,0) {$\gamma_0$};
        \node at (1.6, 0) {$\gamma_0'$}; 
        \node at (0, 0) {$\gamma$};

        \node[left] at (-2, -1.1) {$C_q$};
        \node[below] at (-1,-2) {$\bC^*$};

        \node at (3, 0) {$\longleftarrow$};
        \node at (3, 0.3) {$\pi$};

        \coordinate (00) at (5.25,-1.5);
        \coordinate (01) at (6.75, 1.5);

        \draw[dashed] (5, -2) -- (7, 2);
        \draw (00) .. controls (10.7, -0.75) and (12.5, 0.55) .. (01);

        \draw (00) .. controls (5.5, 0.75) and (6, 1.25) .. (01);
        \draw[dashed] (00) .. controls (6, -1.25) and (6.5, -0.75) .. (01);
        \draw (00) .. controls (4.5, -0.5) and (5.25, 1.25) .. (01);
        \draw (01) .. controls (7.5, 0.5) and (6.75, -1.25) .. (00);

        \coordinate (1left) at (5.05,0);
        \coordinate (1right) at (6.95,0);

        \draw  (1left) .. controls (5.4, 0.9) and (6.6, 0.9) .. (1right);
        \draw[dashed]  (1left) .. controls (5.4, -0.9) and (6.6, -0.9) .. (1right);

        \coordinate (2down) at (8, -1);
        \coordinate (2up) at (9, 1);


        \draw (2down) .. controls (8.33, 0.5) and (8.5, 0.833) .. (2up);
        \draw[dashed] (2down) .. controls (8.5, -0.833) and (8.67, -0.5) .. (2up);
        \draw (2down) .. controls (7.5, -0.33) and (8, 0.833) .. (2up);
        \draw (2up) .. controls (9.5, 0.33) and (9, -0.833) .. (2down);

        \coordinate (2left) at (7.86,0);
        \coordinate (2right) at (9.14,0);

        \draw  (2left) .. controls (8.1, 0.6) and (8.9, 0.6) .. (2right);
        \draw[dashed]  (2left) .. controls (8.1, -0.6) and (8.9, -0.6) .. (2right);
        
        \node[left] at (1left) {$\pi_0^{-1}(\gamma_0)$};
        \node at (8.5, 1.7) {$\pi^{-1}(\gamma)$};
    \end{tikzpicture}
\end{center}

    \caption{Construction of the section $s$ in Lemma \ref{lem:SectionExists}.}
    \label{fig:SectionExists}
\end{figure}
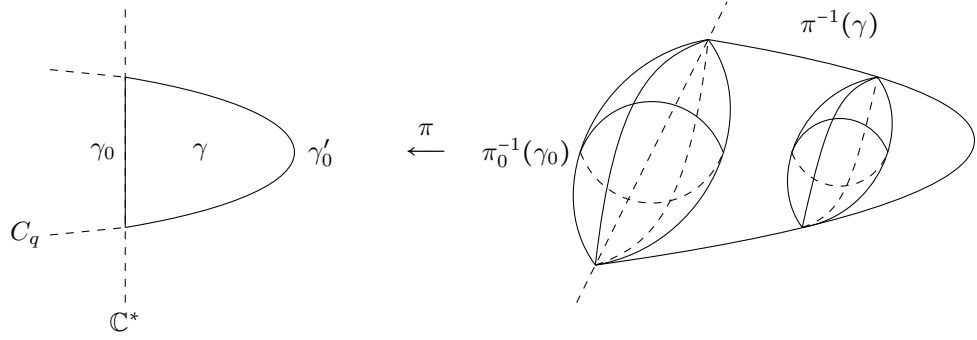

\begin{proof}[Proof of Proposition \ref{prop:RelHomologySES}]
Lemma \ref{lem:CycleInjection} implies that \eqref{eqn:RelHomologySES} is exact at $H_3(\cX^\vee_q; \bZ)$. Lemma \ref{lem:SectionExists} implies that \eqref{eqn:RelHomologySES} is exact at $H_2(\cY_{q,x}; \bZ)$ and is split.
\end{proof}

\subsection{Constructing 4-cycles and matching periods}\label{sect:Construct4Cycles}
In this subsection, we prove Theorem \ref{thm:CycleCorr}. In view of Proposition \ref{prop:RelHomologySES}, we will construct the desired map $\iota: H_3(\cX^\vee_q, \cY_{q,x}; \bZ) \to H_4(\tcX^\vee_{\tq}; \bZ)$ on the direct summands $H_3(\cX^\vee_q; \bZ)$ and $s(H_2(\cY_{q,x}; \bZ))$. Again we fix $(q,x) = \tq \in \tU_\epsilon$.

\begin{lemma}\label{lem:Closed4Cycles}
There is an injective map
$$
  \iota_1: H_3(\cX^\vee_q; \bZ) \to H_4(\tcX^\vee_{\tq}; \bZ)
$$
such that for any $\Gamma \in H_3(\cX^\vee_q; \bZ)$, we have
\begin{equation}\label{eqn:ClosedPeriodCorr}
  \int_{\Gamma} \Omega_q = \frac{1}{2\pi\sqrt{-1}} \int_{\iota_1(\Gamma)}  \tOmega_{\tq}.
\end{equation}
\end{lemma}

\begin{proof}
Recall from Section \ref{sect:HoriVafa} that the projection $Z: \tcX^\vee_{\tq} \to \bC^*$ given by the $Z$-coordinate can be extended over $Z=0$ and the fiber over $Z=0$ is equal to $\cX^\vee_q$. Take a sufficiently small $\epsilon' > 0$ such that over $\{|Z| \le \epsilon'\}$, the projection is a trivial $\cX^\vee_q$-bundle. Then given $\Gamma \in H_3(\cX^\vee_q; \bZ)$, we may extend $\Gamma$ to a locally constant section, and define
$$
  \iota_1(\Gamma) := \Gamma \times \{|Z| = \epsilon'\} \in H_4(\tcX^\vee_{\tq}; \bZ).
$$
Moreover, we confirm \eqref{eqn:ClosedPeriodCorr} by integrating along the $Z$-direction:
$$
  \int_{\Gamma \times \{|Z| = \epsilon'\}} \tOmega_{\tq}  = 2\pi\sqrt{-1} \int_{\Gamma} \Res_{Z=0} \tOmega_{\tq} = 2\pi\sqrt{-1} \int_{\Gamma} \Omega_q.
$$
The injectivity of $\iota_1$ follows from that the corresponding periods \eqref{eqn:ClosedPeriodCorr} are linearly independent as $\Gamma$ ranges through a basis of $H_3(\cX^\vee_q; \bZ)$ (Theorem \ref{thm:PeriodSoln}).
\end{proof}

\begin{lemma}\label{lem:StokesRel3Cycles}
The section $s: H_2(\cY_{q,x}; \bZ) \to H_3(\cX^\vee_q, \cY_{q,x}; \bZ)$ constructed in (the proof of) Lemma \ref{lem:SectionExists} satisfies that for any $\Gamma_0 \in H_2(\cY_{q,x}; \bZ)$, we have
$$
  x\frac{\partial}{\partial x} \int_{s(\Gamma_0)} \Omega_q = \fa \int_{\Gamma_0} \Omega_{q,x}^0.
$$
\end{lemma}

\begin{proof}
Introduce a new coordinate
$$
  \bar{X} := -X^{\fa}Y^{\fb}.
$$
In view of $Y = Y_0^{\fa}$ from \eqref{eqn:Y0Include}, we have
$$
  \frac{dX}{X} \wedge \frac{dY}{Y} = \frac{d\bar{X}}{\bar{X}} \wedge \frac{dY_0}{Y_0}.
$$
The hypersurface $\cY_{q,x}$ in $\cX^\vee_q$ is now defined by the constant equation $\bar{X} = x^{\fa}$. Now we fix a branch of $\log \bar{X}$ defined on the complement of a ray avoiding $x^{\fa}$. Then for any $\Gamma_0 \in H_2(\cY_{q,x}; \bZ)$, the construction of $s(\Gamma_0)$ in the proof of Lemma \ref{lem:SectionExists} may be performed within this branch. (Note that Lemmas \ref{lem:MCH1Surjective} and \ref{lem:MCRelH1Surjective} used in the construction are valid under the pullback $\bC \to \bC^*$ to the universal cover in the $\bar{X}$-coordinate.) Since the boundary $\partial(s(\Gamma_0)) = \Gamma_0$ is contained in $\{\bar{X} = x^{\fa}\}$, or $\{\log \bar{X} = \fa \log x\}$, we have
$$
  \frac{\partial}{\partial \log x} \int_{s(\Gamma_0)} \Omega_q =  \fa \frac{\partial}{\partial \log \bar{X}} \int_{s(\Gamma_0)} \Omega_q = \fa \int_{\Gamma_0} \Res_{\cX^\vee_q \cap \{\bar{X} = x^{\fa}\}} \frac{du \wedge dv \wedge \frac{dY_0}{Y_0}}{H(X, Y, q) \big|_{\bar{X} = x^{\fa}} - uv}  = \fa \int_{\Gamma_0} \Omega_{q,x}^0.
$$
\end{proof}

\begin{lemma}\label{lem:Open4Cycles}
There is an injective map
$$
  \iota_2: s(H_2(\cY_{q,x}; \bZ)) \to H_4(\tcX^\vee_{\tq}; \bZ),
$$
such that for any $\Gamma_0 \in H_2(\cY_{q,x}; \bZ)$, we have
\begin{equation}\label{eqn:OpenPeriodCorr}
  \int_{s(\Gamma_0)} \Omega_q = \frac{1}{2\pi\sqrt{-1}} \int_{\iota_2(s(\Gamma_0))}  \tOmega_{\tq}.
\end{equation}
\end{lemma}

\begin{proof}
Let $\Gamma_0 \in H_2(\cY_{q,x}; \bZ)$. As in the construction of $s(\Gamma_0)$ in the proof of Lemma \ref{lem:SectionExists}, we take $\gamma_0 \in H_1(\bC^*, P_{q,x}; \bZ)$ inducing $\Gamma_0$ under the circle fibration $\pi_0$, $\gamma_0' \in H_1(C_q, P_{q,x}; \bZ)$ homologous to $\gamma_0$ in $(\bC^*)^2$, and 2-chain $\gamma$ in $(\bC^*)^2$ with $\partial \gamma = \gamma_0 - \gamma_0'$. Without loss of generality, we assume that $\gamma_0$ is a smooth path with endpoints in $P_{q,x}$. Since $s(\Gamma_0) = \pi^{-1}(\gamma)$, we have
$$
  \int_{\gamma} \omega = \frac{1}{2\pi\sqrt{-1}} \int_{s(\Gamma_0)} \Omega_q.
$$
See Theorem \ref{thm:RelPeriod} and its proof in \cite[Lemma 13]{CLT13}.

Now, similar to the proof of Lemma \ref{lem:Closed4Cycles}, we view the pair $((\bC^*)^2, C_q)$ in the fibration $Z: ((\bC^*)^3, S_{\tq}) \to \bC^*$ as the extended fiber over $Z = 0$, and deform $\gamma$ to a 3-chain $\tgamma_1 \cong \gamma \times \{|Z| = \epsilon'\}$ over $\{|Z| = \epsilon'\}$ for some sufficiently small $\epsilon' > 0$. This can be done such that the deformation of $\gamma_0' \subset \partial \gamma$ is contained in $S_{\tq}$ and the deformation of $\gamma_0 \subset \partial \gamma$ has the following description: Let $(X(t), Y(t))$, $t \in [0,1]$ be a smooth parametrization of $\gamma_0$ and choose a smooth function $r:[0,1] \to [0,1]$ such that $r(t) = 0$ only at $t = 0,1$. Then the deformation of $\gamma_0$ is parameterized by an angle $\theta \in [0,2\pi]$ as
$$
    X(t, \theta) = (1-r(t)\epsilon'e^{\theta\sqrt{-1}})^{1/\fa}X(t), \qquad Y(t, \theta) = Y(t), \qquad Z(t, \theta) = Z_0(t, \theta) = \epsilon'e^{-\theta\sqrt{-1}}.
$$
Here in the definition of $X(t, \theta)$ we fix the branch of $\fa$-th roots near $1$ such that $1^{1/\fa} = 1$. See Figure \ref{fig:Open4CyclesCorrect} for an illustration of the 3-chain $\tgamma_1$. Note that $Y(t,\theta)$ is independent of $\theta$, and at $t = 0,1$ the point $(X(t,\theta), Y(t, \theta))$ is an original endpoint of $\gamma_0$ contained in $P_{q,x}$. Moreover, we may choose $r(t)$ such that $H(X(t,\theta), Y(t, \theta), q) \neq 0$ whenever $t \neq 0,1$.

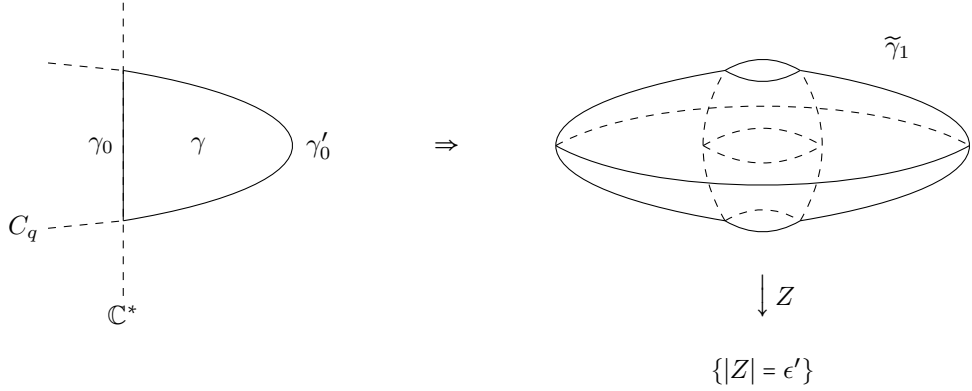
\begin{figure}[h]
\begin{center}
    \begin{tikzpicture}[scale=1]
        
        \node at (-3.5,0) {$\phantom{\cdot}$};

        \coordinate (0) at (-1, -1);
        \coordinate (1) at (-1, 1);        
        
        \draw[dashed] (-1, -2) -- (-1,2);
        \draw (0) -- (1);
        \draw[dashed] (-2, -1.1) -- (-1,-1);
        \draw[dashed] (-2, 1.1) -- (-1,1);
        \draw (0) .. controls (2, -0.5) and (2, 0.5) .. (1);

        \node[left] at (-1,0) {$\gamma_0$};
        \node at (1.6, 0) {$\gamma_0'$}; 
        \node at (0, 0) {$\gamma$};

        \node[left] at (-2, -1.1) {$C_q$};
        \node[below] at (-1,-2) {$\bC^*$};

        \node at (3.3, 0) {$\Rightarrow$};

        \coordinate (10) at (7, -1);
        \coordinate (11) at (7, 1);        
        
        \draw[dashed] (10) to[bend left] (11);
        \draw (10) .. controls (4, -0.5) and (4, 0.5) .. (11);

        \coordinate (20) at (8, -1);
        \coordinate (21) at (8, 1);        
        
        \draw[dashed] (20) to[bend right] (21);
        \draw (20) .. controls (11, -0.5) and (11, 0.5) .. (21);

        \draw (11) to[bend left] (21);
        \draw (11) to[bend right] (21);
        \draw[dashed] (10) to[bend left] (20);
        \draw (10) to[bend right] (20);
        \draw[dashed] (6.71,0) to[bend left] (8.29,0);
        \draw[dashed] (6.71,0) to[bend right] (8.29,0);

        \coordinate (left) at (4.75,0);
        \coordinate (right) at (10.25,0);

        \draw (left) .. controls (6, -0.7) and (9, -0.7) .. (right);
        \draw[dashed] (left) .. controls (6, 0.7) and (9, 0.7) .. (right);
        
        \node at (9.3, 1.3) {$\tgamma_1$};

        \node at (7.5, -2) {$\big\downarrow$};
        \node at (7.8, -2) {$Z$};
        \node at (7.5, -3) {$\{|Z| = \epsilon'\}$};

    \end{tikzpicture}
\end{center}

    \caption{Construction of the 3-chain $\tgamma_1$ in Lemma \ref{lem:Open4Cycles}.}
    \label{fig:Open4CyclesCorrect}
\end{figure}

Then, we add to $\tgamma_1$ a second 3-chain $\tgamma_2$ defined as follows: For any $t \in [0,1]$, $\tgamma_2$ is swept out by a homotopy between the loops $\ell_0 = (X(t, \theta), Y(t), Z_0(t, \theta))$ and $\ell_1 = (X(t, \theta), Y(t, \theta), Z_1(t, \theta))$ where
$$
    Z_1(t, \theta) = -\frac{H(X(t,\theta), Y(t), q)}{1+\tq_{R-2}^{\fa}X(t,\theta)^{-\fa}Y(t)^{-\fb}}.
$$
Note that when $t \in (0,1)$, $(X(t,\theta), Y(t))$ is perturbed off the hypersurface $\{X^{\fa}Y^{\fb} = -\tq_{R-2}^{\fa}\}$ and the denominator above is non-zero. To further validate this definition, we first note that the limit
$$
    \lim_{t \to 0,1} Z_1(t, \theta)
$$
is non-zero for any $\theta$ since the derivative $\frac{\partial H_0(Y_0,q,x)}{\partial Y_0}$ does not vanish at the endpoints $Y(0), Y(1) \in P_{q,x}$ due to the regularity of $H_0$ (Assumption \ref{assump:Regular}). Thus the loop $\ell_1$ is defined in the limit $t \to 0,1$. Moreover, for any $t \in (0,1)$, $X(t)^{\fa}Y(t)^{\fb} = -\tq_{R-2}^{\fa}$ implies that
$$
    -\frac{H(X(t,\theta), Y(t), q)}{1+\tq_{R-2}^{\fa}X(t,\theta)^{-\fa}Y(t)^{-\fb}} = -\frac{H(X(t,\theta), Y(t), q)}{1 - (1-r(t)\epsilon'e^{\theta\sqrt{-1}})^{-1}} = ((r(t)\epsilon')^{-1}e^{-\theta\sqrt{-1}}-1)H(X(t,\theta), Y(t), q).
$$
Since $\epsilon'$ is sufficiently small and all powers of $X$ in $H(X,Y,q)$ (the $m_i$'s) are non-negative, it follows that $\ell_1$ has winding number $-1$ around $Z=0$ and is thus indeed homotopic to $\ell_0$. By construction, $\ell_1$ is contained in $S_{\tq}$ for any $t \in [0,1]$. We may take the homotopy to be one in the $Z$-coordinate only, i.e. it fixes $(X(t,\theta), Y(t))$ for every $t, \theta$.

We take $\tgamma := \tgamma_1 + \tgamma_2$. Its boundary $\partial \tgamma$ consists of three 2-chains: The first is the deformation of $\gamma_0'$ over $\{|Z| = \epsilon'\}$ and is taken to be contained in $S_{\tq}$. The second is the union of all loops $\ell_1$ for $t \in [0,1]$, which is also contained in $S_{\tq}$. The third is the defining homotopy of $\tgamma_2$ restricted to $t =0,1$, which is contained in $P_{q,x} \times \bC^* \subset S_{\tq}$. Therefore, $\tgamma$ is a relative 3-cycle in $H_3((\bC^*)^3, S_{\tq}; \bZ)$, and we may define
$$
  \iota_2(s(\Gamma_0)):= \talpha(\tgamma) = \tpi^{-1}(\tgamma) \in H_4(\tcX^\vee_{\tq}; \bZ).
$$

We have by Theorem \ref{thm:RelPeriod} that
$$
    \int_{\iota_2(s(\Gamma_0))} \tOmega_{\tq} = 2\pi\sqrt{-1}\int_{\tgamma} \tomega = 2\pi\sqrt{-1}\int_{\tgamma_1} \tomega + 2\pi\sqrt{-1}\int_{\tgamma_2} \tomega.
$$
The construction of $\tgamma_1$ implies that
$$
  2\pi\sqrt{-1} \int_{\tgamma_1} \tomega = (2\pi\sqrt{-1})^2 \int_{\gamma} \omega =  2\pi\sqrt{-1} \int_{s(\Gamma_0)} \Omega_q
$$
(again by Theorem \ref{thm:RelPeriod}). Moreover, on $\tgamma_2$, locally $(X,Y)$ are the coordinates of the parameterized curve $\gamma_0$ and the defining homotopy is in the $Z$-coordinate only. Thus we have
$$
    \int_{\tgamma_2} \tomega = 0.
$$
This confirms \eqref{eqn:OpenPeriodCorr}.

Finally, the injectivity of $\iota_2$ is equivalent to that of $\iota_2 \circ s$. By Lemma \ref{lem:StokesRel3Cycles}, applying $\fa^{-1}x \frac{\partial}{\partial x} = \fa^{-1}\tq_{R-2} \frac{\partial}{\partial \tq_{R-2}}$ to the periods \eqref{eqn:OpenPeriodCorr} gives
$$
  \int_{\Gamma_0} \Omega_{q,x}^0,
$$
which are linearly independent as $\Gamma_0$ ranges through a basis of $H_2(\cY_{q,x}; \bZ)$ by an analog of Theorem \ref{thm:PeriodSoln} (see Remark \ref{rem:SmallPFSystem}). The injectivity of $\iota_2 \circ s$ thus follows.
\end{proof}

\begin{proof}[Proof of Theorem \ref{thm:CycleCorr}]
By Proposition \ref{prop:RelHomologySES}, we have
$$
  H_3(\cX^\vee_q, \cY_{q,x}; \bZ) = H_3(\cX^\vee_q; \bZ) \oplus s(H_2(\cY_{q,x}; \bZ)).
$$
We define the desired map $\iota: H_3(\cX^\vee_q, \cY_{q,x}; \bZ) \to H_4(\tcX^\vee_{\tq}; \bZ)$ by the injective maps $\iota_1$ and $\iota_2$ on the direct summands defined in Lemmas \ref{lem:Closed4Cycles} and \ref{lem:Open4Cycles} respectively. Note that $\iota$ is injective, since periods of cycles in the image of $\iota_1$ are independent of $\tq_{R-2} = x$ while periods of cycles in the image of $\iota_2$ have non-trivial dependence on $\tq_{R-2} = x$. Moreover, $\iota$ is surjective over $\bQ$, since \eqref{eqn:RankSum} implies that $\rank(H_3(\cX^\vee_q, \cY_{q,x}; \bZ)) = \rank(H_4(\tcX^\vee_{\tq}; \bZ))$.
\end{proof}

\begin{example}\label{ex:C3Cycles}
\rm{
We illustrate the discussion so far with the basic example $\cX = \bC^3$, $\cL$ an outer brane, and $f=1$, as in Examples \ref{ex:C3}, \ref{ex:C3PicardFuchs}. Among the three linearly independent solutions to $\tcP$ found in Example \ref{ex:C3PicardFuchs}, the constant $1$ is the period of $\tomega$ over the 3-cycle represented by the real 3-torus $(U(1))^3$ in $H_3((\bC^*)^3; \bZ) \subset H_3((\bC^*)^3, S_{\tq}; \bZ)$, and thus the period of $\tOmega_{\tq_1}$ over the corresponding 4-cycle on $\tcX^\vee_{\tq_1}$. As for the other two solutions, we first recall that $\Delta_0 = [-1, 1]$ and $\cX_0 = \Tot(K_{\bP^1})$. At the base of the conic fibration, the set $P_x$ is the set of two points in $\bC^*$ defined by
$$
  0 = H_0(Y_0,x) = -xY_0^{-1} + Y_0 + 1.
$$
In view of Lemmas \ref{lem:StokesRel3Cycles} and \ref{lem:Open4Cycles}, we apply $\tq_1\frac{\partial}{\partial \tq_1}$ to the two additional solutions
$$
  \log \tq_1, \qquad \frac{1}{2} (\log \tq_1)^2 + 2W^{\cX, (\cL, 1)}_{1}(\tq_1)
$$
to $\tcP$ to obtain
$$  
  1, \qquad \log \tq_1 + 2\tq_1\frac{\partial}{\partial \tq_1} W^{\cX, (\cL, 1)}_{1}(\tq_1).
$$
These are the periods of $\omega^0$ over the following 1-cycles in $H_1(\bC^*, P_x; \bZ)$ respectively: a loop around $Y=0$ (based at a point in $P_x$); a path between the two points in $P_x$. They are also the periods of $\Omega_x^0$ over the corresponding 2-cycles on $\cY_x$. We can see that these are the two linearly independent solutions to the Picard-Fuchs system of $\Tot(K_{\bP^1})$: the constant and the closed mirror map (see Remark \ref{rem:SmallPFSystem}).
}
\end{example}

\subsection{Relation to Aganagic-Vafa B-branes}\label{sect:AVBBranes}
We discuss how our constructions and results relate to those based on Aganagic-Vafa B-branes in the literature. As originally proposed by Aganagic-Vafa \cite{AV00}, these are real 2-dimensional subspaces of the Hori-Vafa mirror $\cX^\vee_q$ of $\cX$ parameterized by the open moduli parameter $x$. Moreover, the superpotential of the B-brane should recover the B-model disk function of the corresponding A-brane $\cL$ in $\cX$, which is a first formulation of open mirror symmetry.


For simplicity, in this discussion we assume that $f \in \bZ$, i.e. $\fa = 1$.

We start with the definition of the B-branes. Let $(q,x) \in \tU$. Fix a general reference point $p_* = (X_*, Y_*) \in C_q$. Let $\cC_* \subset \cX^\vee_q$ be the holomorphic curve defined by the equations
$$
  u = H(X, Y, q) = 0, \qquad X = X_*, \qquad Y = Y_*.
$$
Note that $\cC_* \cong \bC$ and $v$ gives a coordinate on $\cC_*$. Now take a point $p = (X_p, Y_p) \in P_{q,x} \subset C_q$ and a path $\gamma_p \subset C_q$ from $p$ to the reference point $p_*$. This data defines a curve $\cC_p$ in
$$
  \{u = H(X, Y, q) = 0\} \subset \cX^\vee_q
$$
consisting of $(v, X, Y)$ parameterized by $v \in \bC$ as follows: The coordinates $(X,Y)$ only depends on the radius $|v|$. As $|v|$ varies from 0 to a fixed, sufficiently large threshold $\Lambda \gg 0$, $(X(|v|), Y(|v|))$ traces out the path $\gamma_p$. When $|v| \ge \Lambda$, $(X(|v|), Y(|v|))$ stays constantly at $p_*$. Note that $\cC_p$ is homeomorphic to $\bC$ but is in general not holomorphic. See Figure \ref{fig:AVBBranes} for an illustration of the B-branes.

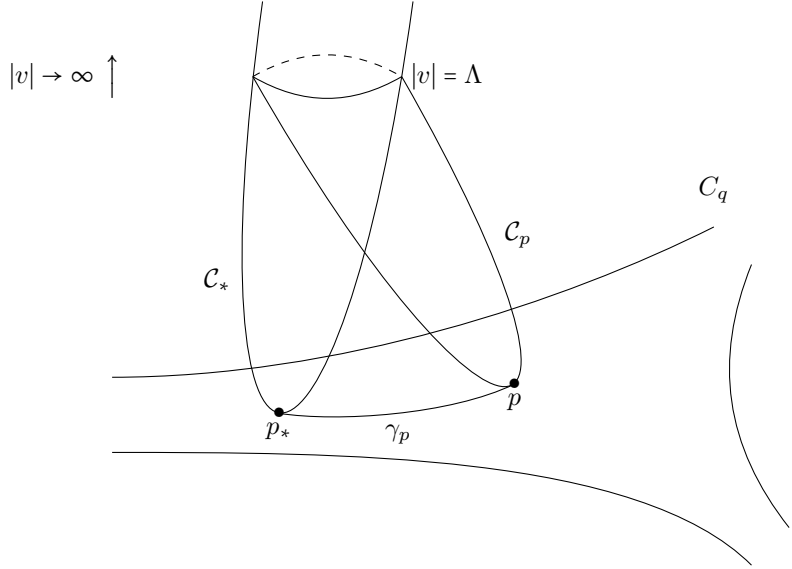
\begin{figure}[h]
\begin{center}
    \begin{tikzpicture}[scale=1]
        
        \draw (0,1) .. controls (3, 1) and (6, 2) .. (8,3);
        \draw (0,0) .. controls (3, 0) and (7, 0) .. (8.5, -1.5);
        \draw (8.5,2.5) to[bend right] (9,-1);

        \draw (2,6) .. controls (1, -1.3) and (3, -1.3) .. (4, 6);
 
        \coordinate (l) at (1.87, 5);
        \coordinate (r) at (3.85, 5);

        \draw (l) .. controls (5, -0.5) and (7, -0.5) .. (r);
        \draw[dashed] (l) to[bend left] (r);
        \draw (l) to[bend right] (r);

        \coordinate (0) at (2.22, 0.52);
        \coordinate (1) at (5.35, 0.91);

        \node at (0) {$\bullet$};
        \node at (1) {$\bullet$};
        \draw (0) .. controls (3, 0.4) and (4.5, 0.5) .. (1);

        \node at (0,5) {$\big \uparrow$};
        \node at (-0.8, 5) {$|v| \to \infty$};
        \node at (1.4, 2.3) {$\cC_*$};
        \node at (5.4, 2.9) {$\cC_p$};
        \node[right] at (r) {$|v| = \Lambda$};

        \node at (8, 3.5) {$C_q$};
        \node[below] at (0) {$p_*$};
        \node[below] at (1) {$p$};
        \node[below] at (3.8, 0.5) {$\gamma_p$};

    \end{tikzpicture}
\end{center}

    \caption{Aganagic-Vafa B-branes $\cC_*$ and $\cC_p$.}
    \label{fig:AVBBranes}
\end{figure}

Fix a branch of $\log Y$ that contains $\{p_*\} \cup P_{q,x}$ and assume that the path $\gamma_p$ is chosen within this branch. The \emph{superpotential} of the brane $\cC_p$ is an integral of a 2-form over all of $\cC_p$ and is reduced, in the radial direction of $v$, to the line integral
$$
  W(\cC_p) = 2\pi\sqrt{-1}\int_{p_*}^p \log Y \frac{d\bar{X}}{\bar{X}}
$$
along $\gamma_p$ in $C_q$. Here we use the coordinate $\bar{X} = -XY^f$ as in the proof of Lemma \ref{lem:StokesRel3Cycles}. Similarly, we have $W(\cC_*) = 0$ for the reference brane $\cC_*$. Then the open mirror theorem, in the case where the A-brane $\cL$ is effective ($\fm=1$), states that for an appropriate choice of $p$ and $\gamma_p$,
\begin{equation}\label{eqn:OpenMirrorOriginal}
  \left(x\frac{\partial}{\partial x} \right)^2 W^{\cX, \cL,f}(q,x) = \frac{1}{2\pi\sqrt{-1}} x\frac{\partial}{\partial x} \left( W(\cC_p) - W(\cC_*) \right) = x\frac{\partial}{\partial x} \int_{p_*}^p \log Y \frac{d\bar{X}}{\bar{X}} = \log Y_p(x).
\end{equation}
See \cite{FLT12} for precise statements that also cover ineffective A-branes in general.

The difference
$$
  W(\cC_p) - W(\cC_*)
$$
can be viewed as an integral on the 2-cycle $\Gamma_0'$ in $\cC_p \cup \cC_*$ defined by $|v| \le \Lambda$ (which generates $H_2(\cC_p \cup \cC_*; \bZ)$). Moreover, $\Gamma_0'$ bounds a relative 3-cycle $\Gamma'$ in $\cX^\vee_q$ consisting of all the points $(v', X, Y)$ such that $(X,Y) \in \gamma_p$ and $|v'|$ is upper-bounded by the radius of $v$ on $\cC_p$. Then we may write
$$
  W(\cC_p) - W(\cC_*) = \int_{\Gamma'} \Omega_q.
$$
Indeed, Mayr \cite{Mayr01} proposed a map 
$$
  H_3(\cX^\vee_q, \cC_p \cup \cC_*; \bZ) \to H_4(\tcX^\vee_{\tq}; \bZ)
$$
under which the periods are related in a way similar to our Theorem \ref{thm:CycleCorr}.

We now relate the above constructions to ours based on the complex hypersurface $\cY_{q,x}$ in $\cX^\vee_q$. Let $p_1, p_2 \in P_{q,x}$ be distinct points and choose a path $\gamma_0$ in $\bC^*$ from $p_1$ to $p_2$ passing through the reference $Y$-coordinate $Y_*$. The relative 1-cycle $\gamma_0 \in H_1(\bC^*, P_{q,x}; \bZ)$ induces a 2-cycle $\Gamma_0 = \alpha_0(\gamma_0) \in H_2(\cY_{q,x}; \bZ)$, which bounds a relative 3-cycle $\Gamma = s(\Gamma_0) \in H_3(\cX^\vee_q, \cY_{q,x}; \bZ)$ (Lemma \ref{lem:SectionExists}). Our Theorem \ref{thm:CycleCorr} (and particularly Lemma \ref{lem:Open4Cycles}) finds a 4-cycle $\tGamma \in H_4(\tcX^\vee_{\tq}; \bZ)$ such that
$$
  \frac{1}{2\pi\sqrt{-1}}\tq_{R-2}\frac{\partial}{\partial \tq_{R-2}} \int_{\tGamma}  \tOmega_{\tq} = x\frac{\partial}{\partial x} \int_{\Gamma} \Omega_q = \int_{\Gamma_0} \Omega_{q,x}^0 = 2\pi\sqrt{-1} \int_{p_1}^{p_2} \frac{dY}{Y}.
$$
We may assume that the above construction may be performed within our choice of branch for $\log Y$. (Note that Lemmas \ref{lem:MCH1Surjective}, \ref{lem:MCRelH1Surjective} used in the construction are valid under the pullback $\bC \to \bC^*$ to the universal cover in the $Y$-coordinate.) This gives
$$
  \frac{1}{(2\pi\sqrt{-1})^2}\left(\tq_{R-2}\frac{\partial}{\partial \tq_{R-2}}\right)^2 \int_{\tGamma}  \tOmega_{\tq} = \frac{1}{2\pi\sqrt{-1}}\left(x\frac{\partial}{\partial x} \right)^2 \int_{\Gamma} \Omega_q = \log Y_{p_2}(x) - \log Y_{p_1}(x).
$$
On the other hand, by Lemma \ref{lem:MCRelH1Surjective}, we lift the path $\gamma_0$ to a path $\gamma_0'$ in $C_q$ from $p_1$ to $p_2$ passing through $p_*$, and use the two segments of $\gamma_0'$ to define B-branes $\cC_{p_1}$ and $\cC_{p_2}$. Similar to above, let $\Gamma_0'$ be the 2-cycle in $\cC_{p_1} \cup \cC_{p_2}$ defined by $|v| \le \Lambda$, which generates $H_2(\cC_{p_1} \cup \cC_{p_2}; \bZ)$ and bounds a relative 3-cycle $\Gamma' \in H_3(\cX^\vee_q, \cC_{p_1} \cup \cC_{p_2}; \bZ)$. Our earlier discussion gives
\begin{equation}\label{eqn:PeriodFromBBranes}
\begin{aligned}
   \frac{1}{2\pi\sqrt{-1}} x\frac{\partial}{\partial x} \int_{\Gamma'} \Omega_q & = \frac{1}{2\pi\sqrt{-1}} x\frac{\partial}{\partial x} \left(W(\cC_{p_2}) - W(\cC_{p_1}) \right) \\
   & = \log Y_{p_2}(x) - \log Y_{p_1}(x) = \frac{1}{(2\pi\sqrt{-1})^2}\left(\tq_{R-2}\frac{\partial}{\partial \tq_{R-2}}\right)^2 \int_{\tGamma} \tOmega_{\tq}.
\end{aligned}
\end{equation}
In particular, we obtain a description of the $\tq_{R-2}$-dependence of the power series part of the period $\int_{\tGamma} \tOmega_{\tq}$ in terms of superpotentials of the B-branes. In certain cases, for instance when $p_1$ and $p_2$ correspond to neighboring effective A-branes on $\cX$, version \eqref{eqn:OpenMirrorOriginal} of the open mirror theorem relates the difference of B-brane superpotentials $W(\cC_{p_2}) - W(\cC_{p_1})$ to the difference of B-model disk functions of the corresponding A-branes. Then the description \eqref{eqn:PeriodFromBBranes} is consistent with Proposition \ref{prop:DiskFnSoln} which describes the $\tq_{R-2}$-dependence of the power series part of certain solutions to $\tcP$ in terms of disk functions of the A-branes.


%% file: hodge.tex

\section{Variations of mixed Hodge structures}\label{sect:MHS}
On the dual side of Section \ref{sect:Cycles}, in this section, we establish the open/closed correspondence at the level of (variations of) mixed Hodge structures ((V)MHS) (in the sense of Deligne \cite{Deligne71a,Deligne71b}) associated to the Hori-Vafa mirror families (Theorem \ref{thm:MHSCorr}). Our study utilizes a combinatorial presentation of the VMHS which we first review.

\subsection{Vector spaces from Laurent polynomials}
Recall that in Section \ref{sect:Laurent} we defined families of Laurent polynomials
$$
  H(X, Y, q), \qquad \tH(X, Y, Z, \tq), \qquad H_0(Y_0, q, x).
$$
Over the parameter domain $(q,x) = \tq \in \tU_\epsilon$, these polynomials are regular with respect to their Newton polytopes $\Delta, \tDelta, \Delta_0$ respectively (Assumption \ref{assump:Regular}). We now recall the construction of Batyrev \cite{Batyrev93} of graded $\bC$-vector spaces associated to these polynomials. Let
$$
  \cS_{\Delta}, \qquad \cS_{\tDelta}, \qquad \cS_{\Delta_0}
$$
be the subring of
$$
  \bC[X_0, X^{\pm1}, Y^{\pm1}], \qquad \bC[X_0, X^{\pm1}, Y^{\pm1}, Z^{\pm1}], \qquad \bC[X_0, Y_0^{\pm1}]
$$
generated as a $\bC$-vector space by monomials
$$
  \left\{X_0^kX^aY^b : \left(a,b\right) \in k\Delta \right\}, \qquad
  \left\{X_0^kX^aY^bZ^c : \left(a,b,c\right) \in k\tDelta \right\}, \qquad
  \left\{X_0^kY_0^b : b \in k\Delta_0 \right\}
$$
respectively. Here we take the convention that the unit $1$ is included in the sets of monomials above by $k=0$. The subrings are graded by the degree of the variable $X_0$, i.e.
$$
  \deg X_0^kX^aY^b = \deg X_0^kX^aY^bZ^c = \deg X_0^kY_0^b = k.
$$
On $\cS_{\Delta}$, consider the following differential operators defined by $H(X, Y, q)$:
$$  
  \cD_0 := X_0\frac{\partial}{\partial X_0} + X_0H(X, Y, q), \quad 
  \cD_X := X\frac{\partial}{\partial X} + X_0X\frac{\partial H(X, Y, q)}{\partial X}, \quad 
  \cD_Y := Y\frac{\partial}{\partial Y} + X_0Y\frac{\partial H(X, Y, q)}{\partial Y}.
$$
Similarly, on $\cS_{\tDelta}$, consider the following differential operators defined by $\tH(X, Y, Z, \tq)$:
\begin{align*}
  & \tcD_0 := X_0\frac{\partial}{\partial X_0} + X_0\tH(X, Y, Z, \tq), 
  && \tcD_X := X\frac{\partial}{\partial X} + X_0X\frac{\partial \tH(X, Y, Z, \tq)}{\partial X},\\
  & \tcD_Y := Y\frac{\partial}{\partial Y} + X_0Y\frac{\partial \tH(X, Y, Z, \tq)}{\partial Y},
  && \tcD_Z := Z\frac{\partial}{\partial Z} + X_0Z\frac{\partial \tH(X, Y, Z, \tq)}{\partial Z}.
\end{align*}
Finally, on $\cS_{\Delta_0}$, consider the following differential operators defined by $H_0(Y_0, q, x)$:
$$  
  \cD^0_0 := X_0\frac{\partial}{\partial X_0} + X_0H_0(Y_0, q, x), \qquad 
  \cD^0_{Y_0} := Y_0\frac{\partial}{\partial Y_0} + X_0Y_0\frac{\partial H_0(Y_0, q, x)}{\partial Y_0}.
$$
Then, define
\begin{align*}
  & \cR_H := \cS_{\Delta} \big/ (\cD_0\cS_{\Delta} + \cD_X\cS_{\Delta} + \cD_Y\cS_{\Delta}), \\
  & \cR_{\tH} := \cS_{\tDelta} \big/ (\tcD_0\cS_{\tDelta} + \tcD_X\cS_{\tDelta} + \tcD_Y\cS_{\tDelta} + \tcD_Z\cS_{\tDelta}),\\
  & \cR_{H_0} := \cS_{\Delta_0} \big/ (\cD^0_0\cS_{\Delta_0} + \cD^0_{Y_0}\cS_{\Delta_0}).
\end{align*}
It holds that
\begin{equation}\label{eqn:RDimension}
  \dim_{\bC} \cR_H = \Vol(\Delta), \qquad   \dim_{\bC} \cR_{\tH} = \Vol(\tDelta), \qquad \dim_{\bC} \cR_{H_0} = \Vol(\Delta_0).
\end{equation}

\subsection{VMHS of affine hypersurfaces in algebraic tori}
Batyrev \cite{Batyrev93} showed that the $\bC$-vector spaces $\cR_H$, $\cR_{\tH}$, $\cR_{H_0}$ endowed with two filtrations can be used to describe the VMHS arising from the affine hypersurfaces $C_q$, $S_{\tq}$, $P_{q,x}$ in algebraic tori defined by the Laurent polynomials. 
We review the relevant results following \cite{KM10}. The two filtrations on $\cR_H$, $\cR_{\tH}$, $\cR_{H_0}$ will descend from two filtrations on $\cS_{\Delta}$, $\cS_{\tDelta}$, $\cS_{\Delta_0}$ respectively.

First, the \emph{$\cE$-filtration} is a decreasing filtration
$$
  \cdots \supseteq \cE^{-k} \supseteq \cdots \supseteq \cE^{-1} \supseteq \cE^0 = \bC 1 \supset \cE^1 = 0
$$
where $\cE^{-k}$ is the $\bC$-vector subspace spanned by monomials of degree $\le k$. It holds that
$$  
  \cE^{-2}\cR_H = \cR_H, \qquad   \cE^{-3}\cR_{\tH} = \cR_{\tH}, \qquad \cE^{-1}\cR_{H_0} = \cR_{H_0}.
$$

Second, the \emph{$\cI$-filtration} is an increasing filtration defined as follows: In $\cS_{\Delta}$, for $0 \le l \le 3$, let $\cI_l\cS_{\Delta}$ be the homogenous ideal generated as a $\bC$-vector subspace by monomials $X_0^kX^aY^b$ such that $\left(a, b\right)$ does not belong to any codimension-$l$ face of $k\Delta$. Note that $\cI_3\cS_{\Delta}$ is generated by all monomials with positive degrees. We set $\cI_4\cS_{\Delta} = \cS_{\Delta}$. Descending to $\cR_H$, and applying the same construction to $\cR_{\tH}$ and $\cR_{H_0}$, we have
\begin{align*}
  & 0 = \cI_0\cR_H \subseteq \cdots \subseteq \cI_3\cR_H \subset \cI_4\cR_H = \cR_H;\\
  & 0 = \cI_0\cR_{\tH} \subseteq \cdots \subseteq \cI_4\cR_{\tH} \subset \cI_5\cR_{\tH} = \cR_{\tH};\\
  & 0 = \cI_0\cR_{H_0} \subseteq \cI_1\cR_{H_0} \subseteq \cI_2\cR_{H_0} \subset \cI_3\cR_{H_0} = \cR_{H_0}.
\end{align*}

As observed by Stienstra \cite[Theorem 7]{Stienstra98} and Konishi-Minabe \cite[Theorem 4.2]{KM10}, the results of Batyrev \cite{Batyrev93} give isomorphisms between $\cR_H$, $\cR_{\tH}$, $\cR_{H_0}$ and
$$
  \cH := H^2((\bC^*)^2, C_q; \bC), \qquad \tcH := H^3((\bC^*)^3, S_{\tq}; \bC), \qquad \cH_0 := H^1(\bC^*, P_{q,x}; \bC)
$$
respectively under which the $\cE$- and $\cI$-filtrations induce the Hodge filtrations $\cF^\bullet$ and weight filtrations $\cW_\bullet$ respectively.

\begin{theorem}[\cite{Batyrev93,Stienstra98,KM10}]\label{thm:BatMHS}
There are $\bC$-vector space isomorphisms
$$
  \rho: \cR_H \xrightarrow{\sim} \cH, \qquad
  \trho: \cR_{\tH} \xrightarrow{\sim} \tcH, \qquad
  \rho_0: \cR_{H_0} \xrightarrow{\sim} \cH_0
$$
under which the following hold:
\begin{itemize} 
  \item We have
  $$
    \rho(1) = [\omega], \qquad \trho(1) = [\tomega], \qquad \rho_0(1) = [\omega^0]
  $$
  where the forms $\omega, \tomega, \omega^0$ are defined in \eqref{eqn:FormsOnTori}.

  \item For any $k$, we have
  $$   
    \rho(\cE^{-k}\cR_H) = \cF^{-k+2}\cH, \qquad 
    \trho(\cE^{-k}\cR_{\tH}) = \cF^{-k+3}\tcH, \qquad
    \rho_0(\cE^{-k}\cR_{H_0}) = \cF^{-k+1}\cH_0.
  $$

  \item We have
  $$
    \cW_1\cH = \rho(\cI_1\cR_H), \quad \cW_2\cH = \cW_3\cH = \rho(\cI_3\cR_H), \quad \cW_4\cH = \rho(\cI_4\cR_H) = \cH;
  $$
  $$
    \cW_{2}\tcH = \trho(\cI_1\cR_{\tH}), \quad \cW_{3}\tcH = \trho(\cI_2\cR_{\tH}), \quad \cW_{4}\tcH = \cW_{5}\tcH = \trho(\cI_4\cR_{\tH}), \quad \cW_{6}\tcH = \trho(\cI_5\cR_{\tH}) = \tcH;
  $$
  $$
    \cW_0\cH_0 = \cW_1\cH_0 = \rho_0(\cI_2\cR_{H_0}), \quad  \cW_2\cH_0 = \rho_0(\cI_3\cR_{H_0}) = \cH_0.
  $$
\end{itemize}
Moreover, the isomorphisms respect the variations over the parameters $q, \tq, (q, x)$ respectively.
\end{theorem}

The variations of $[\omega], [\tomega]$ in the relative cohomology are governed by the Picard-Fuchs systems $\cP, \tcP$ respectively. A similar remark holds for $[\omega^0]$ (see Remark \ref{rem:SmallPFSystem}).



\subsection{VMHS of Hori-Vafa mirrors}
Similarly, Konishi-Minabe \cite{KM10} showed that if the 2-dimensional polytope $\Delta$ is \emph{reflexive}\footnote{For the 2-dimensional polytope $\Delta$, this means that $\vzero$ is in the interior of $\Delta$ and the distance between $\vzero$ and the line generated by each codimension-1 face of $\Delta$ is 1. See \cite[Section 4]{Batyrev94} for a more general and precise definition.}, so that the toric Calabi-Yau 3-orbifold $\cX$ is a local surface, $\cR_H$ also describes the VMHS on the middle-dimensional cohomology
$$
  H^3(\cX^\vee_q; \bC)
$$
of the Hori-Vafa mirror $\cX^\vee_q$. In upcoming joint work \cite{AY25} with Aleshkin, we extend this result to general polytopes of any dimension based on the methods of \cite{Batyrev93,KM10}. In particular, we give isomorphisms between $\cR_H$, $\cR_{\tH}$, $\cR_{H_0}$ and
$$
  \cH^{\HV} := H^3(\cX^\vee_q; \bC), \qquad \tcH^{\HV} := H^4(\tcX^\vee_{\tq}; \bC), \qquad \cH_0^{\HV} := H^2(\cY_{q,x}; \bC)
$$
respectively under which the $\cE$- and $\cI$-filtrations induce the Hodge filtrations $\cF^\bullet$ and weight filtrations $\cW_\bullet$ respectively.

\begin{theorem}[\cite{KM10, AY25}]\label{thm:HVMHS}
There are $\bC$-vector space isomorphisms
$$
  \rho: \cR_H \xrightarrow{\sim} \cH^{\HV}, \qquad
  \trho: \cR_{\tH} \xrightarrow{\sim} \tcH^{\HV}, \qquad
  \rho_0: \cR_{H_0} \xrightarrow{\sim} \cH_0^{\HV}
$$
under which the following hold:
\begin{itemize} 
  \item We have
  $$
    \rho(1) = [\Omega_q], \qquad \trho(1) = [\tOmega_{\tq}], \qquad \rho_0(1) = [\Omega_{q,x}^0]
  $$
  where the forms $\Omega_q, \tOmega_{\tq}, \Omega_{q,x}^0$ are defined in Section \ref{sect:HoriVafa}.

  \item For any $k$, we have
  $$   
    \rho(\cE^{-k}\cR_H) = \cF^{-k+3}\cH^{\HV}, \qquad 
    \trho(\cE^{-k}\cR_{\tH}) = \cF^{-k+4}\tcH^{\HV}, \qquad
    \rho_0(\cE^{-k}\cR_{H_0}) = \cF^{-k+2}\cH_0^{\HV}.
  $$

  \item We have
  $$
    \cW_3\cH^{\HV} = \rho(\cI_1\cR_H), \quad \cW_4\cH^{\HV} = \cW_5\cH^{\HV} = \rho(\cI_3\cR_H), \quad \cW_6\cH^{\HV} = \rho(\cI_4\cR_H) = \cH^{\HV};
  $$
  $$
    \cW_{4}\tcH^{\HV} = \trho(\cI_1\cR_{\tH}), \quad \cW_{5}\tcH^{\HV} = \trho(\cI_2\cR_{\tH}), \quad \cW_{6}\tcH^{\HV} = \cW_{7}\tcH^{\HV} = \trho(\cI_4\cR_{\tH}), \quad \cW_{8}\tcH^{\HV} = \trho(\cI_5\cR_{\tH}) = \tcH^{\HV};
  $$
  $$
    \cW_2\cH_0^{\HV} = \cW_3\cH_0^{\HV} = \rho_0(\cI_2\cR_{H_0}), \quad  \cW_4\cH_0^{\HV} = \rho_0(\cI_3\cR_{H_0}) = \cH_0^{\HV}.
  $$
\end{itemize}
Moreover, the isomorphisms respect the variations over the parameters $q, \tq, (q, x)$ respectively.
\end{theorem}

In particular, we have isomorphisms of VMHS
$$
  \cH \cong \cH^{\HV} \otimes T(1), \qquad \tcH \cong \tcH^{\HV} \otimes T(1), \qquad \cH_0 \cong \cH_0^{\HV} \otimes T(1)
$$
where $\otimes T(1)$ is the Tate twist by $1$. This a dual version of Theorem \ref{thm:RelPeriod} on homology and periods.

\begin{notation} \rm{
For $n \in \bZ$, let $T(-n)$ denote the $(-n)$-th \emph{Tate} Hodge structure which is a pure Hodge structure on the lattice $(2\pi\sqrt{-1})^{-n} \bZ \subset \bC$ of weight $2n$ and type $(n,n)$.
}
\end{notation}

The variations of $[\Omega_q], [\tOmega_{\tq}]$ in the cohomology are governed by the Picard-Fuchs systems $\cP, \tcP$ respectively. A similar remark holds for $[\Omega_{q,x}^0]$ (see Remark \ref{rem:SmallPFSystem}).

\subsection{Open/closed correspondence of VMHS}
Our main technical result of the section is the that $\cR_{\tH}$ is an extension of $\cR_H$ by $\cR_{H_0}$, under which the $\cE$- and $\cI$-filtrations are shifted appropriately.

\begin{proposition}\label{prop:RExtension}
For $(q,x) = \tq \in \tU_\epsilon$, there is a short exact sequence of $\bC$-vector spaces
\begin{equation}\label{eqn:MHSSES}
    \xymatrix{
        0 \ar[r] & \cR_{H_0} \ar[r]^\iota & \cR_{\tH} \ar[r]^\pi & \cR_H \ar[r] & 0
    }
\end{equation}
such that 
$$
  \iota(\cE^{-k}\cR_{H_0}) \subseteq \cE^{-k-1}\cR_{\tH}, \quad 
\pi(\cE^{-k}\cR_{\tH}) \subseteq \cE^{-k}\cR_H, \qquad \text{for any $k$};
$$
$$
  \iota(\cI_2\cR_{H_0}) \subseteq \cI_{1}\cR_{\tH}, \qquad \iota(\cI_3\cR_{H_0}) \subseteq \cI_4\cR_{\tH}; \qquad \qquad
  \pi(\cI_l\cR_{\tH}) \subseteq \cI_{l-1}\cR_H \qquad \text{for any $l$}.
$$
\end{proposition}

The proof of the proposition will be given in Section \ref{sect:RExtensionProof} below. In view of Theorems \ref{thm:BatMHS}, \ref{thm:HVMHS}, it directly translates into the following result.

\begin{corollary}\label{cor:MHSExtension}
Over $(q,x) = \tq \in \tU_\epsilon$, there are short exact sequences of VMHS
$$
  \xymatrix{
        0 \ar[r] & H^1(\bC^*, P_{q,x}; \bC) \ar[r] & H^3((\bC^*)^3, S_{\tq}; \bC) \otimes T(1) \ar[r] & H^2((\bC^*)^2, C_q; \bC) \ar[r] & 0,
    }
$$
\begin{equation}\label{eqn:HVSES}
    \xymatrix{
        0 \ar[r] & H^2(\cY_{q,x}; \bC) \ar[r] & H^4(\tcX^\vee_{\tq}; \bC) \otimes T(1) \ar[r] & H^3(\cX^\vee_q; \bC) \ar[r] & 0.
    }
\end{equation}
\end{corollary}

We compare \eqref{eqn:HVSES} with the sequence coming from the relative cohomology of the pair $(\cX^\vee_q, \cY_{q,x})$. We have the long exact sequence
\begin{equation}\label{eqn:RelCohomologyLES}
  \cdots \to H^2(\cX^\vee_q; \bC) \to H^2(\cY_{q,x}; \bC) \to H^3(\cX^\vee_q, \cY_{q,x}; \bC) \to H^3(\cX^\vee_q; \bC) \to H^3(\cY_{q,x}; \bC) \to \cdots.
\end{equation}
The following lemma is dual to Proposition \ref{prop:RelHomologySES} on integral homology.

\begin{lemma}\label{prop:RelCohomologySES}
The following sequence induced by \eqref{eqn:RelCohomologyLES}
\begin{equation}\label{eqn:RelCohomologySES}
  \xymatrix{
    0 \ar[r] & H^2(\cY_{q,x}; \bC) \ar[r] & H^3(\cX^\vee_q, \cY_{q,x}; \bC) \ar[r] & H^3(\cX^\vee_q; \bC) \ar[r] & 0
  }
\end{equation}
is a short exact sequence.
\end{lemma}

\begin{proof}
On the right, we have $H^3(\cY_{q,x}; \bC) = 0$ by the analysis in Lemma \ref{lem:CycleInjection} applied to cohomology, or alternatively by that $\cY_{q,x}$ is a smooth affine algebraic variety of dimension 2. On the left, for the inclusion $\cX^\vee_q \subset \bC^2 \times (\bC^*)^2$, a Lefschetz-type theorem of Danilov-Khovanski\^{i} \cite[Corollary 3.8]{DK78} states that 
$$
  H^2(\cX^\vee_q; \bC) \cong H^2(\bC^2 \times (\bC^*)^2; \bC) \cong \bC
$$
which is generated by the class of the 2-form $\frac{dX}{X} \wedge \frac{dY}{Y}$. This form also represents a non-trivial class in the previous term $H^2(\cX^\vee_q, \cY_{q,x}; \bC)$ in the sequence \eqref{eqn:RelCohomologyLES}.
\end{proof}

The sequence \eqref{eqn:RelCohomologySES} is indeed an extension of VMHS. By standard arguments (e.g. \cite[Section 3.1]{PS08}), we may compare it to the extension \eqref{eqn:HVSES} and obtain the following result, which is the open/closed correspondence on the level of VMHS. 

\begin{theorem}\label{thm:MHSCorr}
Over $(q,x) = \tq \in \tU_\epsilon$, there is an isomorphism of VMHS
$$
  H^3(\cX^\vee_q, \cY_{q,x}; \bC) \cong H^4(\tcX^\vee_{\tq}; \bC) \otimes T(1)
$$
that identifies $[\Omega_q]$ with $[\tOmega_{\tq}]$.
\end{theorem}

Theorem \ref{thm:MHSCorr} is dual to Theorem \ref{thm:CycleCorr} on cycles and periods. It implies that the variations of $[\Omega_q]$ in the relative cohomology $H^3(\cX^\vee_q, \cY_{q,x}; \bC)$ correspond to the variations of $[\tOmega_{\tq}]$ in the usual cohomology $H^4(\tcX^\vee_{\tq}; \bC) \otimes T(1)$, and are thus governed by the extended Picard-Fuchs system $\tcP$. 

\begin{example}\label{ex:C3MHS}
\rm{
Let $\cX = \bC^3$, $\cL$ be an outer brane, and $f=1$, as in Examples \ref{ex:C3}, \ref{ex:C3PicardFuchs}, \ref{ex:C3Cycles}. In this case, the Laurent polynomials are
\begin{align*}
  & H(X,Y) = X + Y + 1,\\
  & \tH(X, Y, Z, \tq_1) = X + Y + 1 + (\tq_1X^{-1}Y^{-1}+1)Z,\\
  & H_0(Y_0,x) = -xY_0^{-1} + Y_0 + 1.
\end{align*}
A direct computation from the definitions shows that
$$
  \cR_H \cong \bC 1, \qquad \cR_{\tH} \cong \bC 1 \oplus \bC X_0Z \oplus \bC X_0^2Z, \qquad \cR_{H_0} \cong \bC 1 \oplus \bC X_0.
$$
The $\cE$-filtrations are specified by the powers of $X_0$, while the $\cI$-filtrations are given by
$$
\begin{matrix}
  \cI_0\cR_H = \cdots = \cI_3\cR_H = 0, \qquad \cI_4\cR_H = \cR_H,\\
  \cI_0\cR_{\tH} = 0, \qquad \cI_1\cR_{\tH} = \cI_2\cR_{\tH} = \cI_3\cR_{\tH} = \bC X_0^2Z, \qquad \cI_4\cR_{\tH} = \bC X_0Z + \bC X_0^2Z, \qquad \cI_5\cR_{\tH} = \cR_{\tH},\\
  \cI_0\cR_{H_0} = 0, \qquad \cI_1\cR_{H_0} = \cI_2\cR_{H_0} = \bC X_0, \qquad \cI_3\cR_{H_0} = \cR_{H_0}.
\end{matrix}
$$
In Proposition \ref{prop:RExtension}, the map $\iota: \cR_{H_0} \to \cR_{\tH}$ (to be constructed in Lemma \ref{lem:MHSInjection} below) maps $1$ to $X_0Z$ and $X_0$ to $X_0^2Z$; the map $\pi: \cR_{\tH} \to \cR_H$ (to be constructed in Lemma \ref{lem:MHSSurjection} below) maps 1 to 1 and $X_0Z$, $X_0^2Z$ to 0. For MHS of Hori-Vafa mirrors, we have
$$
  H^3(\cX^\vee_q; \bC) \cong T(-3), \qquad H^4(\tcX^\vee_{\tq}; \bC) \cong T(-2) \oplus T(-3) \oplus T(-4), \qquad H^2(\cY_{q,x}; \bC) \cong T(-1) \oplus T(-2).
$$
The isomorphism in Theorem \ref{thm:MHSCorr} reads
$$
  H^3(\cX^\vee_q, \cY_{q,x}; \bC) \cong H^4(\tcX^\vee_{\tq}; \bC) \otimes T(1) \cong T(-1) \oplus T(-2) \oplus T(-3).
$$
}
\end{example}

\subsection{Proof of Proposition \ref{prop:RExtension}}\label{sect:RExtensionProof}
We construct the maps $\iota$, $\pi$ in \eqref{eqn:MHSSES} separately as follows.

\begin{lemma}\label{lem:MHSSurjection}
There is a surjective $\bC$-linear map
$$
  \pi: \cR_{\tH} \to \cR_H
$$
such that for any $k$, $l$,
$$
  \pi(\cE^{-k}\cR_{\tH}) \subseteq \cE^{-k}\cR_H, \qquad \pi(\cI_l\cR_{\tH}) \subseteq \cI_{l-1}\cR_H.
$$
\end{lemma}

\begin{proof}
By construction, the last coordinate of any point in $\tDelta$ is non-negative. Therefore, $\cS_{\tDelta}$ is a subring of $\bC[X_0, X^{\pm1}, Y^{\pm1}, Z]$. We then define
$$
  \pi: \cS_{\tDelta} \to \cS_{\Delta},  \qquad Z \mapsto 0.
$$
Note that this definition is valid because $\Delta$ is the facet of $\tDelta$ where the last coordinate is 0. This also implies that for any $k$, $l$,
$$
  \pi(\cE^{-k}\cS_{\tDelta}) \subseteq \cE^{-k}\cS_{\Delta}, \qquad \pi(\cI_l\cS_{\tDelta}) \subseteq \cI_{l-1}\cS_{\Delta}.
$$
Under the restriction $Z=0$, the operator $\tcD_Z$ is trivial, while \eqref{eqn:DefEqnRelation} implies that $\tcD_0$, $\tcD_X$, $\tcD_Y$ restrict to $\cD_0$, $\cD_X$, $\cD_Y$ respectively. Thus $\pi$ descends to the desired map from $\cR_{\tH}$ to $\cR_H$.
\end{proof}

\begin{lemma}\label{lem:MHSInjection}
There is a $\bC$-linear map
$$
  \iota: \cR_{H_0} \to \cR_{\tH}
$$
such that:
\begin{itemize}
  \item for any $k$, we have $\iota(\cE^{-k}\cR_{H_0}) \subseteq \cE^{-k-1}\cR_{\tH}$;

  \item we have $\iota(\cI_2\cR_{H_0}) \subseteq \cI_{1}\cR_{\tH}$, $\iota(\cI_3\cR_{H_0}) \subseteq \cI_4\cR_{\tH}$;

  \item  the image of $\iota$ is equal to the span of all monomials $X_0^kX^aY^bZ^c$ with $c \ge 1$.

\end{itemize}
\end{lemma}

\begin{proof}
Denote $\xi:= e^{\pi\sqrt{-1}/\fa}$. We start by defining a $\bC$-linear map $\iota: \cS_{\Delta_0} \to \cS_{\tDelta}$. First we set
$$
  \iota(1) = X_0Z
$$
which corresponds to the element $(0,0,1)$ (or $\tb_{R+2}$) in $\tDelta$. Now consider $X_0Y_0^{b_0}$ where $b_0 \in \Delta_0 \cap \bZ$. Recall from Section \ref{sect:ExtraCones} that 
$$
  \Delta_0 = [c_S, c_0] = [\fa n_{i_3}(\tsi^S) - \fb m_{i_3}(\tsi^S), \fa n_{i_2}(\tsi^1) - \fb m_{i_2}(\tsi^1)]
$$
is the union of intervals
$$
  [\fa n_{i_3}(\tsi^s) - \fb m_{i_3}(\tsi^s), \fa n_{i_2}(\tsi^s) - \fb m_{i_2}(\tsi^s)], \qquad s = 1, \dots, S,
$$ 
whose length is $|G_{\tsi^s}|$ (see Lemma \ref{lem:ExtraStab}). Take $s$ such that the above interval contains $b_0$, and we may write
$$
  b_0 = \frac{c_0}{|G_{\tsi^s}|} \left(\fa n_{i_3}(\tsi^s) - \fb m_{i_3}(\tsi^s) \right) + \frac{|G_{\tsi^s}| - c_0}{|G_{\tsi^s}|} \left(\fa n_{i_2}(\tsi^s) - \fb m_{i_2}(\tsi^s) \right)
$$
for some $c_0 \in \{0, \dots, |G_{\tsi^s}|\}$. Then we set
$$
  \iota(X_0Y^{b_0}) = (\xi x)^{-a}X_0^2X^aY^bZ 
$$
where
$$
  a = \frac{|G_{\tsi^s}| - c_0}{|G_{\tsi^s}|} m_{i_2}(\tsi^s) + \frac{c_0}{|G_{\tsi^s}|}m_{i_3}(\tsi^s), \qquad b = \frac{|G_{\tsi^s}| - c_0}{|G_{\tsi^s}|} n_{i_2}(\tsi^s) + \frac{c_0}{|G_{\tsi^s}|}n_{i_3}(\tsi^s).
$$
Note that $(a,b,1)$ corresponds to $c_0$ times the generator presented in Lemma \ref{lem:ExtraStab}. When $c_0 = 0$ (resp. $c_0 = |G_{\tsi^s}|$), $\left(\frac{a}{2},\frac{b}{2},\frac{1}{2}\right)$ is the the middle point of the edge between $\tb_{i_2(\tsi^s)}$ (resp. $\tb_{i_3(\tsi^s)}$) and $\tb_{R+2}$; otherwise $\left(\frac{a}{2},\frac{b}{2},\frac{1}{2}\right)$ lies in the interior of the cone $\tsi^s$. In particular, $\left(\frac{a}{2},\frac{b}{2},\frac{1}{2}\right)$ lies in the interior of $\tDelta$ unless $b_0$ is one of the two boundary points $c_S$ or $c_0$ of $\Delta_0$, in which case $\left(\frac{a}{2},\frac{b}{2},\frac{1}{2}\right)$ is the middle point of the edge between $\tb_{R+2}$ and $\tb_{i_2(\tsi^1)}$ or $\tb_{i_3(\tsi^S)}$.

In general, consider $X_0^kY_0^{b_0}$ where $\frac{b_0}{k} \in [\fa n_{i_3}(\tsi^s) - \fb m_{i_3}(\tsi^s), \fa n_{i_2}(\tsi^s) - \fb m_{i_2}(\tsi^s)]$ for some $s \in \{1, \dots, S\}$. We may write
$$
  b_0 = c_3 \left( \fa n_{i_3}(\tsi^s) - \fb m_{i_3}(\tsi^s) \right) + c_2 \left(\fa n_{i_2}(\tsi^s) - \fb m_{i_2}(\tsi^s) \right)
$$
for $c_2, c_3 \in \bQ_{\ge 0}$ with $c_2 + c_3 = k$. Then we set
$$
  \iota(X_0^kY_0^{b_0}) = (\xi x)^{-a}X_0^{k+1}X^aY^bZ
$$
where
$$
  a = c_2 m_{i_2}(\tsi^s) + c_3m_{i_3}(\tsi^s), \qquad b = c_2 n_{i_2}(\tsi^s) + c_3n_{i_3}(\tsi^s).
$$
Note that this recovers the definitions for $k = 0,1$ above. Moreover, we have
\begin{equation}\label{eqn:ABtoB0}
  b_0 = \fa b - \fb a.
\end{equation}
It follows from the construction that $\iota(\cE^{-k}\cS_{\Delta_0}) \subseteq \cE^{-k-1}\cS_{\tDelta}$ for any $k$ and $\iota(\cI_3\cS_{\Delta_0}) \subseteq \cI_4\cS_{\tDelta}$.

Now we verify that $\iota$ descends to a map $\cR_{H_0} \to \cR_{\tH}$. Let $X_0^kY_0^{b_0} \in \cS_{\Delta_0}$ mapping to $(\xi x)^{-a}X_0^{k+1}X^aY^bZ$ as above. We compute that
\begin{equation}\label{eqn:H0Derivative}
  \begin{aligned}
    & \cD_0^0 (X_0^kY_0^{b_0}) = k X_0^kY_0^{b_0} + H_0(Y_0,q,x) X_0^{k+1}Y_0^{b_0} 
    = k X_0^kY_0^{b_0} + \sum_{i=1}^R  s_i(q)(\xi x)^{m_i} X_0^{k+1}Y_0^{\fa n_i - \fb m_i + b_0},\\
    & \cD_{Y_0}^0 (X_0^kY_0^{b_0}) = b_0 X_0^kY_0^{b_0} + \sum_{i=1}^R (\fa n_i - \fb m_i )s_i(q)(\xi x)^{m_i} X_0^{k+1}Y_0^{\fa n_i - \fb m_i + b_0}.
  \end{aligned}
\end{equation}
We compare
$$
  \iota(\cD_0^0 (X_0^kY_0^{b_0}))  = k(\xi x)^{-a} X_0^{k+1}X^aY^bZ + \sum_{i=1}^R  s_i(q)(\xi x)^{m_i} \iota(X_0^{k+1}Y_0^{\fa n_i - \fb m_i + b_0})
$$
to
\begin{align*}
  (\xi x)^{-a}\tcD_0(X_0^{k+1}X^aY^bZ) = &  (k+1)(\xi x)^{-a}X_0^{k+1}X^aY^bZ + (\xi x)^{-a}\tH(X, Y, Z, \tq)X_0^{k+2}X^aY^bZ \\
  = & (k+1)(\xi x)^{-a}X_0^{k+1}X^aY^bZ + (\xi x)^{-a}(x^{\fa}X^{-\fa}Y^{-\fb}+1)X_0^{k+2}X^aY^bZ^2\\
    & +  \sum_{i=1}^R  s_i(q)(\xi x)^{-a}X_0^{k+2}X^{m_i+a}Y^{n_i+b}Z. 
\end{align*}
We show below that the difference
\begin{align*}
  (\xi x)^{-a}&\tcD_0(X_0^{k+1}X^aY^bZ) -  \iota(\cD_0^0 (X_0^kY_0^{b_0})) \\
  = & (\xi x)^{-a} \left( X_0^{k+1}X^aY^bZ + (x^{\fa}X^{-\fa}Y^{-\fb}+1)X_0^{k+2}X^aY^bZ^2 \right) \\
  & +  \sum_{i=1}^R  s_i(q)\left( (\xi x)^{-a}X_0^{k+2}X^{m_i+a}Y^{n_i+b}Z - (\xi x)^{m_i} \iota(X_0^{k+1}Y_0^{\fa n_i - \fb m_i + b_0}) \right)
\end{align*}
is contained in $\tcD_Z\cS_{\tDelta}$, where recall that
$$
  \tcD_Z = Z\frac{\partial}{\partial Z} + X_0Z\frac{\partial \tH(X, Y, Z, \tq)}{\partial Z} = Z\frac{\partial}{\partial Z} + (x^{\fa}X^{-\fa}Y^{-\fb}+1)X_0Z.
$$
First, we have
$$
  \tcD_Z(X_0^{k+1}X^aY^bZ) = X_0^{k+1}X^aY^bZ + (x^{\fa}X^{-\fa}Y^{-\fb}+1)X_0^{k+2}X^aY^bZ^2.
$$
Now for each $i = 1, \dots, R$, we have
$$
  \iota(X_0^{k+1}Y_0^{\fa n_i - \fb m_i + b_0}) = (\xi x)^{a_i'}X_0^{k+2}X^{a_i'}Y^{b_i'}Z  
$$
for some $a_i'$, $b_i'$ satisfying that
$$
  \fa b_i'-\fb a_i' = \fa n_i - \fb m_i + b_0 = \fa(n_i+b) - \fb(m_i+a)
$$
(see \eqref{eqn:ABtoB0}). Now note that for any $k, a', b'$,
$$
  \tcD_Z(X_0^{k+1}X^{a'}Y^{b'}) = (x^{\fa}X^{-\fa}Y^{-\fb}+1)X_0^{k+2}X^{a'}Y^{b'}Z = X_0^{k+2}X^{a'}Y^{b'}Z - (\xi x) X_0^{k+2}X^{a'-\fa}Y^{b'-\fb}Z.
$$
Applying this repeatedly as $(a', b')$ ranges through the integral points on the segment between $(m_i+a, n_i+b)$ and $(a_i', b_i')$, we see that
$$
  (\xi x)^{-a}X_0^{k+2}X^{m_i+a}Y^{n_i+b}Z - (\xi x)^{m_i-a_i'} X_0^{k+2}X^{a_i'}Y^{b_i'}Z \in \tcD_Z\cS_{\tDelta}. 
$$
Therefore, we have verified that 
$$
  (\xi x)^{-a}\tcD_0(X_0^{k+1}X^aY^bZ) - \iota(\cD_0^0 (X_0^kY_0^{b_0})) \in \tcD_Z\cS_{\tDelta}.
$$
A similar computation shows that
$$
  (\xi x)^{-a}(\fa\tcD_Y - \fb\tcD_X)(X_0^{k+1}X^aY^bZ) - \iota(\cD_{Y_0}^0 (X_0^kY_0^{b_0})) \in \tcD_Z\cS_{\tDelta}.
$$
In other words, we have
$$
  \iota(\cD_0^0 (X_0^kY_0^{b_0})), \iota(\cD_{Y_0}^0 (X_0^kY_0^{b_0})) \in \tcD_0\cS_{\tDelta} + \tcD_X\cS_{\tDelta} + \tcD_Y\cS_{\tDelta} + \tcD_Z\cS_{\tDelta}
$$
for any $k, b_0$. This implies that $\iota$ descends to a map $\cR_{H_0} \to \cR_{\tH}$.

Next, we show that $\iota(\cI_2\cR_{H_0}) \subseteq \cI_{1}\cR_{\tH}$. It follows from the construction that for $X_0^kY_0^{b_0}$ where $\frac{b_0}{k} \neq c_S, c_0$, $\iota(X_0^kY_0^{b_0})$ is a monomial corresponding to a point in the interior of $\tDelta$ and thus lies in $\cI_{1}\cR_{\tH}$. In other words, $\iota(\cI_1\cR_{H_0}) \subseteq \cI_{1}\cR_{\tH}$. We show in fact that
$$
  \cI_1\cR_{H_0} = \cI_2\cR_{H_0}.
$$
For $k=1$, the two expressions
$$
  \cD_0^0 (1) = \sum_{i=1}^R  s_i(q)(\xi x)^{m_i} X_0Y_0^{\fa n_i - \fb m_i}, \quad \cD_{Y_0}^0 (1) = \sum_{i=1}^R (\fa n_i - \fb m_i )s_i(q)(\xi x)^{m_i} X_0Y_0^{\fa n_i - \fb m_i}
$$
computed in \eqref{eqn:H0Derivative} are both zero in $\cR_{H_0}$. By the regularity of $H_0$, the two monomials $X_0Y_0^{c_S}$, $X_0Y_0^{c_0}$ are both linear combinations of monomials of form $X_0Y_0^{b_0}$ for $b_0 \neq c_S, c_0$, and are thus contained in $\cI_1\cR_{H_0}$. Now for a general $k \ge 2$, we may proceed by induction and use the relations
\begin{align*}
  & \cD_0^0 (X_0^{k-1}Y_0^{(k-1)c_S}) = (k-1)X_0^{k-1}Y_0^{(k-1)c_S} + \sum_{i=1}^R  s_i(q)(\xi x)^{m_i} X_0^k Y_0^{\fa n_i - \fb m_i + (k-1)c_S},\\
  & \cD_{Y_0}^0 (X_0^{k-1}Y_0^{(k-1)c_S}) = (k-1)c_S X_0^{k-1}Y_0^{(k-1)c_S} + \sum_{i=1}^R (\fa n_i - \fb m_i )s_i(q)(\xi x)^{m_i} X_0^k Y_0^{\fa n_i - \fb m_i + (k-1)c_S}
\end{align*}
to express $X_0^k Y_0^{kc_S}$ in terms of elements in $\cI_1\cR_{H_0}$, and similarly for $X_0^k Y_0^{kc_0}$.

Finally, we characterize the image of $\iota$ by showing that for any $k,a,b,c$ with $c \ge 1$,
$$
  X_0^kX^aY^bZ^c \in \iota(\cS_{\Delta_0}) + \tcD_0\cS_{\tDelta} + \tcD_X\cS_{\tDelta} + \tcD_Y\cS_{\tDelta} + \tcD_Z\cS_{\tDelta}.
$$
We induct on $c$. The base case $c=1$ follows from the argument above, which shows that any $X_0^kX^aY^bZ$ can be modified via elements of the form $\tcD_Z(X_0^{k-1}X^{a'}Y^{b'})$ into a multiple of $\iota(X_0^{k-1}Y^{\fa b-\fb a})$. Now consider a general $c>1$. We have
\begin{align*}
  & \tcD_0(X_0^{k-1}X^aY^bZ^{c-1}) = (k-1)X_0^{k-1}X^aY^bZ^{c-1} + H(X,Y,q)X_0^kX^aY^bZ^{c-1} + (\tq_{R-2}^{\fa}X^{-\fa}Y^{-\fb}+1)X_0^kX^aY^bZ^c,\\
  & \tcD_X(X_0^{k-1}X^aY^bZ^{c-1}) = aX_0^{k-1}X^aY^bZ^{c-1} + X\frac{\partial H(X,Y,q)}{\partial X}X_0^kX^aY^bZ^{c-1} - \fa \tq_{R-2}^{\fa}X^{-\fa}Y^{-\fb} \cdot X_0^kX^aY^bZ^c.
\end{align*}
Note that $X_0^kX^aY^bZ^c$ differs from $\tcD_0(X_0^{k-1}X^aY^bZ^{c-1}) + \tcD_X(X_0^{k-1}X^aY^bZ^{c-1})$ by monomials where $Z$ has power $c-1$. We may then conclude by the inductive hypothesis.
\end{proof}

\begin{proof}[Proof of Proposition \ref{prop:RExtension}]
We use the maps $\pi$, $\iota$ constructed in Lemmas \ref{lem:MHSSurjection}, \ref{lem:MHSInjection} respectively. Lemma \ref{lem:MHSInjection} verifies that the image of $\iota$ is equal to the kernel of $\pi$, i.e. the sequence \eqref{eqn:MHSSES} is exact in the middle. Moreover, \eqref{eqn:PolyVolSum} and \eqref{eqn:RDimension} imply that
$$
  \dim_{\bC} \cR_{\tH} = \dim_{\bC} \cR_H + \dim_{\bC} \cR_{H_0}.
$$
Thus \eqref{eqn:MHSSES} is exact on the left as well.
\end{proof}

%% file: ref.tex

%% file: openclosedBmodel.bbl
\begin{thebibliography}{AA}

\bibitem{AKV02}
M. Aganagic, A. Klemm, C. Vafa, 
``Disk instantons, mirror symmetry and the duality web,'' 
Z. Naturforsch. A \textbf{57} (2002), no. 1-2, 1--28.

\bibitem{AV00} M. Aganagic, C. Vafa, 
``Mirror symmetry, D-branes and counting holomorphic discs,'' 
\texttt{arXiv:hep-th/0012041}.

\bibitem{AL23} K. Aleshkin, C.-C. M. Liu, 
``Open/closed correspondence and extended LG/CY correspondence for quintic threefolds,''
{\tt arXiv:2309.14628}.

\bibitem{AY25} K. Aleshkin, S. Yu,
in preparation.

\bibitem{AHMM09} M. Alim, M. Hecht, P. Mayr, A. Mertens,
``Mirror symmetry for toric branes on compact hypersurfaces,''
JHEP {\bf 09} (2009), 126.

\bibitem{Batyrev93} V. Batyrev,
``Variations of the mixed Hodge structure of affine hypersurfaces in algebraic tori,''
Duke Math. J. {\bf 69} (1993), no. 2, 349--409.

\bibitem{Batyrev94} V. Batyrev,
``Dual polyhedra and mirror symmetry for Calabi-Yau hypersurfaces in toric varieties,''
J. Algebraic Geom. {\bf 3} (1994), no. 3, 493--535.


\bibitem{BCS05} L. Borisov, L. Chen, G.  Smith,
``The orbifold Chow ring of toric Deligne-Mumford stacks,''
J. Amer. Math. Soc. {\bf 18} (2005), no. 1, 193--215.

\bibitem{BH06} L. Borisov, P. Horja, 
``On the K-theory of smooth toric DM stacks,''
Snowbird lectures on string geometry, 21--42, Contemp. Math., {\bf 401}, Amer. Math. Soc., Providence, RI, 2006.


\bibitem{BBvG20} P. Bousseau, A. Brini, M. van Garrel,
``Stable maps to Looijenga pairs,''
Geom. Topol. {\bf 28} (2024), no. 1, 393--496.

\bibitem{BBvG20b} P. Bousseau, A. Brini, M. van Garrel,
``Stable maps to Looijenga pairs: orbifold examples,''
Lett. Math. Phys. {\bf 111} (2021), no. 4, Paper No. 109.


\bibitem{CdGP91}
P Candelas, X. de la Ossa, P. Green, L. Parkes, 
``A pair of Calabi-Yau manifolds as an exactly soluble superconformal theory,'' 
Nuclear Phys. {\bf B 359} (1991), no. 1, 21--74.


\bibitem{CLT13} K. Chan, S.-C. Lau, H.-H. Tseng,
``Enumerative meaning of mirror maps for toric Calabi-Yau manifolds,''
Adv. Math. {\bf 244} (2013), 605--625.




\bibitem{CR04} W. Chen, Y. Ruan,
``A new cohomology theory of orbifold,''
Commun. Math. Phys. {\bf 248} (2004), no. 1, 1--31.

\bibitem{CCK15} D. Cheong, I. Ciocan-Fontanine, B. Kim,
``Orbifold quasimap theory,''
Math. Ann. {\bf 363} (2015), no. 3--4, 777--816.

\bibitem{CKYZ99}
T.-T. Chiang, A. Klemm, S.-T. Yau, E. Zaslow,
``Local mirror symmetry: Calculations and interpretations,''
Adv. Theor. Math. Phys. {\bf 3} (1999), 495--565.

\bibitem{CR10} 
A. Chiodo and Y. Ruan, 
``Landau-Ginzburg/Calabi-Yau correspondence for quintic
three-folds via symplectic transformations,'' 
Invent. Math. {\bf 182} (2010), no. 1, 117--165.



\bibitem{CCIT15} T. Coates, A. Corti, H. Iritani, H.-H. Tseng, 
``A mirror theorem for toric stacks,'' 
Compos. Math. {\bf 151} (2015), no. 10, 1878--1912.

\bibitem{CCIT20} T. Coates, A. Corti, H. Iritani, H.-H. Tseng, 
``Hodge-theoretic mirror symmetry for toric stacks,'' 
J. Differential Geom. {\bf 114} (2020), 41--115.




\bibitem{CK99} D. Cox, S, Katz, 
{\em Mirror symmetry and algebraic geometry},
Math. Surveys and monographs {\bf 68}, American Math. Soc., Providence, RI, 1999.


\bibitem{DK78} V. I. Danilov, A. G. Khovanski\^{i},
''Newton polyhedra and an algorithm for computing Hodge-Deligne numbers,'' 
Math. USSR-Izv. {\bf 29} (1987), 279--298.

\bibitem{Deligne71a} P. Deligne,
``Th\'eorie de Hodge, I,'' 
in {\em Actes du Congr\`es International des Math\'ematiciens (Nice, 1970), Tome 1}, 425--430, Gauthier-Villars, Paris, 1971.

\bibitem{Deligne71b} P. Deligne,
``Th\'eorie de Hodge, II,''
Inst. Hautes \'Etudes Sci. Publ. Math. {\bf 40} (1971), 5--57.

\bibitem{DK11} C. Doran, M. Kerr,
``Algebraic $K$-theory of toric hypersurfaces,''
Commun. Number Theory Phys. {\bf 5} (2011), no. 2, 397--600.


\bibitem{Dwork62} B. Dwork,
``On the zeta function of a hypersurface,''
Inst. Hautes \'Etudes Sci. Publ. Math. {\bf 12} (1962), 5--68.

\bibitem{Dwork64} B. Dwork,
``On the zeta function of a hypersurface: II,''
Ann. of Math. (2) {\bf 80} (1964), no. 2, 227--299.

\bibitem{FL13} B. Fang, C.-C. M. Liu,
``Open Gromov-Witten invariants of toric Calabi-Yau 3-folds,''
Commun. Math. Phys. {\bf 323} (2013), 285--328.

\bibitem{FLT12} B. Fang, C.-C. M. Liu, H.-H. Tseng,
``Open-closed Gromov-Witten invariants of 3-dimensional Calabi-Yau smooth toric DM stacks,''
Forum Math. Sigma {\bf 10} (2022), paper no. e58, 56 pp.

\bibitem{FLZ20} B. Fang, C.-C. M. Liu, Z. Zong,
``On the Remodeling Conjecture for toric Calabi-Yau 3-orbifolds,''
J. Amer. Math. Soc. {\bf 33} (2020), no. 1, 135--222.


\bibitem{FMN10} B. Fantechi, E. Mann, F. Nironi,
``Smooth toric Deligne-Mumford stacks,''
J. Reine Angew. Math. {\bf 648} (2010), 201--244.



\bibitem{vGGR19} M. van Garrel, T. Graber, H. Ruddat, ``Local Gromov-Witten invariants are log invariants,'' Adv. Math. \textbf{350} (2019), 860--876.

\bibitem{GKZ89} I. Gelfand, M. Kapranov, A. Zelevinsky,
``Hypergeometric functions and toric varieties,'' 
Funktsional. Anal. i Prilozhen. {\bf 23} (1989), no. 2, 12--26.


\bibitem{Givental98} A. Givental,
``A mirror theorem for toric complete intersections,'' 
in {\em Topological field theory, primitive forms and related topics (Kyoto, 1996)}, 141--175, Progress in Mathematics {\bf 160}, Birkh\''auser, Boston, 1998.



\bibitem{GZ02} T. Graber, E. Zaslow,
``Open-string Gromov-Witten invariants: Calculations and a mirror ``theorem'','' in {\em Orbifolds in mathematics and physics (Madison, WI, 2001)}, 107--121, Contemp. Math. {\bf 310}, Amer. Math. Soc., Providence, RI, 2002.

\bibitem{GRZ22} T. Gr\"{a}fnitz, H. Ruddat, E. Zaslow,
``The proper Landau-Ginzburg potential is the open mirror map,''
Adv. Math. \textbf{447} (2024), Paper No. 109639, 69 pp.

\bibitem{HV00} K. Hori, C. Vafa, 
``Mirror symmetry,'' 
\texttt{arXiv:hep-th/0002222}.

\bibitem{HIV00}
K. Hori, A. Iqbal, C. Vafa, 
``D-branes and mirror symmetry,'' 
\texttt{arXiv:hep-th/0005247}.

\bibitem{Hosono06} S. Hosono,
``Central charges, symplectic forms, and hypergeometric series in local mirror symmetry,'' 
in {\em Mirror symmetry, V}, 405--439, AMS/IP Stud. Adv. Math. {\bf 38}, Amer. Math. Soc., Providence, RI, 2006.

\bibitem{Iritani09} H. Iritani, 
``An integral structure in quantum cohomology and mirror symmetry for toric orbifolds,'' 
Adv. Math. \textbf{222} (2009), no. 3, 1016--1079.

\bibitem{Jiang08} Y. Jiang, 
``The orbifold cohomology ring of simplicial toric stack bundles,'' Illinois J. Math. \textbf{52} (2008), no. 2, 493--514.

\bibitem{JT08} Y. Jiang. H-H. Tseng,
``Note on orbifold Chow ring of semi-projective toric Deligne-Mumford stacks,'' 
Comm. Anal. Geom. {\bf 16} (2008), no. 1, 231--250.

\bibitem{JS09}
H. Jockers, M. Soroush, 
``Effective superpotentials for compact D5-brane Calabi-Yau geometries,''
Commun. Math. Phys. {\bf 290} (2009), no. 1, 249--290.



\bibitem{KKP08}
L. Katzarkov, M. Kontsevich, T. Pantev, 
``Hodge theoretic aspects of mirror symmetry,'' 
in {\em From Hodge theory to integrability and TQFT tt*-geometry}, 87--174, Proc. Sympos. Pure Math. {\bf 78}, Amer. Math. Soc., Providence, RI, 2008.

\bibitem{KZ15}
H. Ke, J. Zhou, 
``Gauged linear sigma model for disc invariants,'' 
Lett. Math. Phys. \textbf{105} (2015), 63--88.

\bibitem{KM10} Y. Konishi, S. Minabe, 
``Local B-model and mixed Hodge structure,'' 
Adv. Theor. Math. Phys. \textbf{14} (2010), no. 4, 1089--1145.

\bibitem{Kouch76} A. Kouchnirenko, 
``Poly\`edres de Newton et nombres de Milnor,'' 
Invent. Math. \textbf{32} (1976), no. 1, 1--31.
  
\bibitem{LM01} W. Lerche, P. Mayr, 
``On $\mathcal{N} = 1$ mirror symmetry for open type II strings,'' \texttt{arXiv:hep-th/0111113}.

\bibitem{LMW02a} W. Lerche, P. Mayr, N. Warner, 
``Holomorphic $N = 1$ special geometry of open-closed type II strings,'' \texttt{arXiv:hep-th/0207259}.

\bibitem{LMW02b} W. Lerche, P. Mayr, N. Warner, 
``$N = 1$ special geometry, mixed Hodge variations and toric geometry,'' \texttt{arXiv:hep-th/0208039}.




\bibitem{LLY12} S. Li, B. Lian, S.-T. Yau, 
``Picard-Fuchs equations for relative periods and Abel-Jacobi map for Calabi-Yau hypersurfaces,'' 
Amer. J. Math. \textbf{134} (2012), no. 5, 1345--1384.

\bibitem{LLY97}
B. Lian, K. Liu, and S.-T. Yau, 
``Mirror principle. I,''
Asian J. Math. {\bf 1} (1997), no. 4, 729--763,




\bibitem{LY21} C.-C. M. Liu, S. Yu, 
``Open/closed correspondence via relative/local correspondence,''
Adv. Math. \textbf{410} (2022), paper no. 108696, 43 pp.

\bibitem{LY22} C.-C. M. Liu, S. Yu, 
``Orbifold open/closed correspondence and mirror symmetry,''
{\tt arXiv:2210.11721}.

\bibitem{Mayr01} P. Mayr, ``$N = 1$ mirror symmetry and open/closed string duality,'' \texttt{arXiv:hep-th/0108229}.

\bibitem{MW09}
D. Morrison, J. Walcher, 
``D-branes and normal functions,''
Adv. Theor. Math. Phys. {\bf 13} (2009), no. 2, 553--598.



\bibitem{PSW08} 
R. Pandharipande, J. Solomon, J. Walcher,
``Disk enumeration on the quintic 3-fold,''
J. Amer. Math. Soc. {\bf 21} (2008), 1169--1209.


\bibitem{PS08}
C. Peters, J. Steenbrink, 
{\em Mixed Hodge structures},
Ergebnisse der Mathematik und ihrer Grenzgebiete. 3. Folge/A Series of Modern Surveys in Mathematics {\bf 53}, Springer-Verlag, Berlin, Heidelberg, 2008.


\bibitem{Stienstra98} J. Stienstra,
``Resonant hypergeometric systems and mirror symmetry,'' 
in {\em Integrable systems and algebraic geometry (Kobe/Kyoto, 1997)}, 412--452, World Sci. Publ., River Edge, NJ, 1998.

\bibitem{Walcher07}
J. Walcher, 
``Opening mirror symmetry on the quintic,'' 
Commun. Math. Phys. {\bf 276} (2007), no. 3, 671--689.

\bibitem{Walcher08}
J. Walcher, 
``Open strings and extended mirror symmetry,'' 
in {\em Modular forms and string duality}, 279--297, 
Fields Inst. Commun. {\bf 54}, Amer. Math. Soc., Providence, RI, 2008.

\bibitem{Yu20} S. Yu,
``The Open Crepant Transformation Conjecture for toric Calabi-Yau 3-orbifolds,''
J. Differential Geom. {\bf 130} (2025), no. 1, 27--70.

\bibitem{Yu23} S. Yu, 
{\em Open/closed correspondence and mirror symmetry},
Ph.D. thesis, Columbia University, 2023.

\bibitem{Yu24} S. Yu, 
``Open/closed BPS correspondence and integrality,"
Commun. Math. Phys. {\bf 405}, 219 (2024), 34 pp.


\bibitem{YZ23} S. Yu, Z. Zong,
``Open WDVV equations and Frobenius structures for toric Calabi-Yau 3-folds,''
Forum Math. Sigma {\bf 13} (2025), Paper No. e76, 29 pp.



\end{thebibliography}
